\documentclass[11pt]{article}
\usepackage{amsmath,amsthm,amssymb,amsfonts}
\usepackage[latin1]{inputenc}
\usepackage{enumerate,enumitem}
\usepackage{graphics,graphicx,epsfig,psfrag}
\usepackage{hyperref}

\evensidemargin = \oddsidemargin
\theoremstyle{plain}
\newtheorem{thmin}{Theorem}

\newtheorem*{corollaryC}{Corollary C$'$}
\newtheorem*{corollaryD}{Corollary D$'$}

\newtheorem{theorem}{Theorem}[section]
\newtheorem{proposition}[theorem]{Proposition}

\newtheorem*{gronwall}{Gronwall's Lemma}
\newtheorem{lemma}[theorem]{Lemma}
\newtheorem{keylemma}[theorem]{Key Lemma}
\theoremstyle{definition}

\newtheorem{definition}[theorem]{Definition}

\newtheorem{affirmation}[theorem]{Claim}

\theoremstyle{remark}

\newtheorem*{remarksn}{Remark}
\newtheorem{remark}[theorem]{Remark}

\def\id{{\mathrm{id}}}
 
\def\Int{\mathop{\mathrm{Int}}}

\newcommand{\wt}{\widetilde}

\def\from{\colon} 

\newcommand{\ma}{\measuredangle}
\def\including{containing }
\def\includes{contains }
\def\Bb{\overline{B}}
\def\Bh{\widehat{B}}

\def\Op{\mathop{\mathrm{Op}}\nolimits}

\def\N{{\mathbb N}}

\def\R{{\mathbb R}}
\def\Z{{\mathbb Z}}

\renewcommand{\t}{\theta}
\newcommand{\f}{\varphi}
\newcommand{\eps}{\varepsilon}
\def\taut{{\tilde{\tau}}}

\def\P{{\mathbb P}}
\def\D{{\mathbb D}}

\def\T{{\mathbb T}}

\def\Ss{{\mathbb S}}
\def\cercle{{\mathbb S}}
\def\sphere{{\mathbb S}}

\def\CC{{ C}}

\def\Cinf{{\CC^{\infty}}}

\def\DDi{ {\widetilde{\Diff_+(\cercle^1)}} }

\def\Diff{{\mathrm{Diff}}}
\def\ft{{\tilde{f}}}
\def\gt{{\tilde{g}}}
\def\fbar{{\bar{f}}}
\def\gbar{{\bar{g}}}
\def\hti{{\tilde{h}}}
\def\hbar{{\bar{h}}}
\def\nut{{\tilde{\nu}}}

\newcommand{\FT}{ {\mathcal{F}_\pitchfork } }

\newcommand{\FTV}{ {\mathcal{F}_{\pitchfork,V} }}

\newcommand{\FF}{ {\mathcal{F} } }
\newcommand{\PP}{ {\mathcal{P} } }
\newcommand{\PPt}{ {\PP_\pitchfork } }
\newcommand{\NN}{ {\mathcal{N} } }

\newcommand{\Mal}{ {\mathcal{M}} }

\newcommand{\GG}{ {\mathcal{G} } }

\newcommand{\LL}{{\f }}
\newcommand{\Lid}{{\f_{\id}}}
\newcommand{\Lf}{{\f_f}}

\newcommand{\xibar}{ {\bar\xi} }
\newcommand{\xit}{ {\tilde\xi} }

\newcommand{\ITI}{\mathrm{iti}}

\newcommand{\dist}{\mathop{\mathrm{dist}}\nolimits}

\newcommand{\fix}{\mathop{\mathrm{Fix}}\nolimits}

\def\ly{\fontencoding{U}\fontfamily{lasy}\fontseries{m}\fontshape{n}\selectfont}
\def\guil#1{\leavevmode\hbox{{\ly(\kern-0.20em(\kern+0.20em}}\nobreak{}\,#1\,%
  \nobreak\hbox{{\ly\kern+0.20em)\kern-0.20em)}}}

\def\Bars#1{\lVert #1 \rVert}

\def\lrBars#1{\left\lVert #1 \right\rVert} 
 
\def\norm#1{\left\lVert #1 \right\rVert}
\def\res#1{\mathbin{|}{}_{#1}}

\begin{document}
\title{On the connectedness of the space of codimension one foliations on a closed $3$-manifold}
\author{
\begin{tabular}{c}
{H\'el\`ene \textsc{Eynard-Bontemps}}\\
{\small{IMJ - PRG, UPMC}\footnote{Institut de Math\'ematiques de Jussieu - Paris Rive Gauche, Universit\'e Pierre et Marie Curie, 4 place Jussieu, 75252 Paris cedex 5, France. Tel: 00 33 1 44 27 79 62. Fax: 00 33 1 44 27 63 66.}}\\
{\small{helene.eynard-bontemps@imj-prg.fr}}
\end{tabular}}

\date{}
\maketitle

\begin{abstract}
We study the topology of the space of smooth codimension one foliations on a closed $3$-manifold. We regard this space as the space of integrable plane fields included in the space of all smooth plane fields. It has been known since the late 60's that every plane field can be deformed continuously to an integrable one, so the above inclusion induces a surjective map between connected components. We prove that this map is actually a bijection.
\end{abstract}\bigskip

\medskip


\newpage

\tableofcontents
\newpage

In this article, we are interested in the topology of the space $\FF(M)$ of ($C^\infty$) smooth codimension one foliations on a closed $3$-manifold $M$. We identify such a foliation with its tangent plane field, and hence regard $\FF(M)$ as the subset of \emph{integrable plane fields} inside the space $\PP(M)$ of all plane fields on $M$, endowed with the usual $C^\infty$ topology.

Most plane fields are not integrable: non integrable plane fields form a dense open subset of $\PP(M)$.  It has been known since the late 60's, however,  that any closed 3-manifold admits a smooth codimension one foliation \cite{Li,No}. Moreover, according to works of J. Wood \cite{Wo} and W.P. Thurston \cite{Th2}, any smooth plane field can be deformed to a smooth foliation. In other words, the map $\pi_0\FF(M)\overset{\iota_*}{\to} \pi_0\PP(M)$ induced by the inclusion $\FF(M)\overset{\iota}{\hookrightarrow}\PP(M)$ is \emph{surjective}. It is then tempting to ask whether this inclusion is actually a weak homotopy equivalence, or in Gromov's language whether \emph{foliations satisfy the parametric $h$-principle}. In fact, such an $h$-principle was established by Y.~Eliashberg \cite{El} for a related class of (locally homogeneous) plane fields, namely overtwisted contact structures. We study here the validity of the \emph{one-parameter} $h$-principle for foliations, and obtain the following:

\begin{thmin}\label{t:main} Let $M$ be a closed oriented $3$-manifold, $\PP(M)$ the space of smooth transversely oriented plane fields on $M$ and $\FF(M)$ the space of smooth codimension one foliations on $M$. The inclusion of $\FF(M)$ in $\PP(M)$ induces a bijection between connected components.
\end{thmin}

Theorem \ref{t:main} improves the main result of the author's PhD dissertation \cite{Ey}, which states that any two $C^\infty$ foliations homotopic as plane fields can be connected by a continuous path of $C^1$ foliations. What was missing in \cite{Ey} in order to remain in the $C^\infty$ class was the path-connectedness of the space of smooth $\Z^2$-actions on the segment (cf. Section \ref{s:cleaning}). C.~Bonatti and the author have since proved the \emph{connectedness} of this space \cite{B-E}, which, combined to \cite{Ey}, gives Theorem \ref{t:main}. Path-connectedness, however, remains out of reach, for actions as well as foliations: we do not know whether the map $\pi_0\FF(M)\overset{\iota_*}{\to} \pi_0\PP(M)$ is \emph{injective}. In $\PP(M)$, which is locally contractible, connected and path-connected components are the same, but this is not clear in $\FF(M)$, which is a closed subset with empty interior. 

Surjectivity between higher homotopy groups, on the other hand, is easier to obtain. The following result can be derived from our techniques. We give a complete proof for $k=1$ (cf. Theorem \ref{t:malleable}) and a sketch for the general case (cf. end of Section \ref{s:malleable}).

\begin{thmin}\label{t:pik} Let $M$ be a closed oriented $3$-manifold, $\PP(M)$ the space of smooth transversely oriented plane fields on $M$ and $\FF(M)$ the space of smooth codimension one foliations on $M$. For any $k\ge1$, the map $\pi_k\FF(M)\overset{\iota_*}{\to} \pi_k\PP(M)$ induced by the inclusion is \emph{surjective}.
\end{thmin}

To present the strategy of the proof of Theorem \ref{t:main}, we will first explain how to deform a \emph{single} plane field $\xi$ to a foliation. The argument we describe is due to Thurston \cite{Th2}, who later generalized it to higher dimensions and codimensions \cite{Th3,Th4}. 

\subsection*{Thurston's method}\label{ss:thurston}

Thurston's construction proceeds in three steps. \medskip

\emph{Step 1}. First, we make $\xi$ integrable outside a finite collection of balls (thought of as ``holes" in the resulting foliation) on which it is \emph{almost horizontal} (cf. Definition \ref{d:p-i}), meaning, basically, that it is tangent to the boundary sphere at exactly two points, the \emph{poles}, and transverse to a vector field on the whole ball  (tangent to the boundary) connecting the poles. In this article, such a plane field will be called \emph{almost integrable} (cf. Definition \ref{d:p-i}). To do so, the idea is to construct a triangulation ``in good position'' with respect to $\xi$, and to make $\xi$ integrable in a neighbourhood $V$ of its $2$-skeleton. More precisely, we require all faces and edges to be transverse to $\xi$, and the direction of $\xi$ to be almost constant on each $3$-simplex. We then make $\xi$ integrable in a neighbourhood of every vertex, then every edge, and finally every face. The key point is that, in a neighbourhood of every simplex $\sigma$ of the $2$-skeleton, there exists a nonsingular vector field $\nu$ tangent to $\xi$ and transverse to $\sigma$. The deformation consists in making $\xi$ invariant under $\nu$ in a neighbourhood of $\sigma$. Since $\xi$ is already integrable near $\partial \sigma$, it is already invariant under $\nu$ there and thus remains unchanged. This guarantees the global coherence of these local perturbations. The neighbourhood $V$ of the $2$-skeleton can be chosen so that the complement of $V$ is a collection of balls (one in each $3$-simplex) on which $\xi$ is almost horizontal.

\begin{figure}[htbp]
\centering
\begin{tabular}{ccccc}
\put(-7,0){$\scriptstyle \xi$}
\includegraphics[height=3cm]{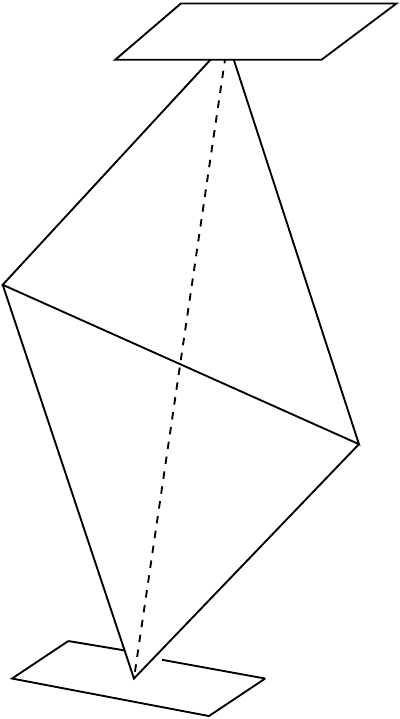} 
& \includegraphics[height=3cm]{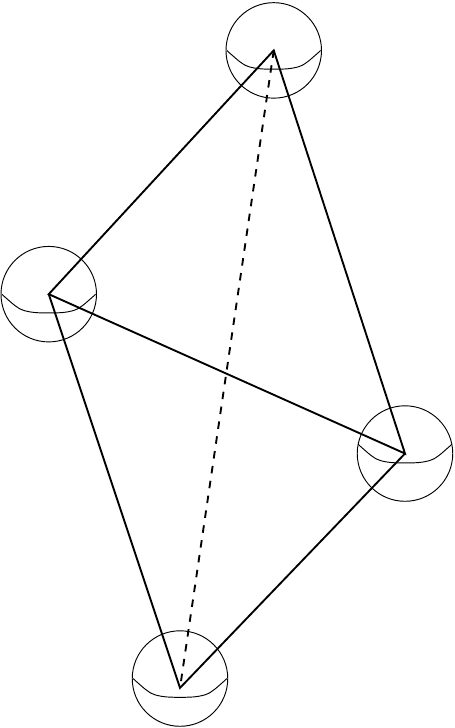} 
& \includegraphics[height=3cm]{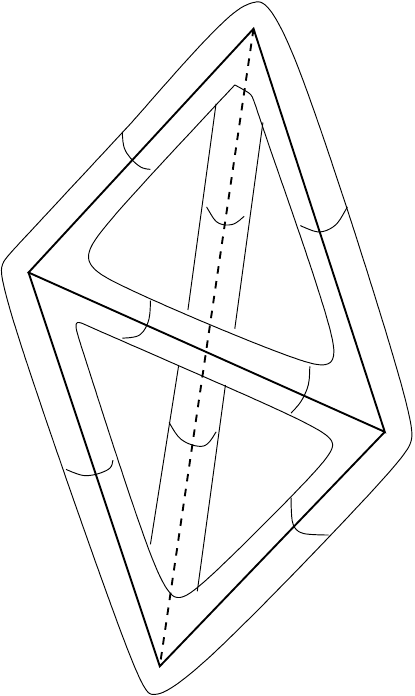}
& \includegraphics[height=3cm]{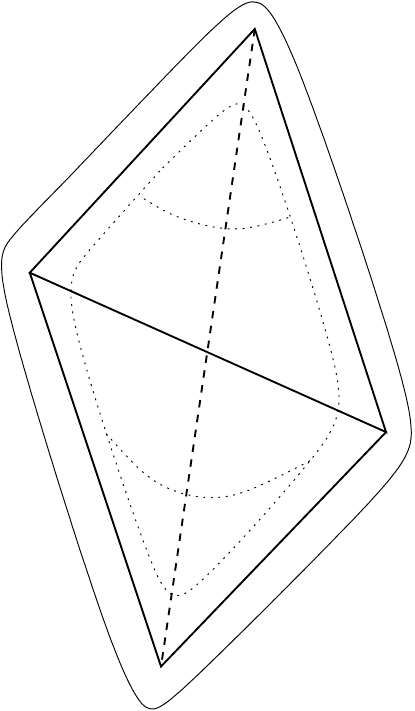}
& \includegraphics[height=3cm]{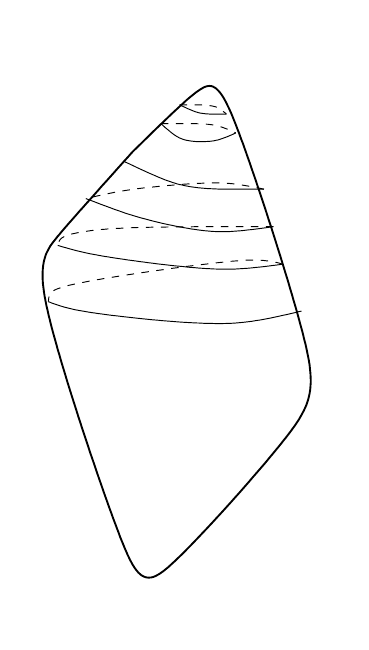}
\end{tabular}
\caption{Making $\xi$ integrable near the $2$-skeleton \label{f:2skeleton}}
\end{figure}

\emph{Step 2}. Due to the Reeb Stability Theorem \cite{Re}, the restriction of $\xi$ to such a ball $B$ cannot be deformed (rel. $\partial B$) to a foliation, unless $\xi\res {\partial B}$ is a foliation by circles outside the poles. To get around this problem, the idea is to replace ball-shaped holes by toric ones by digging tunnels along transverse arcs in $V$ connecting the poles from outside. A sufficient condition for such arcs to exist is that the foliation on $V$ is \emph{taut}, \emph{i.e} that every transverse arc (in particular ``meridians" connecting the poles of a ball) extends to a closed transversal, or equivalently that every leaf is crossed by a closed transversal. In that case, we will say that the almost integrable plane field $\xi$ itself is \emph{taut}. Thurston artfully reduces to this situation by ``ripping" all leaves, making them spiral around new ball-shaped holes where he temporarily sacrifices the integrability (cf. Lemma \ref{l:desintegration} for a parametric version of this trick). 
\begin{figure}[htbp]
\centering
\begin{tabular}{cc}
\includegraphics[height=1.7cm]{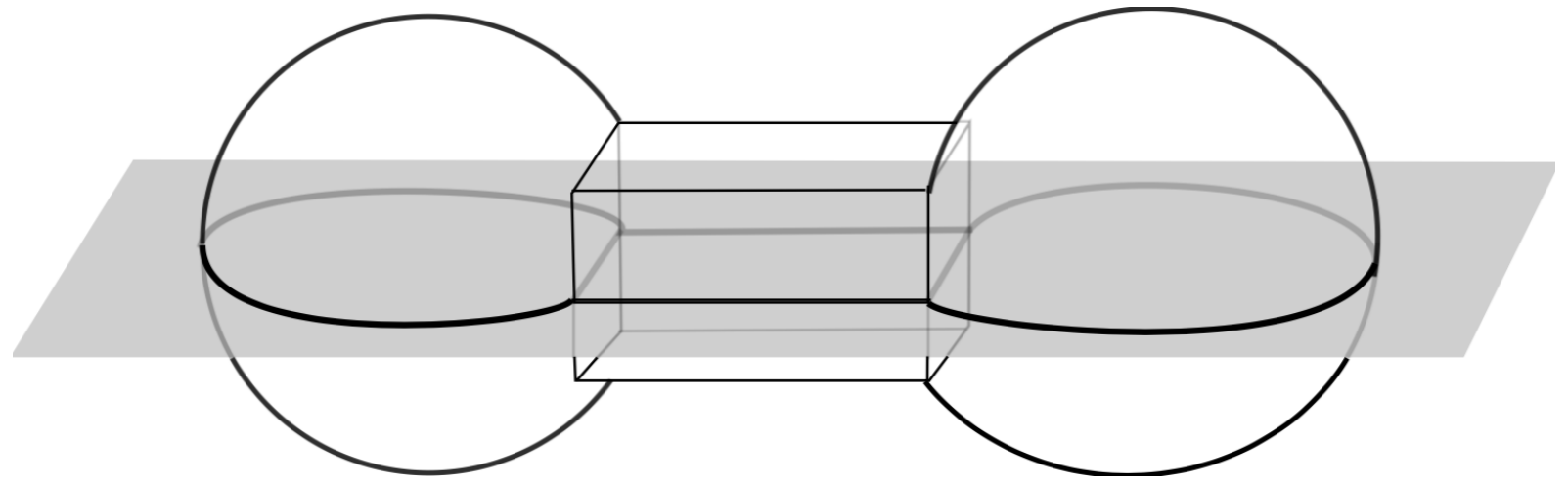} 
& \includegraphics[height=1.7cm]{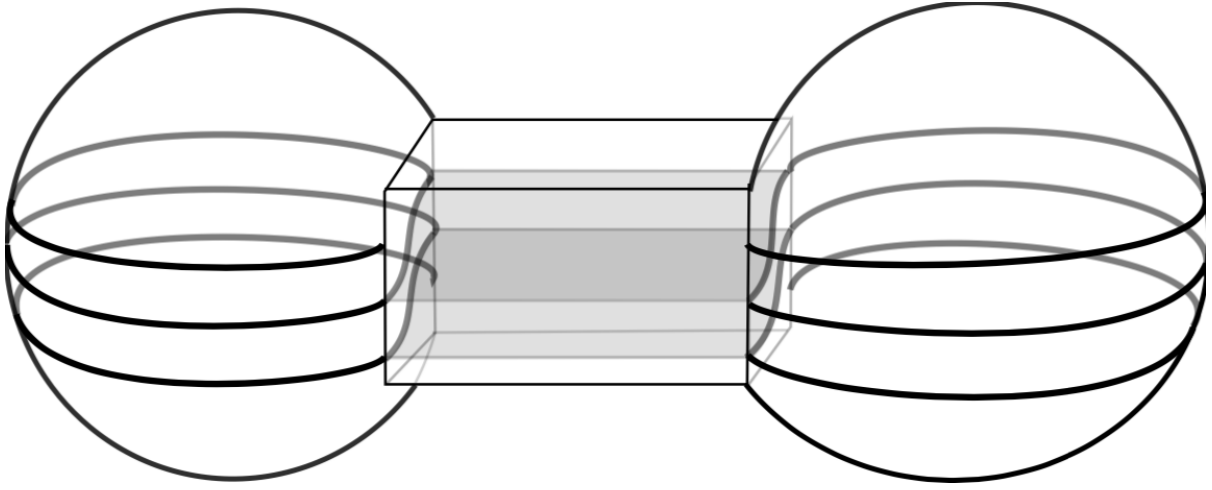} 
\end{tabular}
\caption{Thurston's trick \label{f:whirl}}
\end{figure}

\begin{figure}[htbp]
\centering
 \begin{tabular}{cc}
\includegraphics[height=3cm]{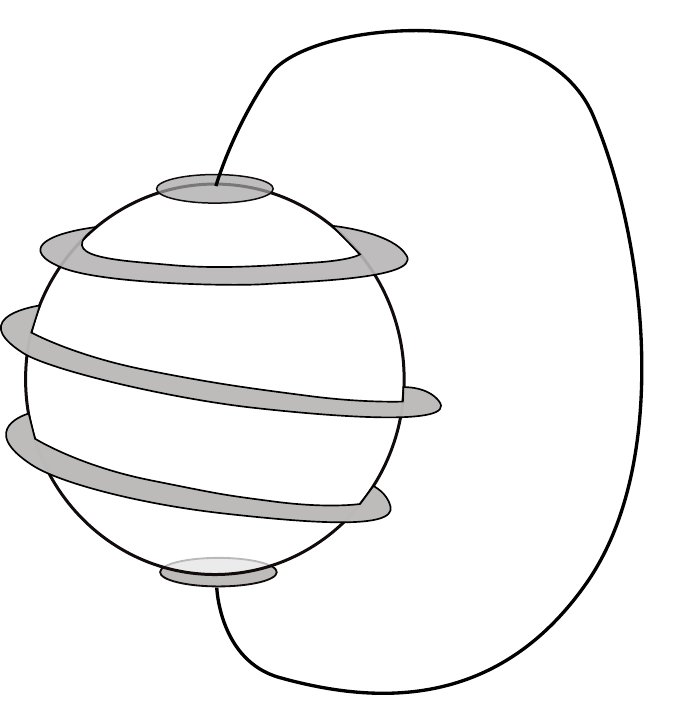} 
  & \includegraphics[height=3cm]{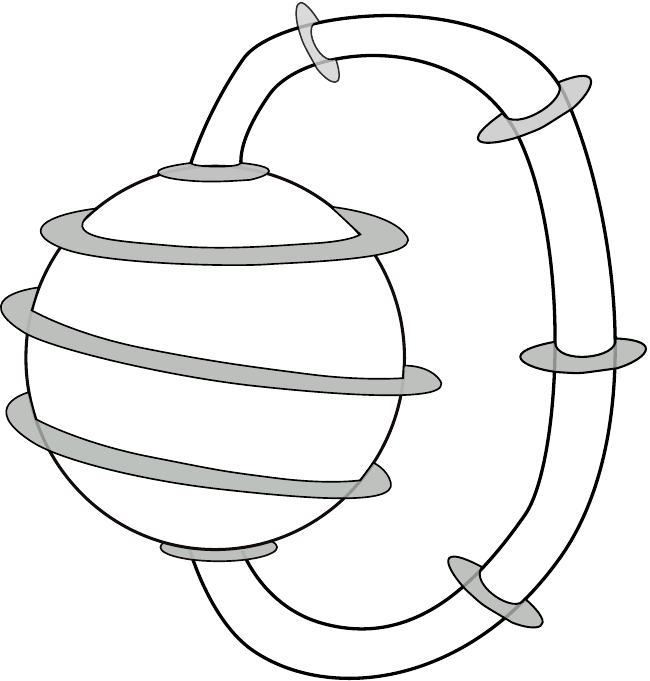} 
\end{tabular}
\caption{Enlarging the holes \label{f:tunnel}}
\end{figure}

\emph{Step 3}. We can now enlarge each hole by digging a tunnel (so far foliated by disks) along a transverse arc connecting its poles. We thus have a collection of solid toric holes outside of which the plane field, still denoted by $\xi$, is integrable and on which $\xi$ is  transverse to the $\sphere^1$ factor of $\D^2\times \Ss^1$. Thurston then shows that such a plane field on a solid torus can always be deformed to a foliation, relative to the boundary. This uses the simplicity of the group of smooth orientation preserving diffeomorphisms of the circle, due to M.~Herman \cite{He1}. P.~Schweitzer later gave a more geometric filling argument in \cite{Sc}, also based on a theorem of Herman \cite{He}, whose advantage, as A.~Larcanch\'e observed in \cite{La}, is that the resulting foliations depend continuously on their trace on the boundary. These foliations of the solid torus, which will be referred to as \emph{Schweitzer foliations}, will be described more precisely in Section \ref{s:prelim}. Let us just say, for now, that they are transverse to the $\sphere^1$ factor of $\D^2\times \sphere^1$  \emph{except} above two circles in $\D^2$ whose products by $\sphere^1$ are torus leaves bounding Reeb components \cite{Re}. In particular, all leaves except these two are crossed by a closed transversal. \medskip

By construction, the foliations obtained by Thurston's process (combined with Schweit\-zer's filling method) are \emph{malleable} in the following sense: a foliation is \emph{malleable} (cf.  Definition \ref{d:malleable}) if it is taut outside a finite collection of solid tori and induces, on each of these, a Schweitzer foliation whose trace on the boundary torus has a  whole one-parameter family of (meridian) circle leaves. We will denote by $\Mal(M)$ the set of malleable foliations on $M$. Thus what Thurston's construction shows is that the map $\pi_0\Mal(M)\to \pi_0\PP(M)$ induced by the inclusion $\Mal(M)\hookrightarrow\PP(M)$ is \emph{onto}.\medskip

\subsection*{Outline of the proof of Theorem \ref{t:main}}

The general idea to prove Theorem \ref{t:main} is to give a relative one-parameter version of Thurston's construction: start with a continuous family $\xi_t$, $t\in [0,1]$, in $\PP(M)$ such that $\xi_0$ and $\xi_1$ are integrable, and deform it (with fixed endpoints $\xi_0$ and $\xi_1$) to a family of integrable plane fields. This raises two major difficulties. First, there are a lot of choices involved in Thurston's process (triangulation, transverse arcs...), and it is not at all clear (and is actually wrong) that such choices can be made continuously with respect to the parameter $t$.  But a perhaps bigger issue is the relative part of the problem: Thurston's process does not leave integrable plane fields unchanged! It deforms them (a great deal) to \emph{malleable} foliations. So let us first restrict to the case when $\xi_0$ and $\xi_1$ are malleable, and then explain how to reduce to this case. 

\paragraph{``Malleable case" (Section \ref{s:malleable}).} As we just saw, the best we can expect from a (naive) parametric version of Thurston's construction  is the following: 

\begin{thmin}[Malleable case]\label{t:malleable} Any continuous path of plane fields connecting two malleable foliations can be deformed (with fixed endpoints) to a path of malleable foliations. In particular, the map $\pi_0\Mal(M)\overset{\iota_*}{\to} \pi_0\PP(M)$ induced by the inclusion $\Mal(M)\overset{\iota}{\hookrightarrow}\PP(M)$ is \emph{injective}. 
\end{thmin}

\begin{remarksn}\label{r:pi1}
This also shows that the map $\pi_1\Mal(M)\to \pi_1\PP(M)$ (for any choice of base point) is \emph{onto} (see Theorem \ref{t:pik}).
\end{remarksn}

Before sketching the proof of Theorem \ref{t:malleable}, note that taut foliations (when they exist) are in particular malleable (cf. Definition \ref{d:malleable}), so the following statement is a direct corollary of the above:

\begin{corollaryC}\label{c:tendu} 
Two taut foliations homotopic as plane fields
are connected by a path of (malleable) foliations.
\end{corollaryC}
 
\begin{remarksn}\label{l:larcanche}\begin{itemize}
\item This extends a result by Larcanch\'e \cite{La}, who proved the above statement in the case of two sufficiently close taut foliations, and of two foliations transverse to the fibers of a circle bundle over a closed surface (and thus taut). Schweitzer's construction plays a key role in her proof as well as ours.  
\item The foliations of the paths we construct, including Larcanch\'e, are malleable, but not taut in general. As a matter of fact, J. Bowden \cite{Bo} and T. Vogel \cite{Vo} recently gave examples of taut foliations homotopic as foliations but which cannot be connected by a path of \emph{taut} foliations. 
\end{itemize}
\end{remarksn}

Now the proof of Theorem \ref{t:malleable} goes as follows. Consider two malleable foliations, $\tau_0$ and $\tau_1$, homotopic as plane fields. Here, we only explain how to connect them by a path of malleable foliations. Think of $\tau_0$ and $\tau_1$ as obtained from two taut almost integrable plane fields $\xi_0$ and $\xi_1$ by Thurston's construction, using Schweitzer's filling method (see ``step 3'' above, and Lemmas \ref{l:correspondance} and \ref{l:inverse} for further detail) and take a path $\xi_t$, $t\in[0,1]$, of plane fields connecting $\xi_0$ and $\xi_1$. This is the path to which we are going to apply a parametric version of Thurston's process. \medskip

\emph{Step 1}. We first deform the whole family $\xi_t$, $t\in[0,1]$, to a family of almost integrable plane fields (all having the same ``holes", the poles varying continuously with respect to the parameter), keeping $\xi_0$ and $\xi_1$ unchanged  (cf. Proposition \ref{p:p-i}). To do this, we pick a triangulation of $M$ such that every $\xi_t$ is almost constant on each $3$-simplex. Unfortunately, since the direction of $\xi_t$ varies with $t$, one cannot require the edges and faces to be transverse to \emph{every} $\xi_t$. And if some $\xi_t$ is tangent to some face $\sigma$ at some point, one cannot find the desired nonsingular vector field $\nu_t$ near $\sigma$ both tangent to $\xi_t$ and transverse to $\sigma$. Fortunately, this problem has already been considered and solved by Eliashberg in the closely related field of contact structures \cite{El}. The main idea is to consider these special $2$-simplices $\sigma$ as ``big vertices" and treat them before any other simplex of the $2$-skeleton. The adaptation of Eliashberg's techniques to foliations is carried out in the appendix. \medskip

\emph{Step 2}. A parametric version of Thurston's second step (cf. Lemma \ref{l:desintegration}) allows us to perturb the new family to a family of \emph{taut} almost integrable plane fields (again keeping $\xi_0$ and $\xi_1$ unchanged).\medskip

\emph{Step 3}. Each of these new plane fields can be made integrable following Thurston's third step, using transverse arcs connecting the poles of the balls for each value of the parameter (whose existence is guaranteed by step 2). This, in particular, turns $\xi_0$ and $\xi_1$ back into $\tau_0$ and $\tau_1$. But we want the resulting foliations to depend continuously on the parameter. This can be achieved if we find transverse arcs which vary continuously with respect to the parameter: then the toric holes to be filled also vary continuously, as well as the holonomy of the foliations induced on their boundaries, and Schweitzer's construction can be performed continuously too (cf. Theorem \ref{t:larcanche}). In the general case however, the transverse arcs vary only \emph{piecewise} continuously, and we only get a piecewise continuous path of foliations. To fill the gaps, we basically need to check that the deformation class (among foliations) of a malleable foliation obtained from a taut almost integrable plane field does not depend on the choice of transverse arcs. This is the content of the Siphon Lemma \ref{l:v-c}, which is the final ingredient of Proposition \ref{p:third} (the relative one-parameter version of Thurston's third step) and thus concludes the proof of Theorem \ref{t:malleable}.


\paragraph{Reduction to the malleable case (Sections \ref{s:malleabilization} and \ref{s:cleaning}).} The injectivity of the map $\pi_0\FF(M)\overset{\iota_*}{\to} \pi_0\PP(M)$ would follow from Theorem \ref{t:malleable} if we could show that any foliation can be deformed to a malleable one through foliations. Recall that a malleable foliation is basically one whose leaves all intersect closed transversals, except possibly finitely many torus leaves bounding Reeb components. What is true for \emph{any} foliation now, according to classical results by S.~P.~Novikov \cite{No} and S.~Goodman \cite{Go}, is that only torus leaves can fail to intersect closed transversals. But such ``problematic" torus leaves, which will henceforth be referred to as \emph{Novikov leaves},  do not necessarily bound Reeb components. We want to get rid of them by deforming the foliation, or rather to replace them by nice ones lying in Schweitzer foliations. In Section \ref{s:malleabilization}, we first restrict to foliations which are described by a simple local model near their Novikov tori (cf. Definition \ref{d:net}). We call such foliations ``\emph{neat}". Local perturbations using the tools of Section \ref{s:prelim} allow us to prove: 

\begin{thmin}[Malleabilization]\label{t:malleabilization} Every neat foliation can be deformed to a malleable one among neat foliations.
\end{thmin}

This, together with Theorem \ref{t:malleable}, implies the following, where $\NN(M)$ denotes the space of neat foliations on a manifold $M$: 

\begin{corollaryD}[Neat case]\label{c:clean} Any continuous path of plane fields connecting two neat foliations can be deformed (with fixed endpoints) to a path of neat foliations. In particular, the map $\pi_0\NN(M)\overset{\iota_*}{\to} \pi_0\PP(M)$ induced by the inclusion $\NN(M)\overset{\iota}{\hookrightarrow}\PP(M)$ is \emph{injective}. 
\end{corollaryD}

Finally, in Section \ref{s:cleaning}, we prove:

\begin{thmin}\label{t:cleaning} Neat foliations are dense among foliations.
\end{thmin}

This is precisely where we loose \emph{path}-connectedness and achieve only connectedness: instead of a \emph{continuous deformation} of any foliation to a neat one, what we achieve is a(n arbitrarily) small \emph{perturbation}. \medskip

Theorem \ref{t:main} follows readily: given a connected component $\CC$ in $\PP(M)$, $\FF(M)\cap \CC$ is connected since $\NN(M)\cap \CC$ is both path-connected by Corollary D' and dense in $\FF(M)\cap \CC$ by Theorem \ref{t:cleaning}.
\bigskip

\noindent\textbf{Acknowledgements}. I am grateful \emph{beyond words} to Emmanuel Giroux for his invaluable advice and support through all the stages of this work.

This final text has benefitted a lot from every opportunity I had to present and discuss my work, and I am deeply grateful to all those who gave me this opportunity, and more generally showed their interest and shared their thoughts. I wish to thank Jonathan Bowden, in particular, for his precious help. I would also like to thank the referees for their many helpful comments and suggestions, especially regarding the Appendix.

\section{From plane fields to foliations on the solid torus}\label{s:prelim}

A key step in Thurston's construction of foliations (cf. ``Step 3" in ``Thurston's method'' above) consists in deforming any given plane field $\xi$ transverse to the $\sphere^1$ factor on the solid torus $\D^2\times \cercle^1$ to a foliation, relative to the boundary. Larcanch\'e proved in \cite{La} that this can actually be done continuously with respect to $\xi$, using a construction introduced by Schweitzer in \cite{Sc} as an alternative to Thurston's method. In this section, we give a brief account of these works.

Given a manifold $M$ (possibly with boundary), $\PP(M)$ (resp. $\FF(M)$) denotes the space of hyperplane fields (resp. codimension $1$ foliations) on $M$. Now given a manifold $B$, we denote by $\PPt(B\times\cercle^1)$ (resp. $\FT (B \times \sphere^1)$) the subspace of $\PP(B\times\sphere^1)$ (resp. $\FF(B\times\sphere^1)$) made up of plane fields transverse to the $\sphere^1$ factor.

\begin{proposition}\label{p:third-param} There is a continuous path of maps $\psi_t:\PPt(\D^2\times\cercle^1)\to \PP(\D^2\times\cercle^1)$, $t\in[0,1]$, such that:
\begin{itemize}
\item $\psi_0$ is the inclusion,
\item $\psi_1$ has value in $\FF(\D^2\times\cercle^1)$,
\item for every $\xi$ in $\PPt(\D^2\times\cercle^1)$, all plane fields $\psi_t(\xi)$, $t\in[0,1]$, coincide along $\partial\D^2\times\sphere^1$.
\end{itemize}
 \end{proposition}

\begin{remark}\label{r:simplecon} 
Since the disk is simply connected, the only foliation of $\D^2 \times \cercle^1$ transverse to $\cercle^1$, up to fibered isotopy, is the foliation by meridian disks $\D^2 \times \{\cdot\}$. In other words, only the foliation of $\partial \D^2 \times \cercle^1$ by meridian circles extends to a foliation of $\D^2 \times \cercle^1$ transverse to $\cercle^1$. Thus, most foliations in $\psi_1(\PPt(\D^2\times\cercle^1))$ will not be everywhere transverse to $\cercle^1$.
\end{remark}

The whole section is devoted to the proof of Proposition \ref{p:third-param}. The main issue is to construct the map $\psi_1$ which, to any plane field transverse to the $\sphere^1$ factor on $\D^2\times\sphere^1$, associates a foliation having the same trace on the boundary.  Let us start with the simple yet key example of a plane field $\xi$ defined by an equation of the form $dz-\rho(r)\lambda d\t=0$, where $(r,\t)$ denote the polar coordinates on $\D^2$, $z$ the coordinate on $\sphere^1$, $\lambda$ some real number and $\rho$ some smooth step function vanishing on $[0,1/2]$ and equal to $1$ near $1$. This plane field induces a linear foliation on $\partial \D^2 \times \cercle^1$, which can be extended to a foliation of the solid torus by putting a Reeb component along the core curve and wrapping the external leaves around it as shown on Fig. \ref{filling}.

\begin{figure}[htbp]
\centering
\begin{tabular}{ccccc}
\includegraphics[height=2.5cm]{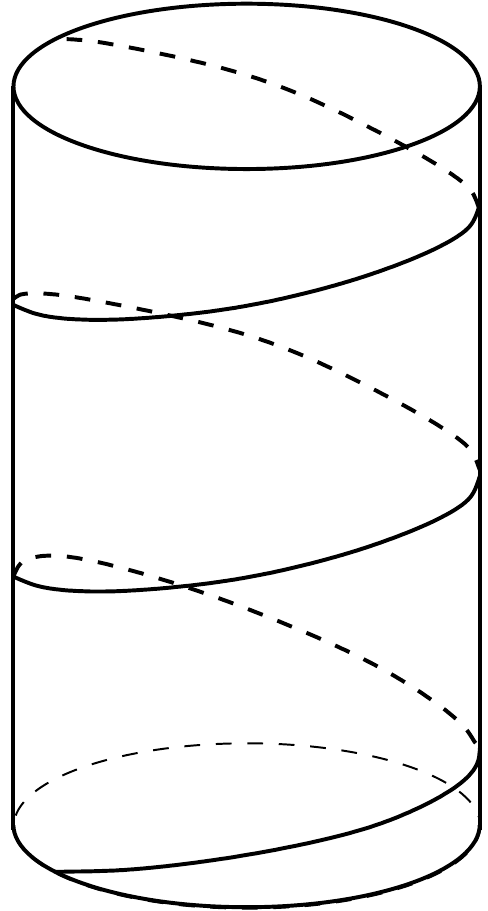} 
  && \includegraphics[height=2.5cm]{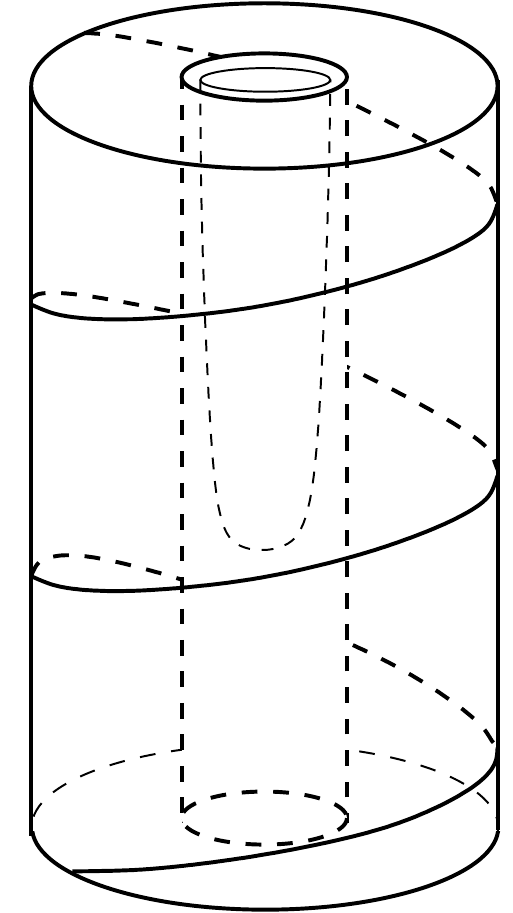} 
&& \includegraphics[height=2.5cm]{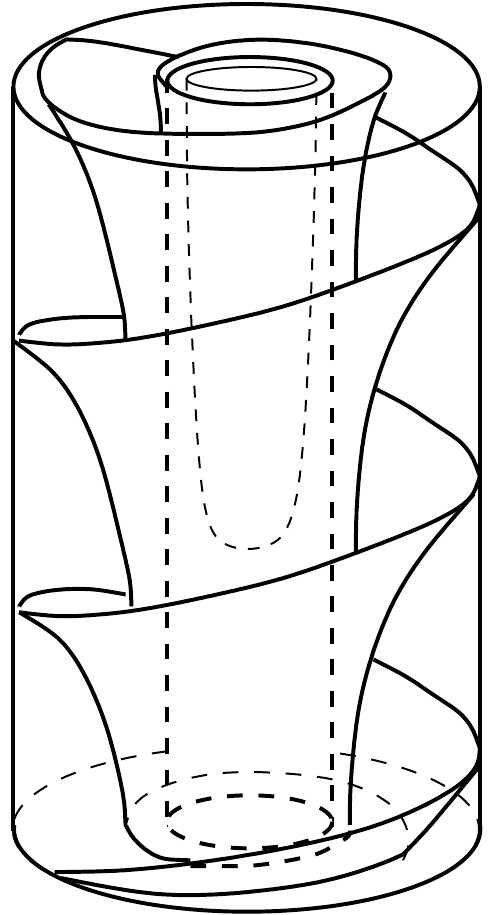}
\end{tabular}
\caption{Reeb filling of a linear foliation\label{filling}}
\end{figure}

Moreover, this foliation is homotopic to the initial plane field relative to the boundary. Lemma \ref{l:reeb-fill} below (applied to  $\omega=dz-\lambda d\t$), gives an analytic description of these objects. Let $\{\rho_0,\rho_{1/2},\rho_1\}$ denote a partition of unity on $[0,1]$ meeting the following conditions:
\begin{itemize}
\item $\rho_1$ equals $1$ near $1$ and $0$ precisely on $[0,1/2]$;
\item $\rho_0$ equals $1$ near $0$ and $0$ precisely on $[1/2,1]$;
\item as a consequence, $\rho_{1/2}$ equals $0$ near $\{0,1\}$ and $1$ precisely at $1/2$.
\end{itemize}

\begin{lemma}[Reeb Filling Lemma]\label{l:reeb-fill}
For every non singular closed $1$-form $\omega$ on $\partial \D^2\times\cercle^1$, the $1$-form
$$ \bar\omega = \rho_1(r) \, \omega + \rho_{1/2}(r) \, dr
 + \rho_0(r) \, dz $$ 
(where $\rho_1(r)\omega$ is viewed as a form on the solid torus) is nonsingular, integrable on $\D^2\times\cercle^1$ and induces $\omega$ on the boundary. Moreover, if $\omega(\partial_z)>0$, the $1$-forms
$$ \bar\omega_t = \rho_1(r) \, \omega
 + \rho_{1/2}(r) \, (t \, dz + (1-t) \, dr) + \rho_0(r) \, dz, \quad 
   t \in [0,1], $$
are all nonsingular, integrable if $\omega=dz$ (but not in general) and define a homotopy of plane fields relative to the boundary between the plane field tangent to the foliation and a plane field transverse to the $\cercle^1$ factor.
\end{lemma}

A foliation of the form $\bar\omega$ will be called a \emph{Reeb filling of $\omega$}, a \emph{Reeb filling of slope $\lambda$} if $\omega=dz-\lambda d\t$, or simply a \emph{Reeb foliation}. 
\medskip

\begin{remark}\label{r:reeb} 
If $\omega=dz$, $t\in[0,1]\mapsto\bar\omega_{1-t}$ defines a deformation \emph{of foliations} between the product foliation by meridian disks and a Reeb filling of slope $0$ (cf. Figure \ref{f:creation-reeb}: to visualize the deformation, rotate each picture of the central sequence around a vertical axis. This sequence represents a continuous path of dimension one foliations invariant under vertical translations, the continuous deformation of their tangent line fields being sketched above).
\end{remark}

\begin{figure}[htbp]
 \centering
\begin{tabular}{cccccc}
\includegraphics[height=3cm]{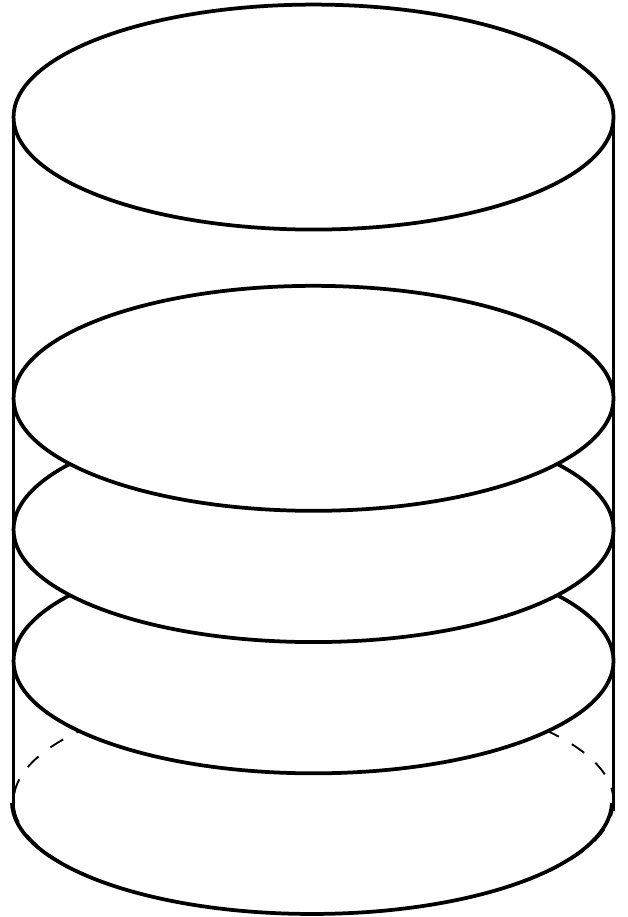} 
  & \includegraphics[height=3.5cm]{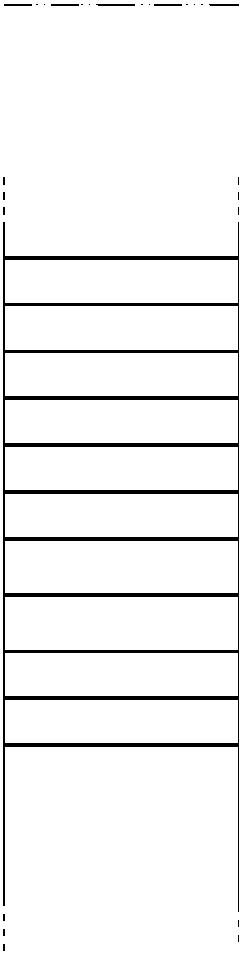} 
& \includegraphics[height=3.5cm]{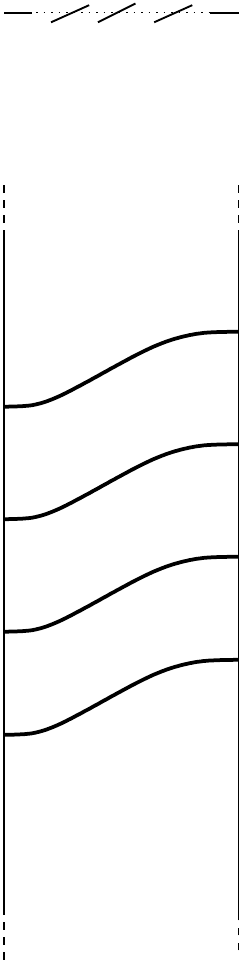}
 & \includegraphics[height=3.5cm]{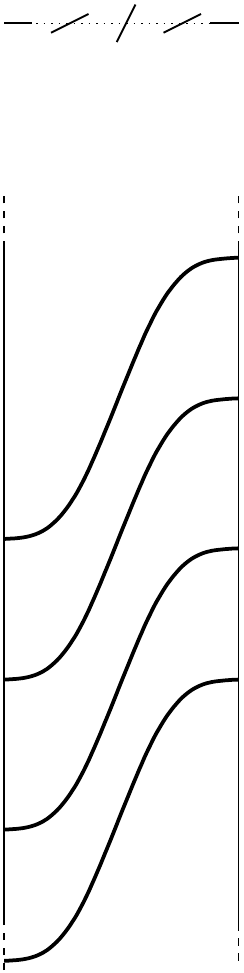}
& \includegraphics[height=3.5cm]{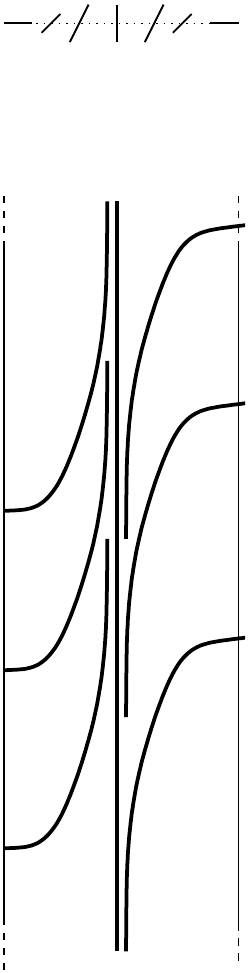}
&\includegraphics[height=3cm]{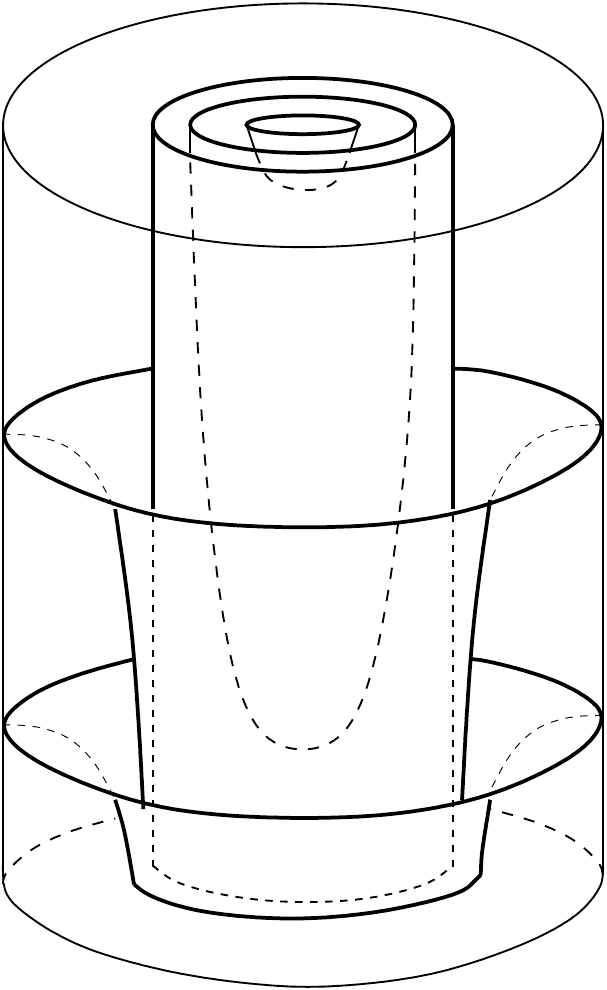} 
\end{tabular}
\caption{Addition of a Reeb component\label{f:creation-reeb}}
\end{figure}

Note that if $\omega$ is not closed, \emph{i.e.} if the foliation on the boundary torus is not \emph{linearizable}, $\bar\omega$ is not integrable. In fact, one can prove (cf. for example \cite[Lemma 2.1]{C-C2}) that Reeb's construction does not generalize to nonlinearizable foliations, due to some rigidity phenomenon concerning the \emph{holonomy} of a $C^2$ foliation near a torus leaf (in relation with Kopell's Lemma \cite{Ko} for commuting $\CC^2$ diffeomorphisms of the interval). The idea is nevertheless to reduce to the linearizable case. To do so, one first needs to translate Proposition \ref{p:third-param} in terms of \emph{holonomy}.

Let $\tau$ be an element of $\FT (\cercle^1 \times \cercle^1)$. The transversality condition implies that, for every $x$ in $\cercle^1$, the leaf through $(1,x)\in\cercle^1\times\cercle^1$ goes all the way around the torus, alternately intersecting every fiber $\{e^{2\pi i t}\} \times \cercle^1$, $t\in[0,1]$, at a point $(e^{2\pi i t}, f_t(x))$. This defines a one-parameter
family $(f_t)_{t\in[0,1]}$ of smooth orientation-preserving diffeomorphisms of the circle, which has a unique lift $(\wt{f_t})_{t\in[0,1]}$ to $\DDi$ --~the group of orientation-preserving diffeomorphisms of $\R$ commuting with the unit
translation~-- satisfying $\wt{f_0}=\id_\R$. We call  the diffeomorphism $\wt{f_1}$ the \emph{holonomy} of the foliation $\tau$, and denote it by $hol(\tau)$. For example, the holonomy of the foliation defined by $dz-\lambda d\t=0$ is the translation $T_{2\pi\lambda}: x\mapsto x+2\pi\lambda$ (hence, we will sometimes call the corresponding Reeb filling the \emph{Reeb filling of $T_{2\pi\lambda}$}). Now the following standard facts (which we will not prove) allow us to reduce Proposition \ref{p:third-param} to Theorem \ref{t:larcanche} below:

\begin{lemma}\label{l:holonomie1} 
The map 
$$\begin{matrix}
 \PPt(\D^2\times\cercle^1) &\to&\FT (\cercle^1 \times
\cercle^1)\\
\xi&\mapsto &\xi\res{ \partial\D^2\times\cercle^1}
\end{matrix}$$
is a trivial fibration with contractible fibres.
\end{lemma}

\begin{lemma}\label{l:holonomie} 
The holonomy map 
$$ hol : \FT (\cercle^1 \times
\cercle^1) \to \DDi $$
is a trivial fibration with contractible fibers.
\end{lemma}

\begin{theorem}[Schweitzer {\cite[``Step 4'', p. 182]{Sc}}, Larcanch\'e {\cite[Theorem 4, p. 396]{La}}] \label{t:larcanche}
There is a continuous path of maps 
$$\f_t:\DDi\to \PP(\D^2\times\sphere^1),\quad t\in[0,1],$$
 such that:
\begin{itemize}
\item $\f_0$ has value in $\PPt(\D^2\times\sphere^1)$,
\item $\f_1$ has value in $\FF(\D^2\times\sphere^1)$,
\item for every $f\in\DDi$, all plane fields $\f_t(f)$, $t\in[0,1]$, coincide along $\partial\D^2\times\sphere^1$ and $hol(\f_t(f)\res{\partial\D^2\times\sphere^1})=f$.
\end{itemize}
The foliations $\LL_f:=\f_1(f)$, $f\in\DDi$, will be referred to as \emph{Schweitzer foliations}. 
\end{theorem}

Hence, in Proposition \ref{p:third-param},  the second point can be replaced by: ``for every $\xi$ in $\PPt(\D^2\times\sphere^1)$, $\psi_1(\xi)$ is a Schweitzer foliation". 

Let us now present the proof of Theorem \ref{t:larcanche}. Again, the main issue is to construct the continuous map $\f_1$ which, to any holonomy $f\in\DDi$, associates a foliation of $\D^2\times\sphere^1$ transverse to the boundary and whose restriction to the boundary has holonomy $f$. As a matter of fact, only $\f_1$ appears explicitly in \cite{Sc,La}, but the existence of $\f_t$ is a straightforward consequence of the construction of $\f_1$, as we will see. 

The idea to construct $\f_1$ is to reduce to the translation case (solved by Reeb's construction) by translating the following decomposition result of Herman for diffeomorphisms in terms of ``foliation merging". 

\begin{theorem} [Herman \cite{He}, Corollaire (5.2) p.127]\label{t:herman}
Let $\mu = (1+\sqrt5)/2$ denote the golden ratio. There is a continuous map 
$$  \begin{aligned} 
   \DDi & \rightarrow \R \times \DDi \\
   f & \mapsto (\lambda_f, g_f)
\end{aligned}  $$
such that $f = T_{\lambda_f} \circ (g_f^{-1} \circ T_\mu \circ g_f)$ for all
$f \in \DDi$, and $(\lambda_{\id}, g_{\id}) = (-\mu, \id)$.
\end{theorem}

\begin{remark}\label{r:herman}\begin{itemize}
\item Schweitzer did not use or mention the continuous character of the map in \cite{Sc}, it is Larcanch\'e who saw the potential of Herman's theorem for the parametric case.
\item 
Actually, Herman proves this result for any number $\mu$ in a full-measure set, and works by Yoccoz show that it is true for any diophantine number $\mu$. However, in what follows, we only need it to be true for one number $\mu$.
\end{itemize}
\end{remark}

We can now describe Schweitzer's foliation $\f_f$ for a given holonomy $f$. Roughly speaking, $\LL_f$ is obtained by taking the
Reeb fillings of $T_{\lambda_f}$ and $g_f^{-1} \circ T_\mu \circ g_f$, gluing them together as Fig. \ref{glue} suggests, and inflating the resulting picture a little to remove the angles. 
\begin{figure}[htbp]
\centering
\begin{tabular}{cc}
\includegraphics[height=2.5cm]{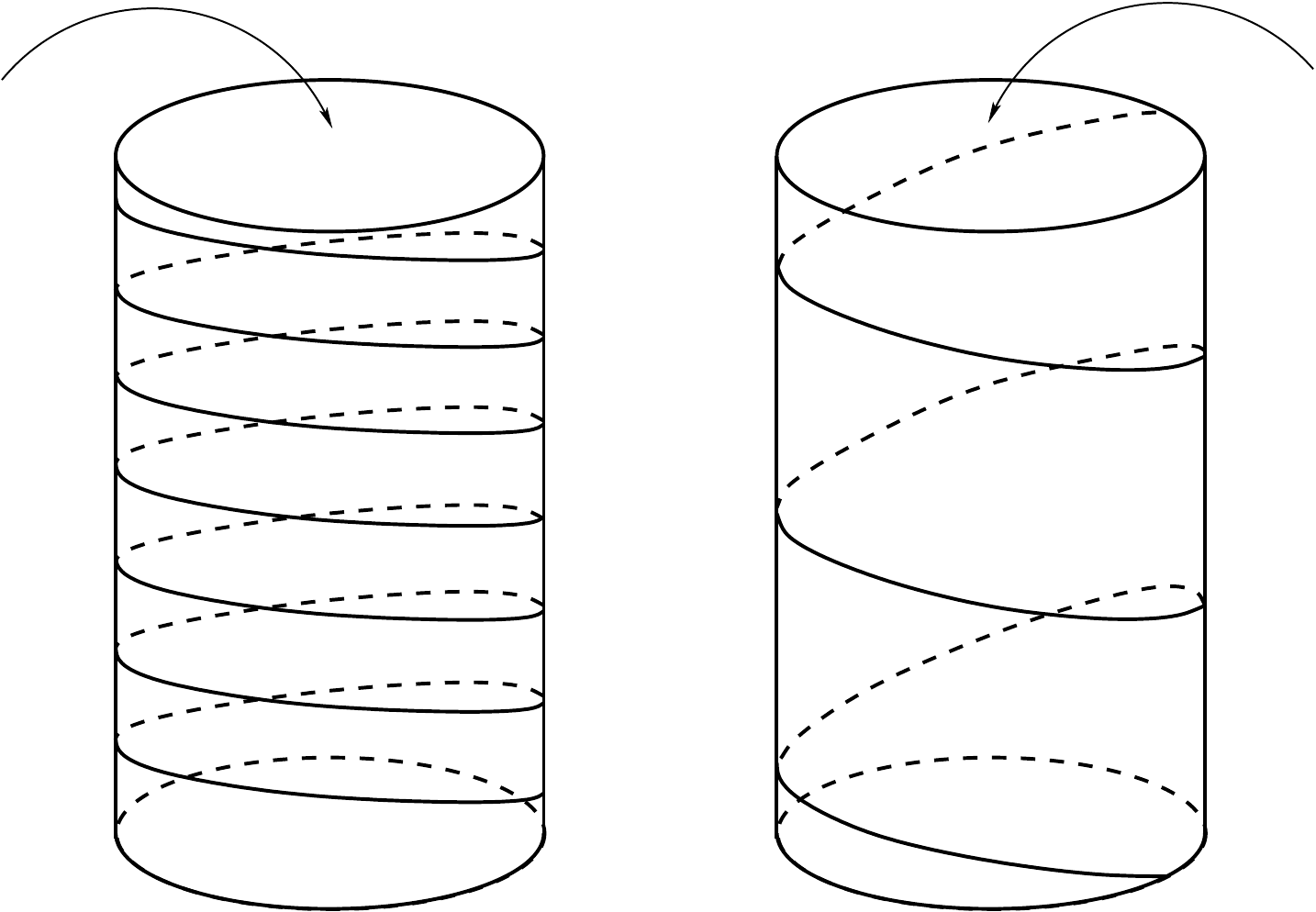} 
\put(0,55){\small{Reeb filling of}}
\put(10,45){$\scriptstyle g_f^{-1} \circ T_\mu \circ g_f$}
\put(-153,55){\small{Reeb filling}}
\put(-143,45){\small{of }$\scriptstyle T_{\lambda_f}$}
\hspace{2.5cm}
  & 
  \includegraphics[height=2.6cm]{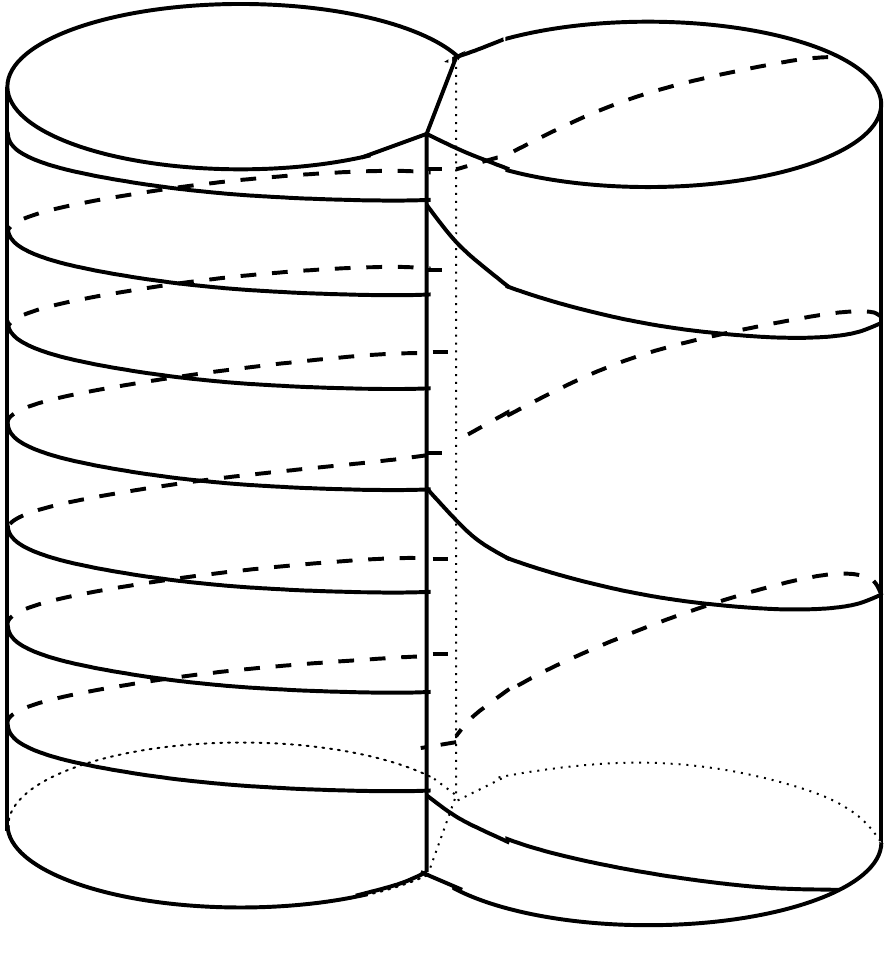} 
\end{tabular}
\caption{Foliation merging\label{glue}}
\end{figure}
The holonomy on the boundary of the resulting bigger solid torus is exactly the composition of the holonomies on the smaller
tori, \emph{i.e} precisely $f$. A more rigorous argument is given below, using suspension foliations over a pair of pants (cf. Lemma \ref{l:pantalon}).

\begin{remark} \label{r:lid}
Note that when $f$ is the identity (i.e. when $\partial \D^2\times\sphere^1$ is foliated by meridian circles), Schweitzer's foliation $\LL_\id$ is not the foliation by meridian disks; it consists of two Reeb fillings (of $T_\mu$ and $T_{-\mu}$ respectively) ``glued together". The ``inflated'' picture is depicted on Fig.~\ref{f:lid}. It will be important for us, however, to observe that $\Lid$ and the foliation by meridian disks can be deformed to one another \emph{through foliations}, relative to the
boundary. Indeed, given the product foliation on $\D^2 \times \cercle^1$, dig two parallel Reeb components in $D_{\pm} \times \cercle^1$ (cf. Remark \ref{r:reeb}), for some small disks $D_{\pm}\subset \D^2$. Then make the slope of the foliations induced on $\partial D_\pm \times\cercle^1$ vary from $0$ to $\pm \mu/2\pi$ respectively. This deformation easily extends to $(\D^2\setminus(D_+\cup D_-))\times\cercle^1$ rel. $\partial \D^2\times\cercle^1$ (cf. Lemma \ref{l:pantalon} below), and to
$D_{\pm}\times \cercle^1$ using the Reeb Filling Lemma \ref{l:reeb-fill}. 

Similarly, the Reeb and Schweitzer foliations associated to a translation $T_\lambda$ can be deformed to one another \emph{through foliations}, relative to the boundary.
\end{remark}
\begin{figure}[htbp]
\centering
\includegraphics[height=4cm]{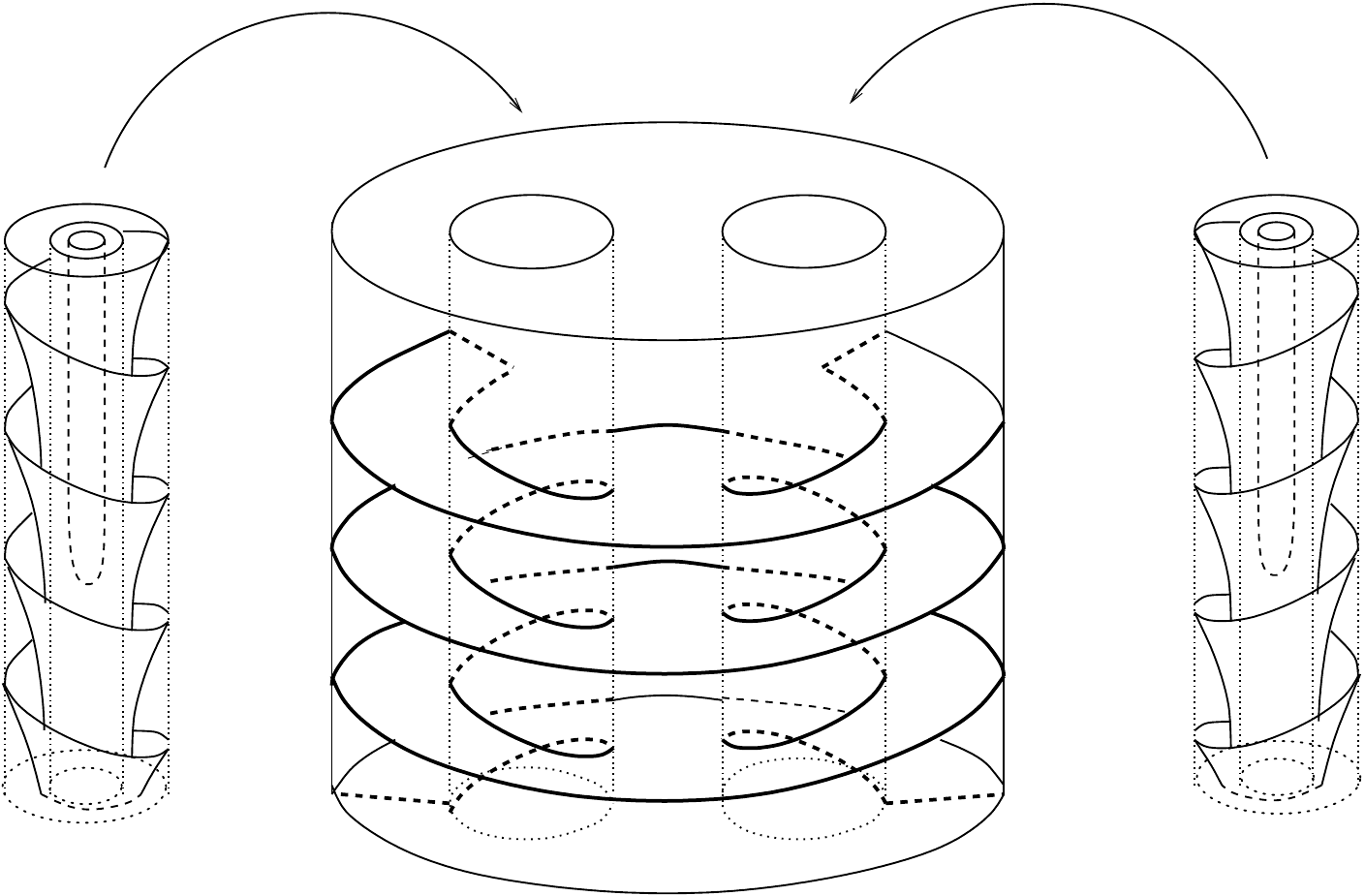} 
\caption{The Schweitzer foliation $\Lid$ \label{f:lid}}
\end{figure}

In the general case, the resulting foliation $\f_f$ is homotopic rel. boundary to a plane field transverse to the $\cercle^1$ factor (because Reeb fillings are), and as Larcanch\'e observed, all of this can be done continuously with respect to $f$ because this is true for the decomposition of $f$ and the ``merging procedure". Let us clarify this last claim using Lemma \ref{l:pantalon} below (which we will not prove), which reflects the flexibility of suspension foliations over a pair of pants, and will be used independently on several occasions in Section \ref{s:malleabilization}. \medskip


To fix notations, let $\P$ denote the (oriented) pair of pants obtained by removing the interiors of the disks $D_\pm$ of radius
$1/4$ centered at $\pm(1/2,0)$ from the unit disk $\D^2\subset\R^2$. Let $\partial_\pm\P=\partial D_\pm$ and
$\partial_0\P=\partial \D^2$ (oriented as the boundary of $D_\pm$ and $\D^2$ respectively). Let $V \subset \P$ be the union of two segments joining $(0,-1) \in \partial \D^2$ to $\pm(1/4,0) \subset \partial_\pm \P$ respectively, and $\FTV(\P \times \cercle^1)$ the subspace of $\FT (\P \times \cercle^1)$ made of foliations inducing the horizontal foliation by $V \times \{\cdot\}$ on $V \times \cercle^1$. 
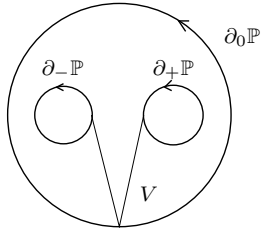
\begin{figure}[htbp]
\centering
\includegraphics[width=3cm]{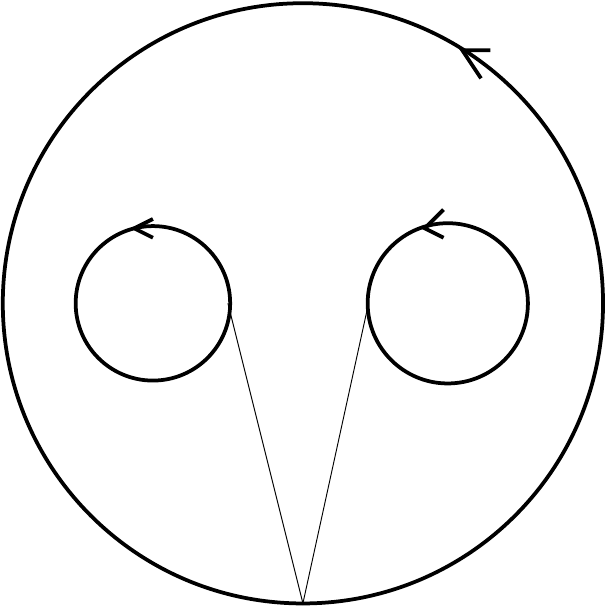}
\put(-3,70){$\scriptstyle \partial_0\P$}
\put(-30,58){$\scriptstyle \partial_+\P$}
\put(-72,58){$\scriptstyle \partial_-\P$}
\put(-35,10){$\scriptstyle V$}
\caption{Boundary components of $\P$ \label{fig:pantalon}}
\end{figure} Denote by $\GG_0$ the group of fibered diffeomorphisms of $\P \times \cercle^1$ covering the identity which restrict to the identity on $(V \cup \partial_{\pm} \P) \times \cercle^1$. The group $\GG_0$ is contractible and acts on $\FTV(\P \times \cercle^1)$. Using $V \cap \partial_i\P$ as a base point on $\partial_{i}\P$, $i \in \{+,-,0\}$, we get holonomy maps $h_i:\FTV (\P \times \cercle^1) \to \DDi$, $h_i(\tau) = hol(\partial_i\tau)$, where $\partial_i\tau$ denotes the foliation induced by $\tau$ on $\partial_i\P$, which satisfy $h_0(\tau) = h_-(\tau) \circ h_+(\tau)$.

\begin{lemma}\label{l:pantalon}
The restriction map
$$  \begin{aligned}
\FTV (\P \times \cercle^1) & \to 
\FT (\partial_-\P \times \cercle^1) \times \FT (\partial_+ \P \times \cercle^1) \\
\tau & \mapsto (\partial_-\tau, \partial_+\tau) 
\end{aligned}  $$ 
is a trivial fibration with contractible fibers (the orbits of $\GG_0$).
\end{lemma}

\noindent\emph{Proof of Proposition \ref{t:larcanche} in terms of Lemma \ref{l:pantalon}}. Let $f\in\DDi$. Define on $\cercle^1\times\cercle^1$ two closed forms $\omega_+=$ $g_f'(z)dz - \mu d\t$ and $\omega_-=dz-\lambda_fd\t$ (cf. Theorem \ref{t:herman}), which define foliations of holonomy $g_f^{-1} \circ T_\mu \circ g_f$ and $T_{\lambda_f}$ respectively. According to Lemma \ref{l:pantalon}, there exists a foliation of $\FTV(\P \times \cercle^1)$ whose restrictions $\partial_{\pm}\tau$ to $\partial_{\pm}\P \times \cercle^1$ are the foliations defined by $\omega_\pm$ and whose restriction $\partial_0\tau$ to $\partial_0\P \times \cercle^1$ has holonomy $f = T_{\lambda_f} \circ (g_f^{-1} \circ T_\mu \circ g_f)$. Then apply the Reeb filling Lemma \ref{l:reeb-fill} to the forms $\omega_{\pm}$ to define $\f_t(f)$ in $D_{\pm}\times \cercle^1$. Everything can be done continuously with respect to $f$.\hfill\qedsymbol

\section{Flexibility of malleable foliations}\label{s:malleable}

In this section, we give a proper definition of almost integrable plane fields and malleable foliations, and prove Theorem \ref{t:malleable}, that is that any path of plane fields connecting two malleable foliations can be deformed with fixed endpoints to a path of malleable foliations. This follows from Proposition \ref{p:p-i}, Lemma \ref{l:desintegration} and Proposition \ref{p:third}, which can be seen as one-parameter versions of Thurston's first, second and third steps in \cite{Th2} (cf. Introduction). We conclude with a sketch of proof of Theorem B.

From now on, given a $3$-manifold $M$, by a \emph{collection of balls} (resp. \emph{solid tori, arcs...}) in $\Int(M)$, we mean a finite union of disjoint such things.


\subsection{Almost-integrable plane fields and malleable foliations} \label{ss:mal-pmal}

\begin{definition} \label{d:p-i}
Let $M$ be a $3$-manifold (possibly with boundary) and $B=\bigcup_i B_i \subset \Int(M)$ a collection of balls. A plane field $\xi$ on $M$ is \emph{almost horizontal on $B$} if it is integrable near $\partial B$ and satisfies the following conditions for some
parametrization (called \emph{adapted}) of each ball $B_i$ by the unit ball $\D^3$:
\begin{itemize}
\item
$\xi$ is tangent to $S_i = \partial B_i = \partial \D^3$ exactly at the poles $z=\pm1$;
\item
for every $\eps > 0$, there exists a nonsingular vector field $\nu$ on $B_i =
\D^3$ everywhere positively transverse to $\xi$ and to the horizontal plane
field $dz=0$ and tangent to $S_i$ outside an $\eps$-neighbourhood of the poles.
\end{itemize}
A plane field $\xi$ on $M$ is \emph{$B$-almost integrable} if it is integrable on $M \setminus \Int B$ and almost horizontal on $B$. A $B$-almost integrable plane field $\xi$ on $M$ is \emph{taut} if the induced foliation on $M \setminus \Int B$ is taut (meaning that every transverse arc in $M \setminus \Int B$ extends to a closed transversal in $M \setminus \Int B$).
Finally, a plane field $\xi$ is \emph{almost integrable} if it is $B$-almost integrable for some $B$.
\end{definition}

\begin{definition}\label{d:malleable-dhf}
A Schweitzer foliation of the solid torus is \emph{simple} if its holonomy on the boundary torus has whole intervals of fixed points, that is if the induced foliation on $\partial \D^2\times\sphere^1$ has a one-parameter family of closed leaves bounding meridian disks.
\end{definition}

\begin{definition}\label{d:malleable}
Let $M$ be a $3$-manifold. A codimension-one foliation $\tau$ on $M$ is \emph{malleable}
if there is a collection of solid tori $W=\bigcup_i W_i\subset \Int(M)$ such that:
\begin{itemize}
\item $\tau$ induces a simple Schweitzer foliation on each $W_i$,
\item $\tau$ induces a taut foliation on $M\setminus \Int W$.
\end{itemize}
\end{definition}

 In particular, taut foliations are malleable. Recall that on a \emph{closed} $3$-manifold, only torus leaves can fail to meet a closed transversal (cf. \cite{No,Go}), and that we referred to such problematic leaves as \emph{Novikov tori} of the foliation (cf. ``Reduction to the malleable case'' in the Introduction). Thus, a foliation on a \emph{closed} $3$-manifold is malleable if all its Novikov tori are torus leaves of (simple) Schweitzer foliations.\medskip

There is a natural correspondence between malleable foliations on the one hand and, on the other hand, taut almost integrable plane fields together with an additional piece of data, namely, for each ball of the associated collection $B$, a transverse arc connecting the poles in $M \setminus \Int B$ (cf. Fig. \ref{f:tunnel}):

\begin{lemma}\label{l:correspondance}
Let $M$ be a $3$-manifold, $B = \bigcup_1^n B_i \subset \Int(M)$ a collection of balls, $\xi$ a taut $B$-almost integrable plane field on $M$ and $\{ A_i \}_{1\le i\le n}$ disjoint arcs transverse to $\xi$, each $A_i$ connecting the poles of $B_i$ in $M \setminus \Int B$. Then there is a malleable foliation $\tau$ on $M$ with the following properties:
\begin{itemize}
\item
the solid tori $W_i$ associated to $\tau$ are neighbourhoods of $B_i \cup A_i$, $1\le i\le n$;
\item
the plane fields $\xi$ and $\tau$ are homotopic relative to $M \setminus
\bigcup_i\Int W_i$.
\end{itemize}
\end{lemma}

Conversely:

\begin{lemma}\label{l:inverse}
Let $M$ be a $3$-manifold, $\tau$ a malleable foliation on $M$ and $W = \bigcup_1^n W_i$ the corresponding collection of solid tori. There exists a plane field $\xi$, together with balls $B_i \subset W_i$ and arcs $A_i \subset \Int W_i \setminus \Int B_i$ such that:
\begin{itemize}
\item
$\xi$ is taut $B$-almost integrable, where $B = \bigcup_i B_i$ ;
\item
each arc $A_i$ is transverse to $\xi$ and connects the poles of $B_i$ ;
\item
the plane fields $\tau$ and $\xi$ are homotopic relative to $M \setminus \Int W$.
\end{itemize}
\end{lemma}  

\begin{proof}[Proof of Lemma \ref{l:correspondance}]
This follows readily from Theorem \ref{t:larcanche}.  For $1 \le i \le n$, we can parametrize a neighbourhood $W_i$ of $B_i \cup A_i$ by $\D^2 \times \cercle^1$, so that: 
\begin{itemize} 
\item $A_i = \{0\} \times J_i$ for some interval $J_i$ of $\cercle^1$ and $\xi$ is tangent to the disks $\D^2 \times \{\cdot\}$ on $\D^2 \times J_i$ ; 
\item $\xi$ is transverse to the $\cercle^1$ factor, and integrable in a neighbourhood of $\partial W_i$.  
\end{itemize} 
Let $f_i \in \DDi$ be the holonomy of the foliation induced by $\xi$ on $\partial W_i$. According to Theorem \ref{t:larcanche}, the tangent plane field to the (simple) Schweitzer foliation $\LL_{f_i}$ is homotopic to $\xi$ relative to $\partial W_i$. Let $\tau$ be the foliation of $M$ which coincides with $\xi$ on $M'= M\setminus \bigcup \Int W_i$ and with $\LL_{f_i}$ on $W_i$. It remains to prove that $\tau\res {M'}=\xi\res {M'}$ is taut. Let $L'$ be a leaf of $\tau\res {M'}$, and $L$ the leaf of $\xi\res {M\setminus \Int B}$ such that $L'=L\cap M$. By assumption on $\xi$, $L$ is crossed by a closed transversal $\Gamma$ to $\xi\res {M\setminus \Int B}$, and we can assume that $\Gamma$ meets $L$ at some point of $L'$. Now it is not difficult to push $\Gamma$ out of $W\setminus B$. 
\end{proof}

\begin{remark}\label{r:correspondance-param}
According to Proposition \ref{p:third-param}, the above construction can actually be performed continuously on \emph{families} of taut almost integrable plane fields given with transverse arcs provided these transverse arcs vary continuously. Proposition \ref{p:third} below will show how to get rid of the latter condition.
\end{remark}

\begin{proof}[Proof of Lemma \ref{l:inverse}]
For $1 \le i \le n$ and for some suitable parametrization of $W_i$ by $\D^2 \times \sphere^1$, the Schweitzer foliation $\tau \res{W_i}$ is homotopic relative to the boundary to a plane field $\bar\xi$ transverse to the $\sphere^1$ factor and
whose restriction to $\partial W_i = \partial \D^2 \times \sphere^1$ is tangent to $\partial \D^2 \times \{z\}$ for all $z$ in some interval $J_i$ of $\cercle^1$. We can thus deform $\bar\xi$ relative to the boundary among plane fields transverse to $\sphere^1$ into a plane field $\xi$ tangent to the disks $\D^2 \times \{\cdot\}$ on $\D^2 \times J_i$. We then define $B_i$ as the ball obtained after rounding the box $W_i\setminus (\D^2\times J_i)=\D^2 \times (\sphere^1 \setminus J_i)$, making sure
that $\partial B_i$ has exactly two tangency points with $\xi$: the poles, located on the core curve $\{0\} \times \sphere^1$ of $W_i$. If we parametrize $B_i$ by $\D^3$ in such a way that the third coordinate coincides with the coordinate $z$ in $\cercle^1$, for every $\eps>0$, the vector field $\partial_z$ on $\D^2 \times (\cercle^1 \setminus J_i)$ can easilly be extended into a vector field $\nu$ satisfying all the properties of Definition \ref{d:p-i}. The sub-arc $A_i$ of $\{0\} \times
J_i$ connecting the poles of $B_i$ is transverse to $\xi$.
\end{proof}


\subsection{Flexibility of taut almost integrable plane fields}\label{ss:flex-p-m}

Here we prove:

\begin{proposition}
\label{p:p-m}
Let $\xi_0$ and $\xi_1$ be taut $B$-almost integrable plane fields on a closed $3$-manifold $M$, for some collection of balls $B\subset M$. Any path $(\xi_t)_{t\in[0,1]}$ of plane fields on $M$ connecting $\xi_0$ to $\xi_1$ is homotopic with fixed endpoints to a path of taut $\Bh$-almost integrable plane fields, where $\Bh$ is a collection of balls \including$B$.
\end{proposition}

Here, we say that a collection of balls $\Bh=\bigcup_1^m \Bh_j$ \emph{ \includes}a collection of balls $B=\bigcup_1^nB_i$ if $\{B_i\}_i$ is a subset of $\{\Bh_j\}_j$. Proposition \ref{p:p-m} follows readily from the next three results, in which the notation $\Op(A)$, for a subspace $A$ of any topological space, refers to a small nonspecified open neighbourhood of $A$. Lemma \ref{l:p-h} and Proposition \ref{p:p-i} provide a relative one-parameter version of Thurston's ``step 1'' (cf. Introduction) while Lemma \ref{l:desintegration} corresponds to ``step 2'' and is applied to the restriction to $N=M\setminus \Int \Bb$ of the path provided by Proposition \ref{p:p-i}. 

\begin{lemma}\label{l:p-h}
Let $\xi_0$ and $\xi_1$ be $B$-almost horizontal plane fields on a closed $3$-manifold $M$, for some collection of balls $B\subset M$. Any path $(\xi_t)_{t\in[0,1]}$ of plane fields on $M$ connecting $\xi_0$ to $\xi_1$ is homotopic with fixed end points to a path of $B$-almost horizontal plane fields.
\end{lemma}

\begin{proposition}\label{p:p-i}
Let $\xi_0$ and $\xi_1$ be $B$-almost integrable plane fields on a closed $3$-manifold $M$, for some collection of balls $B\subset M$. Any path $(\xi_t)_{t\in[0,1]}$ of $B$-almost horizontal plane fields on $M$ connecting $\xi_0$ to $\xi_1$ is homotopic with fixed end points and rel. $\Op(B)$ to a path of $\Bb$-almost integrable plane fields, where $\Bb$ is a collection of balls \including$B$.
\end{proposition}

\begin{lemma}\label{l:desintegration}
Let $\xi_0$ and $\xi_1$ be taut foliations on a compact $3$-manifold $N$ (possibly with boundary), and $(\xi_t)_{t\in[0,1]}$ a path of foliations on $N$ connecting $\xi_0$ to $\xi_1$ and having no component of $\partial N$ as a leaf. Then $(\xi_t)_{t\in[0,1]}$ is homotopic with fixed end points and rel. $\Op(\partial N)$ to a path $(\bar\xi_t)_{t\in[0,1]}$ of \emph{taut} $B'$-almost integrable plane fields, for some collection of balls $B'\subset \Int(N)$.
\end{lemma}

Lemmas \ref{l:p-h} and \ref{l:desintegration} are proved below. Proposition \ref{p:p-i} will be discussed in the appendix. Actually this result is a version for foliations of a theorem established by Eliashberg for contact structures (see \cite[Lemma 3.2.1]{El}). Though the key ideas of the proof are purely geometrical, their implementation requires some tedious technical estimates which will be carried out in full detail. 

The first step of the proof of Lemma \ref{l:p-h} consists in reducing to the case where $\xi_0$ and $\xi_1$ coincide on $B$ and are \emph{horizontal} on each ball of $B$ in some adapted coordinates:

\begin{affirmation}\label{a:p-h} 
Let $M$ be a closed $3$-manifold, $\xi$ a plane field on $M$ almost horizontal on some ball $B\subset M$, and 
$(x,y,z)$ the coordinates induced on $B$ by some adapted parametrization. Then $\xi$ can be deformed relative to $M \setminus \Op (B)$ among $B$-almost horizontal plane fields to a plane field defined by $dz = 0$ on $B$.
\end{affirmation}

\begin{proof}[Proof of Claim \ref{a:p-h}]
We want to straighten out $\xi$ in $B$ while keeping it fixed outside $\Op(B)$. The difficulty is to do this \emph{through $B$-almost horizontal plane fields}.

Let $\alpha$ be an equation of $\xi$ which, near each pole of $B$, coincides with the differential of some function $f$. Define
$\xi_t$, $t \in [0,1]$, as the kernel of the form 
$$\alpha_t = (1-\rho) \alpha +\rho \left((1-t) \alpha + t \, dz\right),$$ 
where $\rho \from M \to [0,1]$ is a smooth function equal to $1$ near $B$ and with compact support in a neighbourhood $U$ of $B$, small enough that all the forms $(1-t) \alpha + t \, dz$ are nonsingular on $U$.

Clearly, $\xi_0 = \xi$, all the plane fields $\xi_t$ coincide with $\xi$ outside $U$, and $\xi_1 \res B$ is defined by $dz = 0$. Moreover, all the plane fields $\xi_t$ are integrable near the poles of $B$ for $\alpha_t$ equals $(1-t) \, df + t \, dz$ there.

Let us now show that $\xi_t$ is tangent to $\partial B$ exactly at the poles $p_\pm$. Let $p \in \partial B \setminus \{p_\pm\}$ and $\eps < \dist(p, \{p_\pm\})$. According to Definition \ref{d:p-i}, there exists a vector field $\nu$ on $B$ positively transverse to $\xi$ and to the $z$-levels and tangent to $\partial B$ outside an $\eps$-neighbourhood of the poles. By construction, $\alpha_t(\nu) > 0$, and since $\nu$ is tangent to $\partial B$ at $p$, no plane field $\xi_t$ is tangent to $\partial B$ at $p$. Incidentally, we see that the vector field $\nu$ is positively transverse to both $\xi_t$ and the $z$-levels.

We finally need to perform a $\CC^0$-small perturbation of the $\xi_t$'s so that they become integrable near $\partial B$, and hence almost horizontal on $B$. Fix $\eps$ small enough that the plane fields $\xi_t$ are integrable in a $2 \eps$-neighbourhood of the poles, denote by $\nu$ the vector field described above and extend it to a vector field transverse to the $\xi_t$'s in a neighbourhood of $B$. Let $S$ be the surface obtained from $\partial B$ by removing an $\eps$-neighbourhood of the poles. We can parametrize a collar neighbourhood $W = S \times \D^1$ of $S$ by $\cercle^1 \times \D^1 \times \D^1$ so that $S = \cercle^1 \times \D^1 \times \{0\}$ and that every curve $\{\cdot\} \times \D^1 \times \{\cdot\}$ is an orbit segment of $\nu$. Since $\nu$ is transverse to $\xi_t$, there exists a unique vector field $\eta_t$ on $W$ tangent to $\xi_t$ and to each rectangle $\{\cdot\} \times \D^1 \times \D^1$ and whose last component is $1$. Now define $\bar\xi_t$ to be a $\CC^0$-small perturbation of $\xi_t$ with the following properties (see the appendix for similar constructions):
\begin{itemize}
\item $\bar\xi_t$ coincides with $\xi_t$ along $S$ and outside $S \times (-\delta, \delta) \subset W$ with $\delta$ arbitrarily small;
\item $\bar\xi_t$ contains $\eta_t$ at every point of $W$ ;
\item $\bar\xi_t$ is invariant under $\eta_t$ near $S$ and thus integrable there.
\end{itemize}
Note that on every region of the type $S' \times (-\delta', \delta') \subset S \times \D^1$ where $\xi_t$ is integrable, $\bar\xi_t$ is equal to $\xi_t$. This shows in particular that $\bar\xi_i = \xi_i$ for $i = 0, 1$ and that $\bar\xi_t = \xi_t$ near $\partial S \times \D^1$ for all $t$. 
\end{proof}

\begin{proof}[Proof of Lemma \ref{l:p-h}] It is enough to consider the case where $B$ consists of a unique ball. Let
$(\xi_t)_{t \in [0,1]}$ be a path of plane fields from $\xi_0$ to $\xi_1$. Using Claim \ref{a:p-h}, we assume that, in some adapted parametrization $\psi_i :\D^3 \to B$, the equation of $\xi_i$, $i = 0, 1$, is $dz = 0$. Since the group of diffeomorphisms of $\D^3$ is connected (according to a theorem of J.~Cerf \cite{Ce}), there exists an isotopy of $B$ between $\id$ and  $\psi_1 \circ \psi_0^{-1}$. Deforming $\xi_0$ by an extension of this isotopy to $M$ (among plane fields which are obviously almost horizontal on $B$), we reduce to the case where $\xi_0$ coincides with $\xi_1$ on $B$ and $\psi_0 = \psi_1 =\psi$. 

One can reparametrize the path $(\xi_t)_t$ so that $\xi_t$ coincides with $\xi_0$ for $t \in [0,1/3]$ and with $\xi_1$ for $t \in
[2/3,1]$,  and deform it slightly near $\psi(0)$, keeping $\xi_0$ and $\xi_1$ unchanged, so that each $\psi^*\xi_t$ is constant on the euclidean ball of radius $\eps$ centered at $0$. 

We then define the following family of balls: 
\begin{itemize}
\item $B_t$ is the image under $\psi$ of the euclidean ball centered at $0$ and of radius $3 (\eps - 1) t + 1$ for $t \in [0,1/3]$ ;
\item $B_t = B_{1/3}$ for all $t \in [1/3,2/3]$ ;
\item $B_t = B_{1-t}$ for $t \in [2/3,1]$.
\end{itemize}
By construction, $\xi_t$ is almost horizontal on $B_t$ for all $t\in[0,1]$. Let $\phi_t$ be an isotopy supported in a neighbourhood of $B$ satisfying $\phi_0 = \phi_1 = \id$ and $\phi_t(B) = B_t$ for all $t$. Then $(\phi_t^*\xi_t)$ is a path of $B$-almost horizontal plane fields homotopic to $(\xi_t)_{t\in[0,1]}$ with fixed end points. 
\end{proof}

\begin{proof}[Proof of Lemma \ref{l:desintegration}]
Let us start with the nonparametric version of Lemma \ref{l:desintegration} (referred to as \emph{Thurston's trick} in the Introduction), \emph{i.e.} how to homotope a single foliation $\xi$ on $N$ (having no component of $\partial N$ as a leaf) rel. $\Op(\partial N)$ to a taut $B'$-almost integrable plane field, for some collection of balls $B'\subset\Int(N)$. One can find finitely many disjoint arcs in $\Int(N)$ transverse to $\xi$ whose union meets every leaf of $\xi$. Let $A$ be one of them. We can assume that $\xi$ is tangent to $\D^2 \times \{\cdot\}$ in a neighbourhood $C = \D^2 \times \D^1$ of $A$, where $A = \{0\} \times [-1/2,1/2]$. Let $D_+$ and $D_-$ be two small disks in $\D^2$ and $\P = \D^2 \setminus \Int(D_+ \cup D_-)$. 
\begin{figure}[htbp]
 \centering
\begin{tabular}{ccc}
\includegraphics[height=3cm]{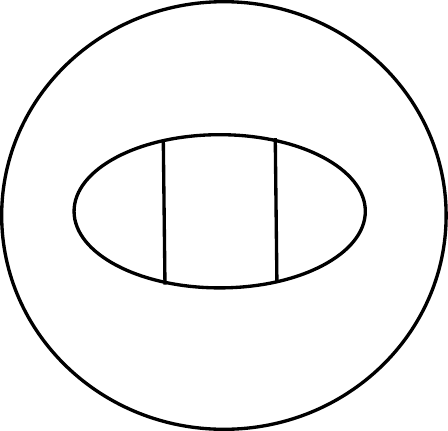} 
\put(-82,77){$\scriptstyle \D^2$}
\put(-70,40){$\scriptstyle D_-$}
\put(-30,40){$\scriptstyle D_+$}
\put(-49,40){$\scriptstyle R$}
&  & \includegraphics[height=3cm]{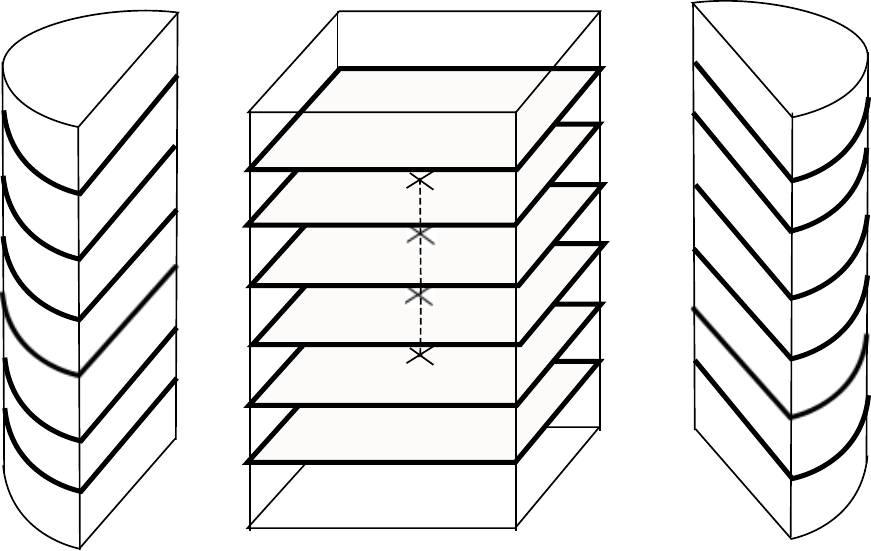} 
\put(-125,90){$\scriptstyle D_-$}
\put(-20,90){$\scriptstyle D_+$}
\put(-70,90){$\scriptstyle R$}
\put(-67,42){$\scriptstyle A$}
\end{tabular}
\caption{Schematic view of the deformation (I)\label{f:trick0}}
\small The initial (trivial) foliation $\xi=\xi^0$ near an arc $A$
\end{figure}
Let us first describe the deformation of $\xi$ on $\P \times \D^1 \subset C$. Let $f_u$, ${u \in [0,1]}$, be a path of diffeomorphisms of $\D^1$ coinciding with the identity near the boundary, such that $f_0 = \id$ and $f_u(x) > x$ for all $x \in
[-1/2,1/2]$ and all $u > 0$.  Using an analogue of Lemma \ref{l:pantalon}, we construct a deformation $u\in[0,1]\mapsto \xi^u \res{\P\times \D^1}$ such that:
\begin{itemize}
\item $\xi^0 \res{\P \times \D^1} = \xi \res{\P \times \D^1}$;
\item for every $u \in [0,1]$, the foliation $\xi^u \res{\P \times \D^1}$ is transverse to the $\D^1$ factor, coincides with $\xi$ near $\partial (\D^2 \times \D^1) \cap \P \times \D^1$, and induces foliations of holonomy $f_{u}$, $f_{u}^{-1}$ and $\id$ on $\partial D_+ \times \D^1$, $\partial D_- \times \D^1$ and $\partial \D^2 \times \D^1$ respectively.
\end{itemize}
\begin{figure}[htbp]
 \centering
\includegraphics[height=3cm]{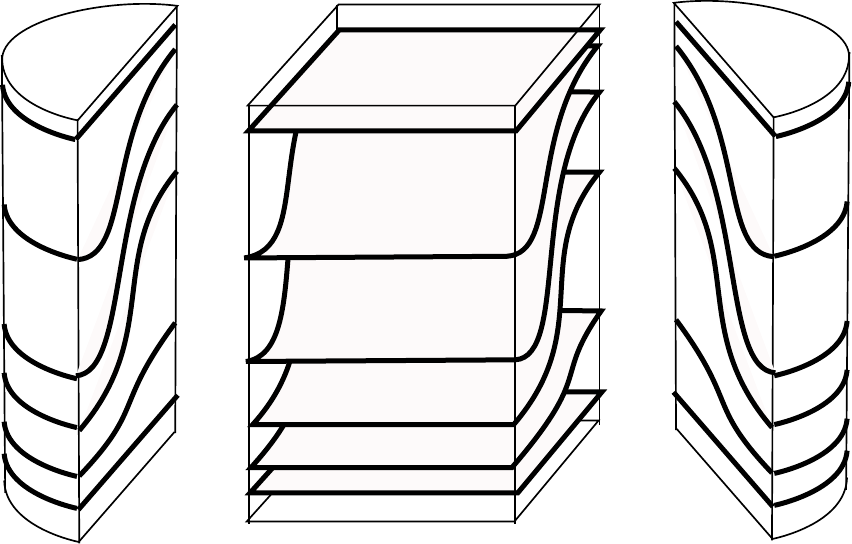} 
\caption{Schematic view of the deformation (II)\label{f:trick1}}
\small
The final plane field $\xi^1$ near $A$
\end{figure}
One easily extends each $\xi^u \res {\P \times \D^1}$ (continuously with respect to $u \in [0,1]$) to a \emph{plane field} on $C = \D^2 \times \D^1$ in such a way that $\xi^0 = \xi$. But for $u\neq0$, $\xi^u$ cannot be integrable on 
 $D_\pm \times \D^1$ for we have made its holonomy nontrivial on the lateral boundary. Denote by $B_\pm$ a ball obtained by rounding the corners of $D_\pm \times \D^1$.

Let $\bar\xi$ denote the plane field on $N$ obtained by carrying out the above perturbation $u\mapsto \xi^u$ in a neighbourhood of every transverse arc $A$, and $B'$ the collection of balls $B_\pm$. Then $\bar\xi$ is integrable on $N
\setminus \Int B'$, almost horizontal on $B'$, and every leaf of the foliation defined by $\bar\xi$ on $N \setminus \Int B'$ meets the boundary of some ball $B_\pm$ in the ``subtropical" region where the induced foliation spirals up or down (from $\mp1/2$ to $\pm1/2 \in \D^1$). In particular, it is noncompact and meets a closed transversal: the ``equator" of $\partial B_\pm$. So $\bar\xi$ is taut $B'$-almost integrable.
\begin{figure}[htbp]
 \centering
\includegraphics[height=2.5cm]{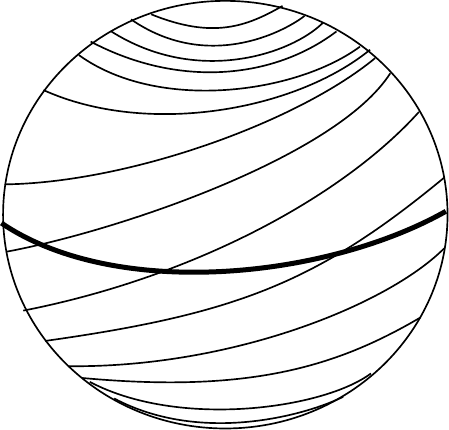} 
\caption{$B_-$ and its equator\label{f:B+}}
\end{figure}

Let us now turn to the proof of Lemma \ref{l:desintegration} itself, \emph{i.e.} let us now give a relative one-parameter version of the above. For every $t \in [0,1]$, we can find finitely many disjoint arcs in $\Int(N)$ transverse to $\xi_t$ whose union meets every leaf of $\xi_t$. Extending them slightly if necessary, we can assume they have the same property with respect to
$\xi_s$ for all $s$ close enough to $t$. So we can cover $[0,1]$ by the interiors (as subsets of $[0,1]$) of finitely many segments $J_k$, $1\le k \le m$, such that, for each $k$, there exists a collection of arcs $\{A^i(k)\}_{i}$ in $\Int N$ transverse to $\xi_t$ whose union meets every leaf of $\xi_t$ for all $t$ in $\Op (J_k)$. Up to a slight perturbation of $(\xi_t)_{t\in[0,1]}$, we can actually assume that every $\xi_t$, $t\in \Op (J_k)$, is tangent to $\D^2\times\{.\}$  in a neighbourhood $C^i(k) = \D^2 \times \D^1$ of $A^i(k)$, where $A^i(k) = \{0\} \times [-1/2,1/2]$. Moreover, these neighbourhoods $C^i(k)$ can be taken disjoint.

Now for every $k$, we apply the nonparametric trick to every $\xi_t$, $t\in J_k$, on each $C^i(k)$ (leaving $\xi_t$ unchanged for $t\notin J_k$), using the path of diffeomorphisms $u\in[0,1]\mapsto f_{\rho_k(t) u}$ instead of $u\mapsto f_u$, where $\rho_k$ denotes a bump function on $J_k$ vanishing only on $\partial J_k$. This can be done continuously with respect to $t$, leaving $\xi_0$ and $\xi_1$ unchanged while making every other $\xi_t$ taut oustide balls (two in each $C^i(k)$ such that $t\in J_k$, which will be denoted by $B^i_+(k)$ and $B^i_-(k)$). 

Note that at that point the collection of balls for a given $\xi_t$ depends on $t$. But the same (big) collection would work for every plane field if $\xi_t$ was almost horizontal on $B^i_\pm(k)$ even when $t\notin J_k$. This can be arranged by pulling back $\xi_t$ by a diffeomorphism $\phi_t$ (depending continuously on $t$ and equal to $\id$ for $t\in J_k$) supported in a neighbourhood of $B^i_\pm(k)$ and sending $B^i_\pm(k)$ to a sufficiently small neighbourhood of its ``center'' for $t$ away from $J_k$ (just like in the previous proof). 
\end{proof}

\begin{remark}\label{r:multi} We stated and proved Lemma \ref{l:desintegration} for a one-parameter family because this is what we needed at the moment, but the proof adapts without problem to any number of parameters. 
\end{remark}


\subsection{Deforming families of taut almost integrable plane fields}\label{ss:CVL}

\begin{proposition}\label{p:third} Let $M$ be a closed $3$-manifold, $B\subset M$ a collection of balls, $(\xi_t)_{t\in[0,1]}$ a path of taut $B$-almost integrable plane fields on $M$ and $\tau_0$ and $\tau_1$ malleable foliations obtained from $\xi_0$ and $\xi_1$ by Lemma \ref{l:correspondance}. Then the corresponding deformations from $\xi_i$ to $\tau_i$, $i=0,1$, extend to a deformation of $(\xi_t)_{t\in[0,1]}$ to a path $(\tau_t)_{t\in[0,1]}$ of malleable foliations connecting $\tau_0$ to $\tau_1$. 
\end{proposition}

This will follow from the study of two particular cases: the one, settled in Remark \ref{r:correspondance-param}, where one is given families of transverse arcs connecting the poles of the balls varying continuously with the parameter, and the cases that $\xi_t$ does not depend on $t$, which can be rephrased as follows:

\begin{lemma}[Siphon Lemma]\label{l:v-c}
Let $M$ be a closed $3$-manifold, $B = \bigcup_1^n B_i \subset M$ a collection of balls, $\xi$ a taut $B$-almost integrable
plane field on $M$ and $A^\pm = \bigcup A^\pm_i \subset M \setminus \Int B$ two collections of arcs transverse to $\xi$, each $A^\pm_i$ connecting the poles of $B_i$. The malleable foliations $\tau^\pm$ built from $\xi$ and $A^\pm$ can be connected by a path of malleable foliations. What's more, the loop of plane fields formed by the homotopies from $\xi$ to $\tau^-$, $\tau^-$ to $\tau^+$, and $\tau^+$ to $\xi$ bounds a disk of plane fields on $M$.
\end{lemma}
\begin{figure}[h!]
\centering
\includegraphics[height=4cm]{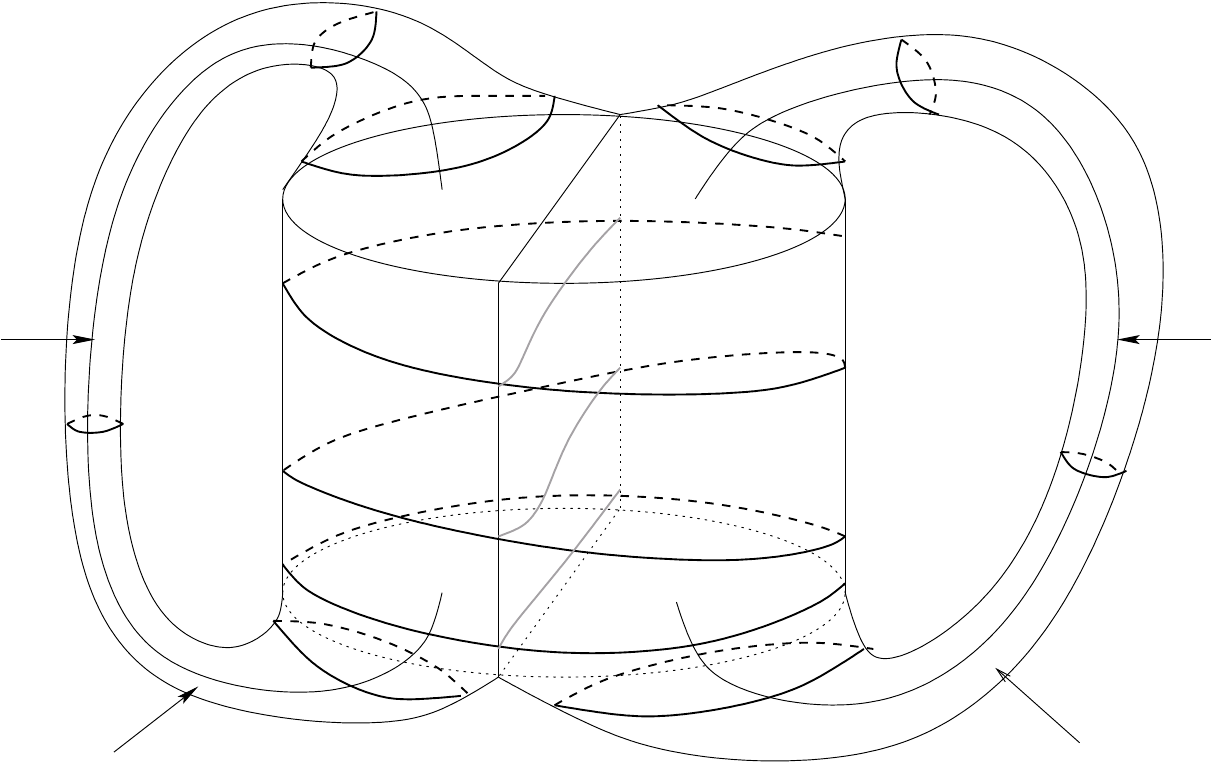}
\put(-140,118){$\scriptstyle W_-$}
\put(-193,60){$\scriptstyle \bar A_-$}
\put(-120,65){$\scriptstyle C_-$}
\put(-40,115){$\scriptstyle W_+$}
\put(0,60){$\scriptstyle \bar A_+$}
\put(-80,65){$\scriptstyle C_+$}
\put(-17,-2){$\scriptstyle \Lf$}
\put(-170,-6){$\scriptstyle \Lid$}
\caption{The foliation $\tau_1$ on $W$}
\label{fig:vase}
\end{figure}
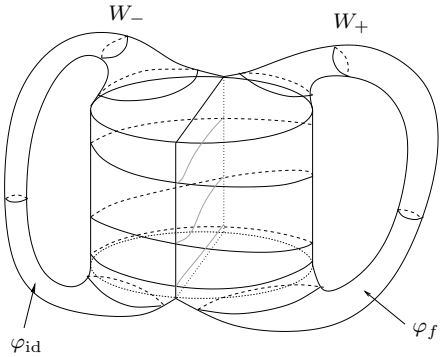
\begin{proof}
For simplicity, let us assume $n=1$, so that $B$ is a single ball, and $A^\pm$ are two arcs transverse to $\xi$ connecting the poles of $B$. The main features of the following set up are depicted in Figure \ref{fig:vase}. Parametrize $B$ minus two small polar caps by $\D^2 \times [-1/4,1/4]$ in such a way that $\xi$ is transverse to the second factor and tangent to $\D^2 \times \{\cdot\}$ in a neighbourhood of $\D^2 \times \{\pm1/4\}$. Deform the arcs $A^\pm$ slightly into disjoint arcs $\bar A^\pm$ connecting $(\pm 1/2,0,1/4)$ to $(\pm1/2,0,-1/4) \in C = \D^2 \times [-1/4,1/4]$ respectively, transversely to $\xi$ in 
$\overline{M\setminus C}$. Let
$$D^\pm = \D^2\cap \{\pm x\ge 0\},$$
where $x$ denotes the first coordinate on $\D^2$, 
$$C^\pm = C \cap \{\pm x\ge 0\} = D^\pm \times [-1/4,1/4] \subset C$$
and let $W^\pm$ be solid tori obtained by smoothing the union of $C^\pm$ with a neighbourhood of $\bar A^\pm$ trivially foliated by $\xi$, such that
$$W^+ \cap W^- = C \cap \{x=0\} = C^+ \cap C^-.$$
Parametrize $W^\pm$ by $\D^2 \times \cercle^1 = \D^2 \times \R / \Z$ so that
$$C^\pm = D^\pm \times [-1/4,1/4] \subset D^\pm \times \R/\Z = W^\pm.$$
Finally, denote by $g \in \Diff^\infty_+([-1/4,1/4])$ the holonomy, for the base point $(0,-1) \in \partial \D^2$, of the foliation induced by $\xi$ on the lateral boundary of $C$, $\bar f$ its extension by the identity to a diffeomorphism of $\cercle^1 = \R/\Z$, and $f \in \DDi$ the lift of $\bar f$ fixing $\pm 1/4$.

Now let $C'$ be a slight shrinking of $C$ so that $\xi$ is integrable on $C\setminus C'$, and let $(\psi_t)_{t\in[-1,1]}$ be a continuous path of diffeomorphisms of $M$ supported in $C=\D^2 \times [-1/4,1/4]$, leaving the last coordinate unchanged and such that $\psi_0=\id$ and $\psi_{\pm 1}(C')\subset C^\pm$. Define $\xi_t=(\psi_t)_*\xi$ for all $t\in[-1,1]$, $\xi_t=\xi_{-1}$ for all $t\in[-2,-1]$ and $\xi_t=\xi_1$ for all $ t\in [1,2]$. In particular, $\xi_{\pm 1}$ induces a  foliation by disks on $W^\mp$ and consequently a foliation of holonomy $f$ on $\partial W^\pm = \partial D^\pm \times \cercle^1$ (the base point being $(0,-1) \in \partial D^\pm \subset \partial \D^2$). Let $\tau_{\pm2}$ be a foliation of $M$ coinciding with $\xi_{\pm 2}$ on $M\setminus W^\pm$ and with $\f_f$ on $W^\pm$. This foliation is clearly isotopic to the foliation $\tau^\pm$ built from $\xi$ and the transverse arc $A^\pm$ using the process of Lemma \ref{l:correspondance}. Hence, to prove Lemma \ref{l:v-c}, it suffices to prove that $(\xi_t\res W)_{t\in[-2,2]}$, can be deformed (rel. boundary) to a continuous path of malleable foliations $(\tau_t\res W)$ connecting $\tau_{-2}\res W$ to $\tau_2\res W$, the deformation from $\xi_{\pm 2}$ to $\tau_{\pm 2}$ being the one of Lemma \ref{l:correspondance} (outside $W$, one simply takes $\tau_t=\xi$). Let us just describe the path $\tau_t$, $t\in[-2,2]$, since the existence of the homotopy between $(\xi_t)$ and $(\tau_t)$ is a direct consequence of Proposition  \ref{p:third-param} and its proof. The deformation $\tau_t$, $t\in[-2,-1]$, consists in deforming the foliation by disks in $W^+$ to $\Lid$ (cf. Remark \ref{r:lid}). Now as $t$ goes from $-1$ to $1$, the holonomies of the foliations induced by $\xi_t$ on $\partial W^-$ and $\partial W^+$ vary from $f$ to $\id$ and from $\id$ to $f$ respectively. Define $\tau_t\res {W^\pm}$, $t\in[-1,1]$, to be the extensions of these foliations given by Proposition \ref{t:larcanche}, so that $\tau_1$ induces $\f_f$ on $W^+$ and $\Lid$ on $W^-$. Then $\tau_t$, $t\in[1,2]$, consists in deforming the foliation $\Lid$ to a foliation by disks in $W^-$, which gives the desired malleable foliation $\tau_2$. 
\end{proof}

\begin{proof}[Proof of Proposition \ref{p:third}]
Assume again, for simplicity, that $B$ is a single ball. To each plane field $\xi_t$ corresponds a particular parametrization of $B$ by the unit euclidean ball $\D^3$. Denote by $p_t$ and $q_t\in B$ the corresponding north and south poles. For every $t$, any arc on $\partial B$ transverse to $\xi_t$ joining the poles $p_t$ and $q_t$ can be extended in $M \setminus \Int B$ to a closed transversal, since the foliation defined by $\xi_t$ on $M\setminus \Int B$ is taut. Hence, we get a family $A_t$ of transverse
arcs connecting the poles of $B$ on the outside. To deal with the relative part of Proposition \ref{p:third}, assume $A_0$ and $A_1$ are prescribed, and denote by $\tau_0$ and $\tau_1$ the foliations obtained by applying Lemma \ref{l:correspondance} to $\xi_0$ and $\xi_1$. Now for all $s$ close enough to some given $t$, $A_t$ remains transverse to $\xi_s$, and we can slightly move its ends in a continuous way so that they coincide with $p_s$ and $q_s$ for all $s$. On every interval $[\frac k m,\frac{k+1}{m}]$, $0 \le k \le m-1$, with $m$ sufficiently large, we thus have a continuous path $t \mapsto A_t(k) $ of
transverse arcs (with $A_0(0)=A_0$ and $A_1(m-1)=A_1$). According to the parametric version of Lemma \ref{l:correspondance} (cf. Remark \ref{r:correspondance-param}), we can deform $t\in[\frac k m, \frac{k+1}{m}] \mapsto
\xi_t$ to a continuous path $t \in [\frac k m, \frac{k+1}{m}] \mapsto \tau_t(k)$ (with $\tau_0(0)=\tau_0$ and $\tau_1(m-1)=\tau_1$). Combining this with the Siphon Lemma \ref{l:v-c} (applied to $\xi=\xi_{k/m}$,  $A^-=A_{k/m}(k-1)$ and $A^+=A_{k/m}(k)$), we get the desired homotopy from $(\xi_t)_{t\in[0,1]}$ to a path of malleable foliations joining $\tau_0$ and $\tau_1$. 
\end{proof}

We can now conclude with the proof of Theorem \ref{t:malleable} and a sketch of proof of Theorem~\ref{t:pik}.

\begin{proof}[Proof of Theorem \ref{t:malleable}] Let $(\zeta_t)_{t\in[0,1]}$ be a path of plane fields on a closed $3$-manifold $M$ connecting two malleable foliations $\tau_0$ and $\tau_1$, whose associated collections of tori are denoted by  $W_0$ and $W_1$. For $i=0,1$, Lemma \ref{l:inverse} associates to $\tau_i$ and $W_i$ a collection of balls $B_i$ and a taut $B_i$-almost integrable plane field $\xi_i$ homotopic to $\tau_i$. We can assume that $B_0$ and $B_1$ have the same number of balls, completing one or the other if necessary with some small balls $\D^3$ on which $\xi_i$ is horizontal. Deforming $\xi_0$ by an isotopy of $M$, we can then reduce to the case $B_0 = B_1 = B$. By the homotopy extension property, we can extend the deformation from $(\zeta_t)_{t\in\{0,1\}}$ to $(\xi_t)_{t\in\{0,1\}}$ to a deformation from $(\zeta_t)_{t\in[0,1]}$ to a path $(\xi_t)_{t \in [0,1]}$ of plane fields connecting $\xi_0$ to $\xi_1$.  According to Proposition \ref{p:p-m}, this path can be deformed with fixed endpoints to a path $(\widehat\xi_t)_{t \in [0,1]}$, of taut $\Bh$-almost integrable plane fields, for some collection of balls $\Bh = \bigcup_{j=1}^n \Bh_j$ containing $B$. Now according to Proposition \ref{p:third}, this path can in turn be deformed to a path $(\tau_t)_{t \in [0,1]}$ of malleable foliations, the deformation from $\widehat\xi_i=\xi_i$ to $\tau_i$, $i=0,1$, being the reverse of the one from $\tau_i$ to $\xi_i$ considered earlier. It is then easy to combine the above homotopies from $(\zeta_t)_{t\in[0,1]}$ to $(\xi_t)_{t\in[0,1]}$, from $(\xi_t)_{t\in[0,1]}$ to $(\widehat\xi_t)_{t\in[0,1]}$ and from $(\widehat\xi_t)_{t\in[0,1]}$ to $(\tau_t)_{t\in[0,1]}$ to obtain a homotopy from $(\zeta_t)_{t\in[0,1]}$ to $(\tau_t)_{t \in [0,1]}$ \emph{with fixed endpoints}. 
\end{proof}

\begin{proof}[Sketch of proof of Theorem \ref{t:pik}] Let $\tau$ be a malleable foliation on a closed  $3$-manifold $M$ and $(\zeta_t)_{t\in \D^{k}}$ a continuous family of plane fields such that $\zeta_{t}=\tau$ for all $t\in \partial \D^{k}$. We want to deform this family, relative to $\partial  \D^{k}$, to a family of malleable foliations. 

First, let $\xi$ be a taut $B$-almost integrable plane field obtained from $\tau$ by Lemma \ref{l:inverse} and consider a new family of plane fields $(\xi_t)_{t\in \D^{k}}$ defined by $\xi_t=\tau_{2t}$ for $\norm{t}\le \frac12$ and such that, for every $t\in \partial \D^k$, the path $r\in[\frac12,1]\mapsto \xi_{rt}$ is the one from $\tau$ to $\xi$ given by Lemma~\ref{l:inverse}. We apply the following analogue of Proposition \ref{p:p-m} to $(\xi_t)_{t\in \D^{k}}$:

\begin{proposition}
Let $\xi$ be a taut $B$-almost integrable plane field on a closed $3$-manifold $M$, for some collection of balls $B\subset M$, and $(\xi_t)_{t\in \D^{k}}$ a continuous family of plane fields such that $\xi_{t}=\xi$ for all $t\in \partial \D^{k}$. Then $(\xi_t)_{t\in \D^{k}}$ is homotopic, relative to $\partial  \D^{k}$, to a family of taut $\Bh$-almost integrable plane fields, where $\Bh$ is a collection of balls \including$B$.
\end{proposition}

Like Proposition \ref{p:p-m}, this is proved in three steps.
First, like in Lemma \ref{l:p-h}, we make all the plane fields almost horizontal on $B$. Then we apply Proposition \ref{param} (the generalization of Proposition \ref{p:p-i} to which the appendix is devoted) to make them all $\Bb$-almost integrable for some collection $\Bb$ containing $B$. Finally we apply a multiparametric version of Lemma \ref{l:desintegration} (cf. Remark \ref{r:multi}) to make them all taut and $\Bh$-almost integrable for some collection $\Bh$ containing $\Bb$. All of this is done relative to $\partial \D^k$. We denote by $(\widehat\xi_t)_{t\in \D^{k}}$ the resulting family of plane fields.

We then apply the following analogue of Proposition \ref{p:third}:

\begin{proposition}
Let $(\widehat\xi_t)_{t\in \D^{k}}$ be a family of taut $\Bh$-almost integrable plane fields on $M$ such that $\xi_t=\xi$ for all $t\in\partial \D^k$,
 and let $\tau$ be a malleable foliation obtained from $\xi$ by Lemma \ref{l:correspondance}. Then the deformation from $\widehat\xi_t=\xi$ to $\tau_t=\tau$, $t\in\partial \D^k$, given by Lemma \ref{l:correspondance}, extends to a deformation from $(\widehat\xi_t)_{t\in \D^{k}}$ to a family $(\tau_t)_{t\in \D^{k}}$ of malleable foliations. 
\end{proposition}

Collapsing $\partial \D^k$ to a point, one can equivalently start with a family $(\widehat\xi_t)_{t\in \sphere^{k}}$ of taut $\Bh$-almost integrable plane fields on $M$ such that $\xi_{t_0}=\xi$ for some $t_0\in \sphere^k$ and extend the deformation from $\xi_{t_0}=\xi$ to $\tau_{t_0}=\tau$ to all of $\sphere^k$. 

The proof generalizes that of Proposition \ref{p:third}. Assume again, for simplicity, that $\Bh$ is made of a single ball, and denote by $A$ the arc used to turn $\xi$ into $\tau$ in Lemma \ref{l:correspondance}. We then triangulate the parameter space $\sphere^k$  finely enough so that for any vertex $v$ of this triangulation, there exists a continuous family of arcs $t\in\mathrm{Star}(v)\mapsto A_t(v)$ such that $A_t(v)$ is transverse to $\xi_t$ and connects the poles of $\Bh$ for this plane field for every $t\in\mathrm{Star}(v)$. We impose in addition that $t_0$ is a vertex and that $A_{t_0}(t_0)= A$. 

We first define the homotopy from $\widehat\xi_v$ to $\tau_v$, for every vertex $v$, as the one given by Lemma \ref{l:correspondance} applied to the arc $A_v(v)$. Then as in the Siphon Lemma, this extends to a homotopy from $(\xi_t)_{t\in\sphere^k}$ to a family of malleable foliations $(\tau_t)_{t\in\sphere^k}$. More precisely, the Siphon Lemma itself provides an extension to the one-skeleton, and the next skeleta are dealt with similarly. As an example, if $t$ is the center of some $j$-simplex $\sigma$, the foliation $\tau_t$ is obtained from $\widehat\xi_t$ by digging out $j+1$ worm holes connecting the top and bottom of $\Bh$ (along the $j+1$ arcs associated to the $j+1$ vertices of $\sigma$) and then by ``equally distributing'' the holonomy of $\widehat\xi_t$ along $\partial \Bh$ in the $j+1$ solid tori obtained as a union of a worm hole with a cylinder in $\Bh$. .\medskip

To conclude the proof of Theorem \ref{t:pik}, we combine the above homotopies from $(\zeta_t)_{t\in \D^{k}}$ to $(\xi_t)_{t\in \D^{k}}$, from $(\xi_t)_{t\in \D^{k}}$ to $(\widehat\xi_t)_{t\in \D^{k}}$ and from $(\widehat\xi_t)_{t\in \D^{k}}$ to $(\tau_t)_{t\in \D^{k}}$ to obtain a homotopy from $(\zeta_t)_{t\in \D^{k}}$ to $(\tau_t)_{t\in \D^{k}}$ \emph{relative to $\partial\D^k$}. 
\end{proof}

\section{Malleabilization of neat foliations}\label{s:malleabilization}

This section is devoted to the proof of Theorem \ref{t:malleabilization}, that is that every neat foliation can be connected to a malleable one by a path of neat foliations. We have not defined neat foliations yet, we have only said that they were described by a simple model near their Novikov tori. Let us now describe this model.

\begin{definition}\label{d:model}
Given $\eps > 0$, we call any foliation on $\T^2 \times [-\eps,\eps]$ defined by an equation of the form:
\begin{align*}
\begin{cases}
dz - u(z)(a^+ dx_1 + b^+ dx_2) ,\quad (x_1,x_2,z) \in \T^2 \times [0,\eps]\\
dz - u(z)(a^- dx_1 + b^- dx_2), \quad (x_1,x_2,z) \in \T^2 \times [-\eps,0]
\end{cases}
\end{align*}
where $(a^{\pm},b^{\pm}) \in \R^2 \setminus \{(0,0)\}$ and $u$ is a smooth
function vanishing only at $0$, a \emph{model foliation}.
\end{definition}
\begin{figure}[htbp]
\centering
\includegraphics[height=4cm]{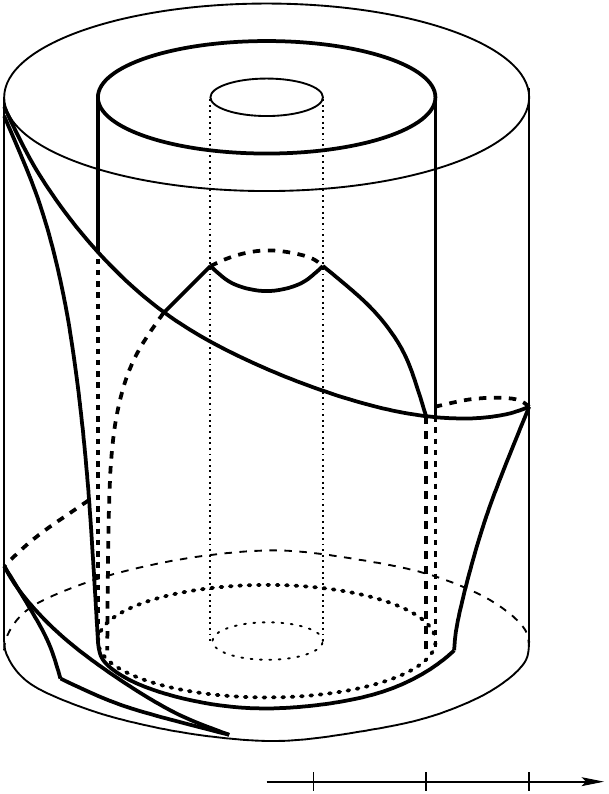}
\put(-28,-10){$0$}
\put(-50,-10){$-\eps$}
\put(-13,-10){$\eps$}
\caption{Model foliation near a torus leaf, with $(a^-,b^-)=(0,1)$ and $(a^+,b^+)=(4,5)$}
\label{net}
\end{figure}

\begin{remark} 
If $u$ is not infinitely flat at $0$ in the statement above, then $(a^+,b^+) = (a^-,b^-)$ since the equation is supposed to be $\Cinf$. If $u$ is infinitely flat at $0$ however, the vectors $(a^+,b^+)$ and $(a^-,b^-)$ may differ.
\end{remark}

\begin{definition}\label{d:net}
Let $\tau$ be a smooth foliation on a closed $3$-manifold. A torus leaf $T$ of $\tau$ is \emph{neat} if there is a parametrized neighbourhood $N \simeq \T^2 \times [-\eps,\eps]$ of $T$ on which $\tau$ induces a model foliation.

The foliation itself is \emph{neat} if all its Novikov tori are neat.
\end{definition}

\begin{remark}
A neat foliation has finitely many Novikov tori.
\end{remark}

The definitions of \emph{malleable} and \emph{neat} extend in a natural way to foliations of a compact manifold transverse to the boundary. Theorem \ref{t:malleabilization} is then a direct consequence of the following local deformation result:

\begin{proposition}\label{p:mall-local}
Every model foliation on $\T^2 \times [-1,1]$ can be deformed  to a malleable foliation through neat foliations relative to the boundary.
\end{proposition}

The proof consists of two steps: first we kill the initial torus leaf $\T^2 \times \{0\}$, creating one or two new ones lying in Reeb foliations (cf. Lemma \ref{l:rolling}). This uses the form of the foliation near a neat leaf in a fundamental way. Then, we replace these Reeb foliations by a collection of ``parallel" simple Schweitzer foliations (cf. Lemma \ref{l:frag-fol}).


\subsection{Rolling up a torus leaf}\label{ss:rolling}

\begin{lemma}\label{l:rolling}
Every model foliation on $\T^2 \times [-1,1]$ can be deformed through neat foliations and relative to the boundary to a foliation taut outside one or two solid tori foliated by Reeb fillings.
\end{lemma}

\begin{proof} \textbf{Easy case}. There is a case in which getting rid of the toric leaf of a model foliation by a deformation of foliations \emph{without adding any new toric leaf} is easy: this is when the equation of the foliation is of the form:
$$dz - u(z)(a dx_1 + b dx_2) ,\quad (x_1,x_2,z) \in \T^2 \times [-1,1]$$
with $(a,b) \in\R^2 \setminus \{0\}$ and $u$ a smooth function vanishing only at $0$ and having the same sign on both sides of $0$. Then simply take a small deformation $u_t$, $t\in[0,1]$, of $u=u_0$ such that, for all $t>0$, $u_t$ is a smooth non-vanishing function coinciding with $u$ outside a small neighbourhood of $0$. Then the equations
$$dz - u_t(z)(a dx_1 + b dx_2) ,\quad (x_1,x_2,z) \in \T^2 \times [-1,1]$$
define foliations $\tau_t$ which, for $t > 0$, have no torus leaves anymore.

The idea in the general case is to reduce to this easy case by changing the ``slope" of the model foliation on one side of the central torus leaf, having beforehand inserted a Reeb filling to serve as a ``siphon" for the excess (or lack) of slope. \medskip

\noindent\textbf{Set up: choice of nice coordinates}. Let $\tau$ be a model foliation on $\T^2 \times [-1,1]$, defined by the equations:
\begin{align*}
\begin{cases}
dz - u(z)(a^+ dx_1 + b^+ dx_2) ,\quad (x_1,x_2,z) \in \T^2 \times [0,1]\\
dz - u(z)(a^- dx_1 + b^- dx_2), \quad (x_1,x_2,z) \in \T^2 \times [-1,0]
\end{cases}
\end{align*}
where $(a^{\pm},b^{\pm}) \in \R^2 \setminus \{(0,0)\}$ and $u$ is a smooth function vanishing only at $0$. Up to a deformation of $u$ near $0$ (and thus of $\tau$ near $T = \T^2 \times \{0\}$), we can assume that this function is infinitely
flat at $0$. Furthermore, up to a linear change of coordinates on $\T^2$, we can assume that $a^+$ and $a^-$ are different from $0$. The function $v$ equal to $|a^+ u|$ on $[0,1]$ and $|a^- u|$ on $[-1,0]$ is smooth, so $\tau$ is actually described by an equation of the form $dz - v(z)(a^\pm dx_1 + b^\pm dx_2)$ with $v$ smooth, nonnegative, vanishing only at $0$ and $a^\pm \in \{-1,1\}$. We distinguish two cases, depending on whether $(a^+,b^+)$ differs or not from $-(a^-,b^-)$. The second case reduces to the first one by a continuous deformation of $\tau$ which consists in splitting the torus leaf $\T^2 \times \{0\}$ into two, $\T^2 \times \{\pm \eps\}$, inserting a neat foliation of the form $dz-w(z)(a' dx_1 + b' dx_2)$ in the middle, with $w$ smooth and vanishing only at $\pm\eps$, and $(a',b') \neq \pm(a^+,b^+)$.

In the first case, there exists an integer vector of $\Z^2$ which forms a direct basis both with $(a^+,b^+)$ and $(a^-,b^-)$. In other words, up to a linear change of coordinates, we can assume $a^\pm > 0$. So replacing $v$ by $a^+v$ on $[0,1]$ and $a^-v$ on $[-1,0]$, (which leaves $v$ smooth and positive outside $0$), we can assume that $\tau$ has an equation of the form:
\begin{align*}
\begin{cases}
dz - v(z)(dx_1 + b^+ dx_2) ,\quad (x_1,x_2,z) \in \T^2 \times [0,1]\\
dz - v(z)(dx_1 + b^- dx_2), \quad (x_1,x_2,z) \in \T^2 \times [-1,0].
\end{cases}
\end{align*}

\noindent\textbf{Deformation on $\T^2 \times [-1,1/2]$}. Recall we want to reduce to the ``easy case" above, \emph{i.e} to the case $b^+=b^-$. So let us consider a continuous path $b_t^+$, $t\in[0,1]$, between $b_0^+=b^+$ and $b_1^+=b^-$. For later purposes, let us also require $b_t^+$ to be equal to $b^+$ for all $t \in [0,1/2]$. Now on $\T^2 \times [-1,1/2]$, we consider the following path of foliations $\tau_t$, $t\in[0,1]$:
\begin{align*}
\begin{cases}
dz - v(z)(dx_1 + b^- dx_2) ,\quad (x_1,x_2,z) \in \T^2 \times [-1,0]\\
dz - v(z)(dx_1 + b^+_t dx_2), \quad (x_1,x_2,z) \in \T^2 \times [0,1/2].
\end{cases}
\end{align*}
For $t = 1$ we are in the situation of the ``easy case" and the central (unique) torus leaf can be removed by a deformation of foliations relative to the boundary.
\begin{figure}[htb]
\centering
\begin{tabular}{ccc}
\includegraphics[height=3.3cm]{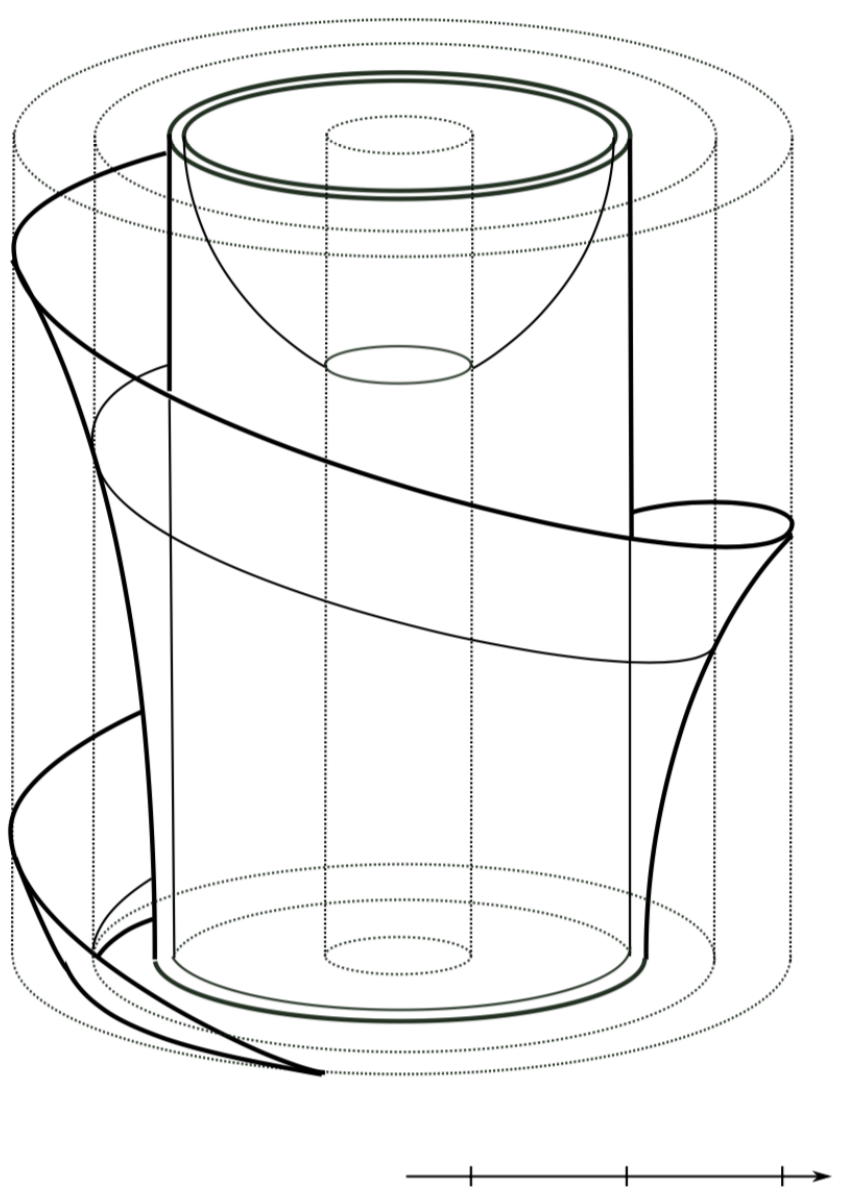} 
\put(-37,-12){$\scriptstyle -1$}
\put(-18,-12){$\scriptstyle 0$}
\put(-6,-12){$\scriptstyle 1$}
\hspace{0.5cm}
& \includegraphics[height=3.3cm]{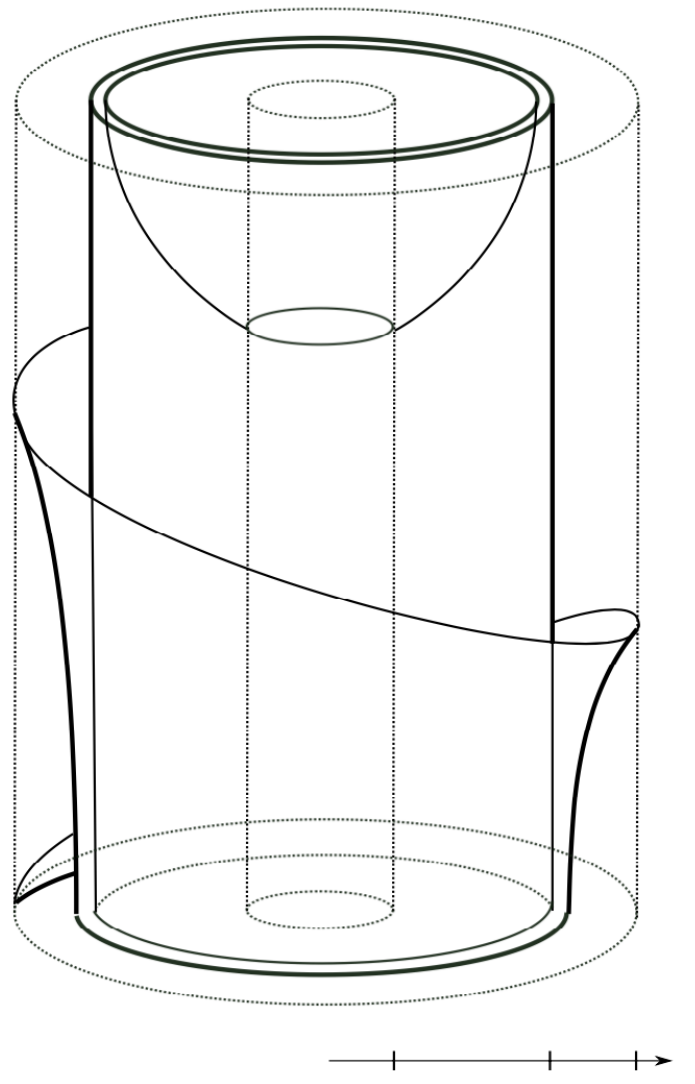}
\put(-32,-12){$\scriptstyle -1$}
\put(-13,-12){$\scriptstyle 0$}
\put(-6,-12){${\scriptscriptstyle \tfrac12}$}
\hspace{0.5cm}
& \includegraphics[height=3.3cm]{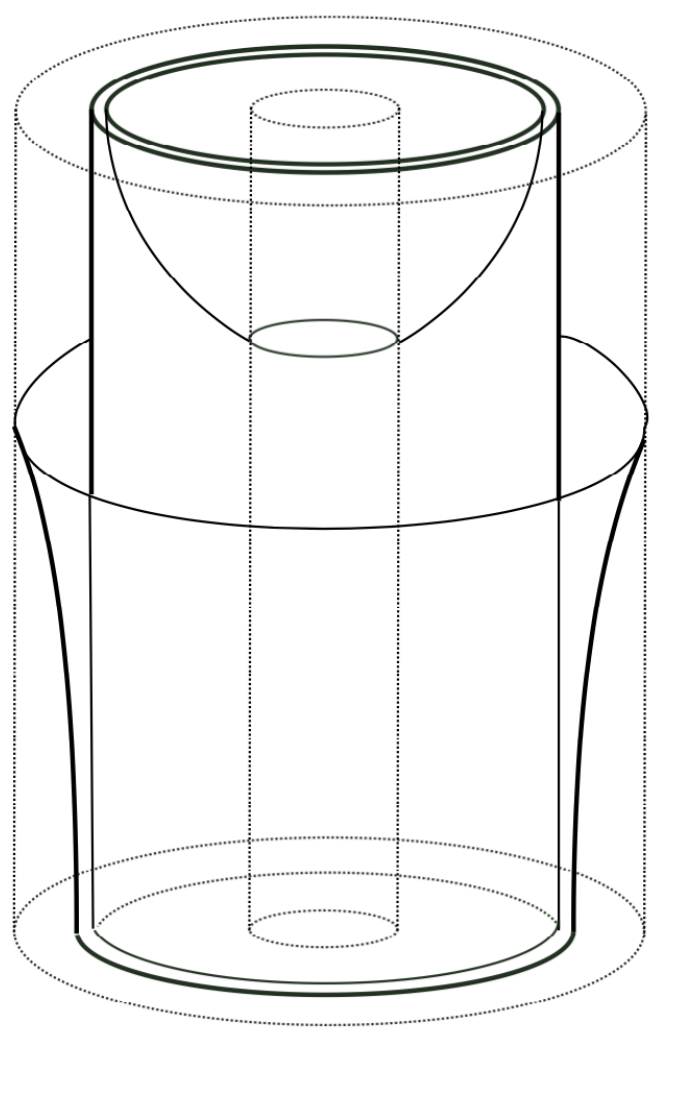}
\end{tabular}
\caption{Deformation on $\T^2\times[-1,1/2]$
\label{pente}}
\end{figure}
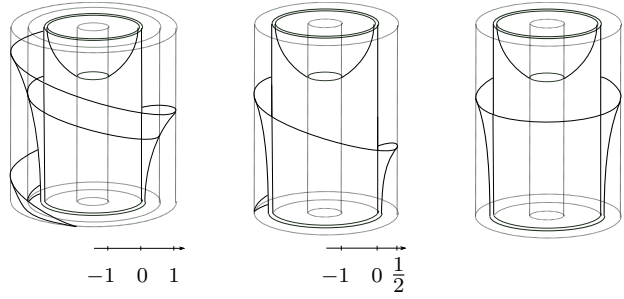

\medskip

\noindent\textbf{Deformation on $\T^2 \times [1/2,1]$}. We now need to extend the above deformation to $\T^2 \times [1/2,1]$ (relative to $\T^2 \times \{1\}$).

For $t\in[0,1/2]$, the foliation induced on $\partial(\T^2 \times [1/2,1])$ must remain unchanged. The foliation inside, however, is going to be modified by the insertion of a Reeb component along a transverse circle. More precisely, let $p$ be a point in the open annulus $A=\cercle^1\times(1/2,1)$. The circle $\cercle^1 \times \{p\} \subset \cercle^1 \times A = \T^2 \times (1/2,1)$ is transverse to the foliation, so for a small enough disk $D$ centered at $p$ in $A$, $\cercle^1 \times D$ is foliated by disks. According to Remark \ref{r:reeb}, there exists a deformation $\tau_t$, $t \in [0,1/2]$, of foliations on $\T^2 \times
[1/2,1]$ relative to the complement of $\cercle^1 \times D$ such that $\tau_0 = \tau$ and $\tau_{1/2}$ induces a Reeb filling of slope $0$ in $\cercle^1 \times D$.

Now we want to define a deformation $\tau_t$, $t \in [1/2,1]$, of foliations on $\T^2 \times [1/2,1]$ which induce linear foliations of equation $dx_1 + b^+ dx_2$ on $\T^2 \times \{1\}$ and $dx_1 + b^+_t dx_2$ on $\T^2 \times \{1/2\}$
(with the right coorientation). To that end, consider the pair of pants $P = A \setminus D$. According to Lemma \ref{l:pantalon}, we can define a path of foliations on $\cercle^1 \times P$ inducing the desired foliations on the boundary components of $\cercle^1 \times \partial A$, and inducing on $\cercle^1 \times \partial D $ a continuous path of linearizable foliations (\emph{i.e} whose holonomies are (compositions of) translations), which can be extended inside $\cercle^1 \times D $ by a continuous path of Reeb fillings.

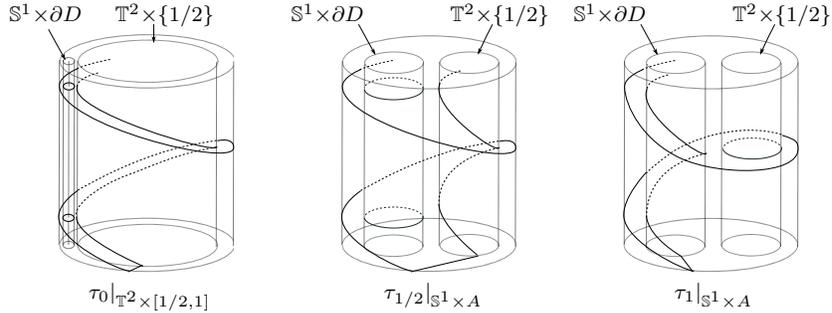
\begin{figure}[htb]
\centering
\begin{tabular}{ccc}
\includegraphics[height=3.3cm]{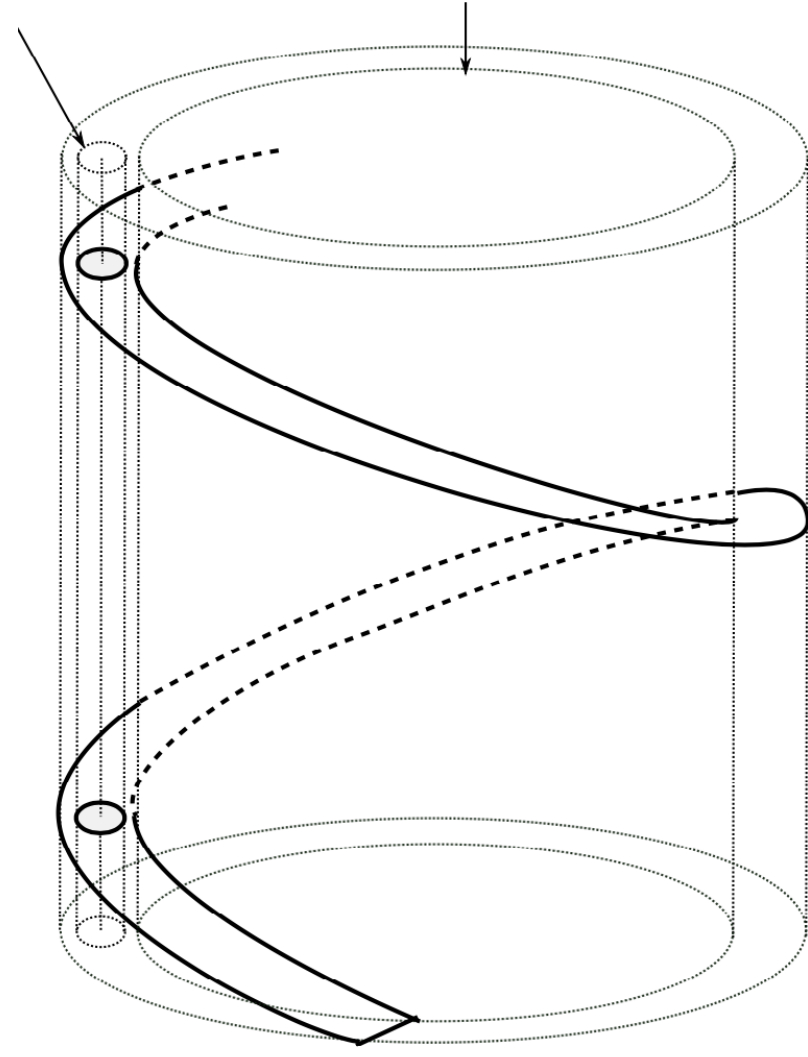}
\put(-85,95){$\scriptstyle \cercle^1\times\partial D$}
\put(-45,95){$\scriptstyle\T^2\times\{1/2\}$}
\put(-55,-10){$\scriptstyle\tau_0\res{\T^2\times[1/2,1]}$}
\hspace{1cm}
 & \includegraphics[height=3.3cm]{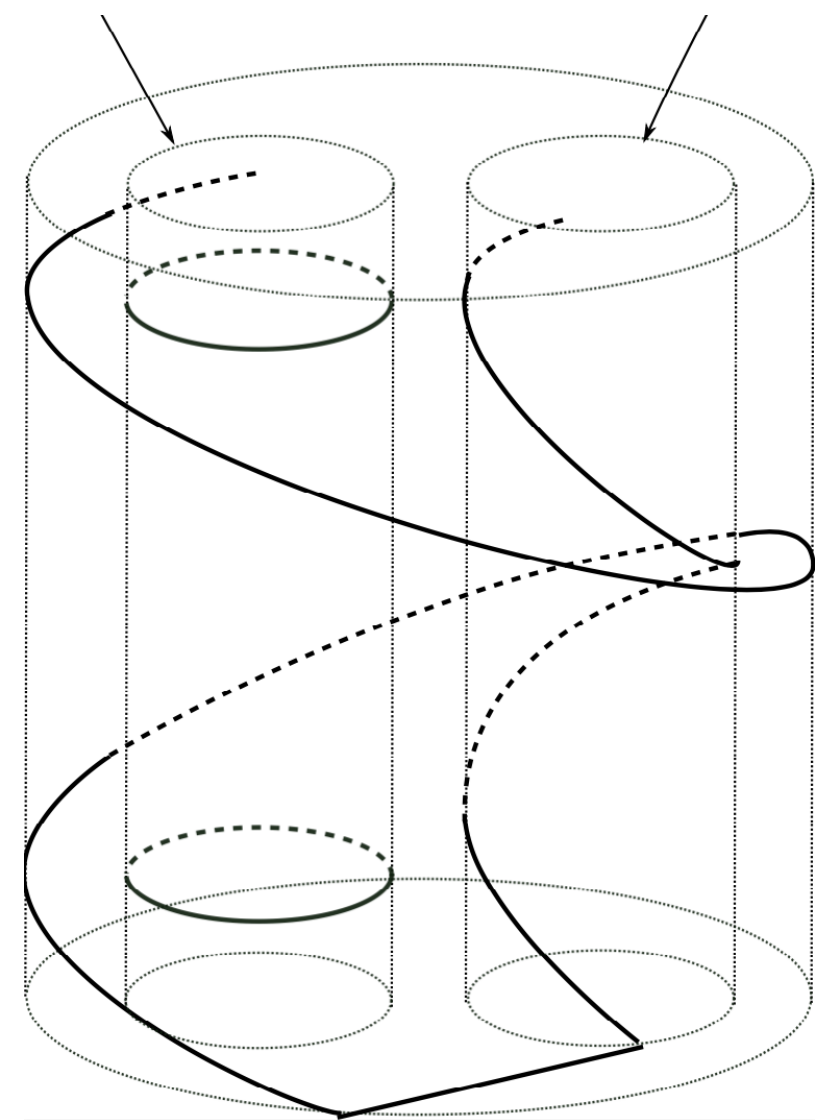}
\put(-85,95){$\scriptstyle \cercle^1\times\partial D$}
\put(-25,95){$\scriptstyle\T^2\times\{1/2\}$}
\put(-50,-10){$\scriptstyle{\tau_{1/2}}\res{\cercle^1\times A}$}
\hspace{1cm}
& \includegraphics[height=3.3cm]{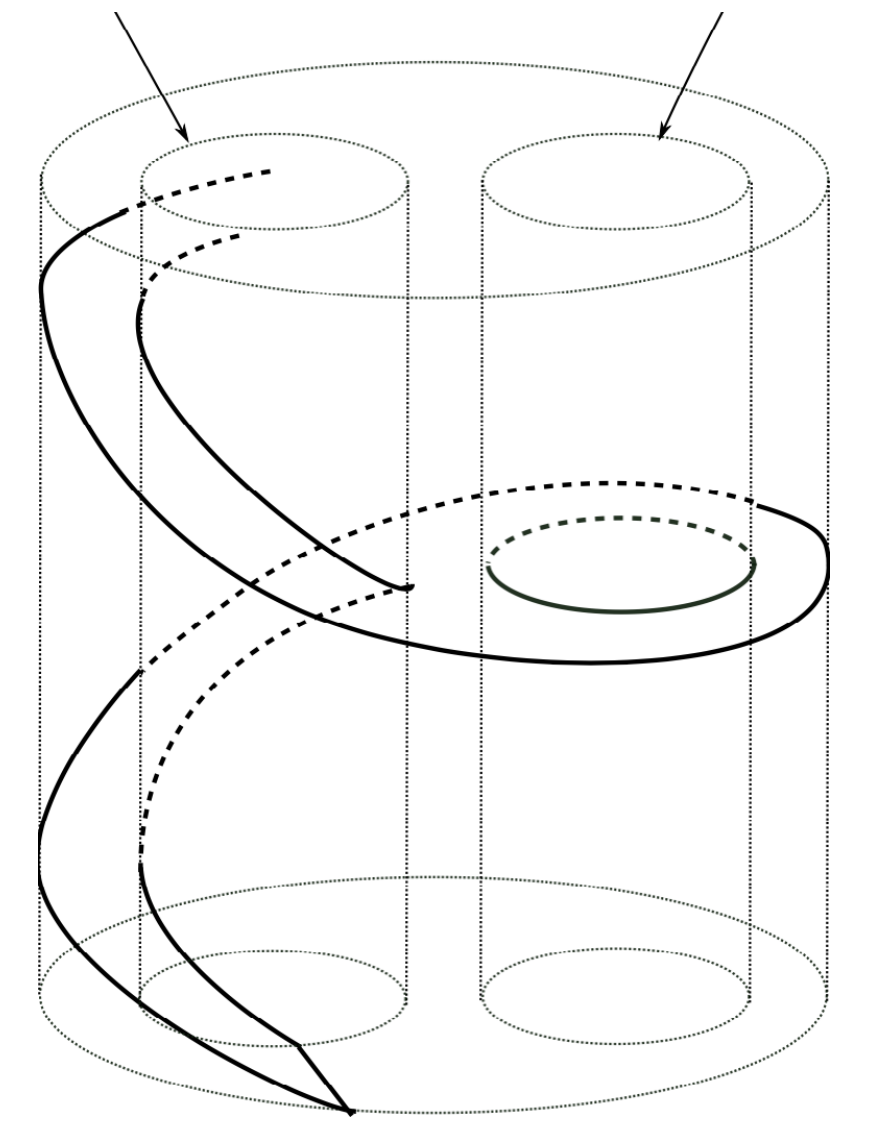}
\put(-85,95){$\scriptstyle \cercle^1\times\partial D$}
\put(-25,95){$\scriptstyle\T^2\times\{1/2\}$}
\put(-47,-10){$\scriptstyle\tau_{1}\res{\cercle^1\times A}$}
\end{tabular}
\caption{Deformation on  $\T^2\times[1/2,1]$ 
\label{pantalon}}
\end{figure}

The global foliation $\tau_1$ is transverse to the first $\cercle^1$
factor outside $\cercle^1 \times D$, where it induces a Reeb filling, which
concludes the proof.
\end{proof}


\subsection{Holonomy fragmentation}\label{ss:frag}

\begin{lemma}\label{l:frag-fol}
A Reeb filling on $\D^2 \times \cercle^1$ can be deformed to a malleable foliation through neat foliations and
relative to the boundary.
\end{lemma}

We already know that a Reeb filling can be deformed to a Schweitzer foliation rel. boundary (cf. Remark \ref{r:lid}). But if the holonomy $f$ on the boundary has no interval of fixed points (\emph{i.e} if it is not trivial, since it is a translation), this Schweitzer foliation is not \emph{simple}. The idea is to replace it by a collection of ``parallel" simple Schweitzer foliations whose holonomies form a decomposition of $f$. This uses the following fragmentation lemma for diffeomorphisms, along with the flexibility of suspension foliations over a punctured disk just as Theorem \ref{t:larcanche} follows from Herman's decomposition Theorem \ref{t:herman} together with the flexibility of suspension foliations over a pair of pants (Lemma \ref{l:pantalon}).

\begin{lemma}\label{l:fragmentation}
Every element of $\DDi$ is the composition of finitely many elements of $\DDi$, each having intervals of fixed points.
\end{lemma}

\begin{proof}
Let $f \in \DDi$. If $|f(x)-x| < 1/2$ for all $x \in \R$, there exists a diffeomorphism $g \in \DDi$ which coincides with the identity near $0$ and with $f$ near $1/2$. Hence $f = g \circ (g^{-1} \circ f)$, where $g$ and $g^{-1} \circ f$ each have an interval of fixed points.

In the general case, the function $v = f - \id$ is $1$-periodic and satisfies $\max v - \min v < 1$. Thus $v = n \lambda + w$ for some $n\in\N$, with $|w(x)|< 1/2$ for all $x \in \R$ and $\lambda \in \; (-1/2,1/2)$. Hence, 
$$f = T_{\lambda}^n \circ  (\id + w)$$ 
and each component of the righthandside falls into the first case.
\end{proof}

\begin{proof}[Proof of Lemma \ref{l:frag-fol}] Let $\tau$ be a Reeb filling of a translation $T_\lambda$ and $f_1 \circ... \circ f_n$ a decomposition of $T_\lambda$ into diffeomorphisms each having intervals of fixed points (cf. Lemma \ref{l:fragmentation}). First of all, according to Remark \ref{r:lid}, $\tau$ can be deformed among (neat) foliations and relative to the boundary into a Schweitzer foliation of the same slope. Now let $D_1$ be a disk in $\D^2$ big enough that the foliation on $(\D^2 \setminus D_1) \times \cercle^1$ is conjugate to the product foliation of the linear foliation on the boundary $\partial \D^2 \times \cercle^1$ by a small interval, and let $D_2$,...,$D_n$ be small disks in $\D^2 \setminus D_1$ so that $D_i \times \cercle^1$ is foliated by disks for all $i \in \{2,...,n\}$. Again according to Remark \ref{r:lid}, these trivial foliations can be deformed to $\Lid$ rel. $\partial D_i \times \cercle^1$. Denote by $\bar \tau$ the resulting foliation on $\D^2 \times \cercle^1$. Let $f_2^t$,...,$f_n^t$, $t\in[0,1]$ be continuous paths in $\DDi$ such that $f_i^0=\id$ and $f_i^1=f_i$ for all $i\in\{2,...,n\}$, and let $f_1^t = T_\lambda \circ (f_2^t \circ... \circ f_n^t)^{-1}$ for all $t \in [0,1]$. Finally, let $P = \D^2 \setminus (D_1 \cup ... \cup D_n)$. According to (Lemma \ref{l:holonomie} and a straightforward generalization of) Lemma \ref{l:pantalon}, there is a continuous path of foliations $\bar \tau_t$, $t\in[0,1]$, on $P \times \cercle^1$, constant on 
$\partial \D^2 \times \cercle^1$ such that $\bar\tau_0 = \bar\tau \res{P\times \cercle^1}$ and the holonomy of $\bar \tau_t$ on $\partial D_i \times \cercle^1$ is $f_i^t$. Now according to Theorem \ref{t:larcanche}, this can be extended to $\cup_i 
D_i \times \cercle^1$ by continuous paths of Schweitzer foliations, and the final foliation $\bar\tau_1$ of $\D^2 \times \cercle^1$ is malleable.
\end{proof}

\begin{remark}\label{r:agreable} Let $\pi : M \to S$ be a circle bundle over a compact oriented surface, and consider the space of cooriented foliations on $M$ positively transverse to the fibers except above a finite number of simple closed curves in the interior of $S$ whose preimages by $\pi$ are neat leaves. We show in \cite{Ey}, using the same kind of arguments as above, that this space is path-connected, and that this remains true if we fix the foliation on the boundary of $M$ (if there is any). This extends the following result of Larcanch\'e \cite{La}: given a circle bundle $\pi : M \to S$ over a compact oriented surface $S$, the inclusion map from the space of foliations transverse to the fibres into the space of all foliations on $M$ is homotopic to a
constant map.
\end{remark}

\section{Density of neat foliations} \label{s:cleaning}

The aim of this section is to prove Theorem \ref{t:cleaning}, that is that any smooth foliation of a closed $3$-manifold can be made \emph{neat} (cf. Definition \ref{d:net}) by an arbitrarily small perturbation. The idea is very simple. A neat foliation is one which has only finitely many Novikov tori (\emph{i.e.} torus leaves which meet no closed transversal) near which it is described by a simple explicit model (cf. Definition \ref{d:model}). A random foliation on the other hand can have infinitely many Novikov tori, but those are gathered in a finite number of disjoint saturated sets of the form $\T^2 \times [a,b]$, where the foliation is transverse to the second factor (cf. \cite[Theorem 2]{Th1}, or for example \cite{B-F}). We will refer to such regions as \emph{Novikov stacks} of the foliation. We have to perturb the foliation in a neighbourhood of these thickened tori (leaving it unchanged on the complement) into one with finitely many torus leaves each surrounded by a nice model foliation. To that aim, we first translate this requirement in terms of holonomy (cf. Section \ref{ss:hol-net} below). Then, in Section \ref{ss:approx}, we use a result of C. Bonatti and A. Haefliger \cite{B-H} to reduce our problem of approximation of foliations to an approximation result for holonomy representations proved in \cite{B-E}.

\subsection{Neat foliations in terms of holonomy} \label{ss:hol-net}

Let $S$ be a saturated set of the form $\T^2 \times J$, where $J$ denotes a segment (possibly reduced to a point), of a foliated manifold $(M, \tau)$, on which $\tau$ is transverse to the second factor. Let $\Gamma$ be a small extension of the parametrized transverse arc $t \in J \mapsto (0,0,t) \in \T^2 \times J \simeq S$ and let $\Diff_+(\R,J)$ denote the group of germs of $\Cinf$ orientation preserving diffeomorphisms of $\R$ defined in a neighbourhood of $J$. Then the holonomy of $\tau$ on the transverse arc $\Gamma$ induces a homomorphism $h : \pi_1(\T^2,(0,0)) \simeq \Z^2 \to \Diff_+(\R,J)$. Actually, since such a homomorphism is completely determined by the image of the standard basis of $\Z^2$, what we call \emph{holonomy of $\tau$ on $\Gamma$} is simply the pair of commuting germs $(h(1,0), h(0,1))$. Let us now give a simple characterization of the neatness of a foliation in terms of holonomy.

\begin{definition} \label{d:neat-hol}
A pair $(f,g)$ of commuting elements of $\Diff_+(\R,J)$ is called \emph{neat} if $f$ and $g$ have finitely many common fixed points, and if for each such $z_0 \in \fix(f) \cap \fix(g)$, there is a $\Cinf$ vector field $\nu$ on $\R$ such that the left and right semi-germs of $f$ and $g$ at $z_0$ belong to the flow of the corresponding semi-germ of $\nu$.
\end{definition}

\begin{remark} \label{r:neat-hol}
It follows directly from classical results of G.~Szekeres \cite{Sz}, N.~Kopell \cite{Ko} and F.~Takens \cite{Ta} that if $f$ and $g$ are nowhere simultaneously infinitely tangent to the identity (or, in short, ``\emph{i.t.i}"), then $(f,g)$ is neat.
\end{remark}

\begin{proposition}\label{p:neat-hol}
The torus leaves of $\tau \res S$ are all neat if and only if the holonomy of $\tau$ on $\Gamma$ is neat.
\end{proposition}

\begin{proof} 
Let $(f,g)$ be the holonomy of $\tau$ on $\Gamma$. The torus leaves of $\tau \res S$ correspond to the common fixed points of $f$ and $g$. Assume that $(f,g)$ is neat. Let $T$ be a torus leaf of $\tau \res S$ and $z_0\in J$ the corresponding common fixed point of $f$ and $g$. By definition of neat holonomy, there exists a $\Cinf$ vector field $\nu = u \partial_z$ on $\R$
vanishing only at $z_0$, and numbers $a^+, a^-, b^+, b^-$ so that the semi-germs of $f$ and $g$ at $z_0^\pm$ coincide with the germs of the time-$a^\pm$ and $b^\pm$ maps of $\nu$ respectively. Now consider the foliation $\tau'$ on $M' = \T^2 \times \R$ defined by the equations:
\begin{align*}
\begin{cases}
dz - u(z)(a^+ dx_1 + b^+ dx_2) ,\quad (x_1,x_2,z) \in \T^2 \times [z_0,+\infty)\\
dz - u(z)(a^- dx_1 + b^- dx_2), \quad (x_1,x_2,z) \in \T^2 \times (-\infty,z_0]
\end{cases}
\end{align*}
The holonomy of this foliation on the transverse arc $\Gamma' = \{(0,0)\} \times \R$ has the same germ at $z_0$ as the holonomy $(f,g)$ of $\tau$ on $\Gamma$. Hence there is a diffeomorphism from a neighbourhood of $\T^2 \times
\{0\}$ in $M'$ to a neighbourhood of $T$ in $M$ carrying $\tau'$ to $\tau$ (see \cite[Theorem 2.3.9]{C-C1}, for example, for a proof of this standard fact), which means precisely that $T$ is a neat leaf of $\tau$.

Now assume that all the torus leaves of $\tau \res S$ are neat. Let $z_0 \in J$ be a common fixed point of $f$ and $g$ and $T$ the corresponding leaf of $\tau$. Since $T$ is neat, there is a parametrized neighbourhood $N \simeq \T^2 \times (-\eps,\eps)$ of $T \simeq  \T^2 \times \{0\}$ on which $\tau$ is defined by equations of the form:
\begin{align*}
\begin{cases}
dz - u(z)(a^+ dx_1 + b^+ dx_2) ,\quad (x_1,x_2,z) \in \T^2 \times [0,\eps)\\
dz - u(z)(a^- dx_1 + b^- dx_2), \quad (x_1,x_2,z) \in \T^2 \times (-\eps,0]
\end{cases}
\end{align*}
where $(a^{\pm},b^{\pm}) \in \R^2 \setminus \{(0,0)\}$ and $u$ is a smooth function vanishing only at $0$. Hence, the germ at $z_0$ of the holonomy of $\tau$ on $\Gamma$ is conjugate to the germ at $0$ of the holonomy $(\fbar,\gbar) \in (\Diff_+(\R,0))^2$ of the above foliation on $\Gamma' = \{(0,0)\} \times (-\eps,\eps)$. But the semi-germs of $\fbar$ and $\gbar$ at $0^\pm$ are just those of the time-$a^\pm$ and $b^\pm$ maps of the smooth vector field $\nu = u \partial_z$.
This shows that $(f,g)$ is neat.
\end{proof}

\subsection{An approximation result for foliations and holonomies} \label{ss:approx}

According to Proposition \ref{p:neat-hol}, what is left to prove is that for every Novikov stack $S$ of a foliated manifold $(M,\tau)$ (coming with a transverse arc $\Gamma$), $\tau$ can be perturbed, relative to the complement of a neighbourhood of $S$, into a foliation also having $S$ as a saturated set but whose holonomy on $\Gamma$ is neat. According to the following result of \cite{B-F} based on the main theorem of \cite{B-H}, this boils down to showing that any commuting pair $(f,g) \in
\Diff_+(\R,J)^2$ can be approximated by neat pairs:

\begin{proposition}[cf. \cite{B-F}, Proposition 1.b.1] \label{p:B-F} 
Let $(f,g)$ be the holonomy of $\tau$ on $\Gamma$ and $\ft,\gt$ two commuting local diffeomorphisms of $\R$ defined near $J$, $\Cinf$-close to $f$ and $g$ respectively, coinciding with them outside a small neighbourhood of $J$. Then there exists a foliation $\taut$ of $M$ $\Cinf$-close to $\tau$ which coincides with $\tau$ outside a small neighbourhood of $S$ and whose holonomy on $\Gamma$ is $(\ft,\gt)$.
\end{proposition}

We thus need the following approximation result for commuting germs of diffeomorphisms:

\begin{proposition} \label{p:neat-approx} 
Every commuting pair $(f,g) \in ( \Diff_+(\R,J))^2$ can be $\Cinf$-approximated by a neat pair ($\ft$,$\gt$), coinciding with $(f,g)$ outside a small neighbourhood of $J$.
\end{proposition}

We will obtain this as a consequence of Theorem \ref{t:B-E} below. Given an element $f$ of $\Diff_+(\R,J)$, we denote by $\ITI(f)$ the set of points where $f$ is infinitely tangent to the identity.

\begin{definition}[cf. \cite{B-E}]  \label{d:p-clean} 
A pair $(f,g)$ of commuting elements of $\Diff_+(\R,J)$ is called \emph{piecewise clean} if $I \setminus \ITI(f) \cap \ITI(g)$, for some small neighbourhood $I$ of $J$, has finitely many connected components on the closure of which the restrictions
of $f$ and $g$ either belong to a common (germ of a) $\Cinf$ flow or are iterates of the same (germ of a) smooth diffeomorphism.
\end{definition}

\begin{theorem}[cf. \cite{B-E}] \label{t:B-E}
Any pair $(f,g)$ of commuting elements of $\Diff_+(\R,J)$ can be $\Cinf$-approximated by a piecewise clean pair ($\fbar$,$\gbar$), coinciding with $(f,g)$ outside a small neighbourhood of $J$.
\end{theorem}

\begin{remark} 
Actually, what is proved in \cite{B-E} (cf. Proposition 2.22) is an analogue of the above for diffeomorphisms of a segment, rather than germs of diffeomorphisms near a segment. But the germinal version follows directly from Proposition 2.22 and its key ingredient Proposition 2.15 in \cite{B-E}, which is also the heart of Lemma \ref{l:pf-diff} below (see Proposition \ref{l:approx-be} below for a simplified version of Proposition 2.15 of \cite{B-E}).
\end{remark}

Now the fact that any piecewiese clean pair ($\fbar$,$\gbar$) can be approximated by a neat pair is obtained by applying one of the following lemmas (or its germinal version) to the closure of each connected component of $I \setminus \ITI(\fbar) \cap \ITI(\gbar)$ independently (the diffeomorphisms involved being infinitely tangent to the identity at the boundary). The diffeomorphisms of the resulting pair $(\ft,\gt)$ might still have whole intervals of common fixed points, but then it is easy to perturb $(\id,\id) \in (\Diff_+[a,b])^2$ slightly into a pair of commuting diffeomorphisms (twice the same for example) having only $a$ and $b$ as fixed points.

\begin{lemma} \label{l:pf-champ}
Every $\Cinf$ map $\nu$ from $[0,1]$ to $\R$ that is nowhere infinitely flat on $(0,1)$ can be $\Cinf$-approximated by a map with the same property, the same $\infty$-jet at the boundary and finitely many zeroes.
\end{lemma}

\begin{lemma} \label{l:pf-diff} 
Every $h \in \Diff_+[0,1]$ that is nowhere infinitely tangent to the identity on $(0,1)$ can be $\Cinf$-approximated by some $\hti \in \Diff_+[0,1]$ of the same kind with finitely many fixed points, near each of which $\hti$ belongs to the flow of some $\Cinf$
vector field.
\end{lemma}

\begin{proof}[Proof of \ref{l:pf-champ}] 
If we forget about the $\infty$-jets at the boundary, this is just a standard transversality result. A little more care is needed if we want to preserve the jets. Actually, if $\nu$ is not infinitely flat at $0$ nor $1$, there is nothing to do. So let us consider the case where $\nu$ is infinitely flat at $0$, say, and let us perturb it near $0$ so that the resulting map has finitely many zeros there. The idea is basically to multiply $\nu$ by some smooth step function equal to $0$ on some small neighbourhood $[0,t]$ of
$0$ (and to $1$ away from there) and then spread the restriction to $[t,1]$ of the resulting function to all of $[0,1]$ (to get rid of the interval of zeros $[0,t]$).

More precisely, let $\rho$ be a smooth map from $[0,+\infty)$ to $[0,1]$ vanishing on $[0,1]$, equal to $1$ on $[2,+\infty)$ and increasing on $[1,2]$, and consider, for all $t \in (0,1]$, the map
$$\nu_t : x \in [0,1] \mapsto \rho(\tfrac x t) \nu(x).$$
Let us check that $t \in [0,1] \mapsto \nu_t$, with $\nu_0=\nu$, is continuous at $0$ (in $\Cinf$ topology). On $[2t,1]$, $\left|\nu_t^{(n)}(x) - \nu^{(n)}(x) \right| = 0$. And on $[0,2t]$,
\begin{equation*}
\sup_{x\in[0,2t]} \left|\nu_t^{(n)}(x) - \nu^{(n)}(x)\right|  
= \sup_{x\in[0,2t]} \left|\sum_{k=0}^n \tbinom{n}{k} \tfrac 1 {t^{n-k}}
\rho^{(n-k)} \left(\tfrac x t\right) \nu^{(k)}(x) - \nu^{(n)}(x)\right|
= o(t)
\end{equation*}
since $\sup_{x\in[0,2t]} \left|\nu^{(k)}(x)\right| = o(t^l)$ for all $l$, $\nu$ being infinitely flat at $0$. So $t \mapsto \nu_t$ is indeed continuous at $0$.Now let $(h_t)_{t\in[0,1/2]}$ be a continuous family of increasing $\Cinf$ maps on $[0,1]$ satisfying $h_0 = \id$ and for all $t \in [0,1/2]$, $h_t(0) = t$ and $h_t = \id$ near $1$. For $t$ small enough, $\nut = \nu_t \circ h_t$ is
$\Cinf$-close to $\nu_0 \circ h_0 = \nu$. It is furthermore infinitely flat at $0$, equal to $\nu$ near $1$, and its zeros in $(0,1]$ are the preimages under $h_t$ of those of $\nu \res{(t,1]}$, and are thus finite in number near $0$. One concludes by repeating the above process near $1$ if necessary. The resulting map has finitely many zeros near $0$ and $1$ and is still nowhere infinitely flat on $(0,1)$, so has finitely many zeros on $[0,1]$. \end{proof}

\begin{proof}[Proof of Lemma \ref{l:pf-diff}]
First apply Lemma \ref{l:pf-champ} to $h_0 = h-\id$, denote by $\hbar_0$ the resulting map and define $\hbar$ as $\id + \hbar_0$, which satisfies all the requirements of Lemma \ref{l:pf-diff} except maybe the last one. Actually, according to a result of Takens \cite[Theorem 4 p. 165]{Ta}, $\hbar$ does belong to the flow of some $\Cinf$ vector field near each \emph{interior} fixed point because it is not \emph{i.t.i} there. This however might not be true at $0$ and $1$, but this problem can be solved by some \emph{local} perturbation as follows. Assume for example that $\hbar$ is \emph{i.t.i} at $0$, and denote by $c$ the smallest fixed point of $\hbar$ different from $0$. According to well-known results by Szekeres \cite{Sz} and Kopell \cite{Ko}, $\hbar \res{[0,c)}$ belongs to the flow of a unique $\CC^1$ vector field $\nu$ called the Szekeres vector field of $\hbar
\res{[0,c)}$, and we may apply Proposition 2.15 in \cite{B-E}, that we restate below in our present simplified setting:

\begin{proposition}\label{l:approx-be}
Let $f$ be a smooth diffeomorphism of $[0,c)$, \emph{i.t.i} at $0$, without fixed points in $(0,c)$, and let $\nu$ be its Szekeres vector field.

Then, for all $\eps > 0$, $a \in (0,c]$ and $k \in \N$, there exists $x_0 \in (0,a]$ and a vector field on $[0,c)$ coinciding with $\nu$ on $[x_0,c)$, $\Cinf$ on $[0,c)$, infinitely flat at $0$, and $\eps$-$\CC^k$-small on $[0,\max(f^2(x_0),f^{-2}(x_0))]$.
\end{proposition}

Let $\nut$ be a vector field obtained by applying the above Proposition to $f = \hbar\res{[0,c)}$. It follows directly from the proof of 2.15 in \cite{B-E} that $\nut$ does not vanish on $(0,c)$, but one can apply Lemma \ref{l:pf-champ} instead to make sure that $\nut$ has finitely many zeros in $(0,c)$, and none infinitely flat. The time-$1$ map $\ft$ of $\nut$ coincides with
$\hbar \res{[0,c)}$ on $[f^{\pm 1}(x_0),c)$ and is $\Cinf$-close to $\id$ on $[0,f^{\pm 1}(x_0)]$, as is $\hbar$ if $x_0$ is small enough. Repeating the above process near $1$ if necessary we get a $\Cinf$ approximation of $\hbar$ with all the required properties.
\end{proof}

\section{Appendix: Flexibility of almost integrable plane fields}\label{s:eliashberg}

This section is devoted to the proof of Proposition \ref{p:p-i}, which states that any two homotopic almost integrable plane fields are actually homotopic through almost integrable plane fields. Actually, in order to prove Theorem \ref{t:pik}, we need a more general statement, replacing the parameter space $[0,1]$ and its boundary $\{0,1\}$ by a compact finite dimensional polyhedron $K$ and a closed subpolyhedron $L$ of $K$ (typically, $K = \D^n$ and $L = \cercle^{n-1}$). 

We will use the following vocabulary. A $K$-\emph{plane field} $\xi$ on a manifold $M$ is a family $\xi_t$, $t\in K$, of plane fields on $M$. Now given a subset $X \subset K \times M$, we say that a $K$-plane field $\xi$ is integrable on $X$ if for every $t\in K$, the plane field $\xi_t$ is integrable on $X_t = X \cap (\{t\} \times M)$. In practice, $X$ is often of the form $(K \times A) \cup (L \times M)$, where $A$ is a subset of $M$. We say that a $K$-plane field $\xi$ is almost horizontal on a collection of balls $B\subset M$ if, for every $t\in K$, the plane field $\xi_t$ is almost horizontal on $B$. Finally, we say that a $K$-plane field $\xi$ is $(K' \times B)$-almost integrable if for every $t\in K'\subset K$, the plane field $\xi_t$ is $B$-almost integrable.

Recall that given a subset $A$ of a topological space, the notation $\Op(A)$ refers to a small nonspecified open neighbourhood of $A$.

\begin{proposition} \label{param} Consider a closed $3$-manifold $M$, a collection of balls $B$ in $M$, a compact finite dimensional polyhedron $K$ and a closed subpolyhedron $L$ of $K$. 
Let $\xi$ be an $(L \times
B)$-almost integrable $K$-plane field on $M$, almost horizontal on $B$. There exists a $K$-plane field $\xibar$ on $M$ with the following properties:
\begin{enumerate}
\item
$\xibar$ is homotopic to $\xi$ relative to $(K \times \Op B) \cup (L \times 
M)$;
\item
$\xibar$ is $(K \times \Bb)$-almost integrable for some collection of balls $\Bb$ containing $B$.
\end{enumerate}
\end{proposition}

In order to deform plane fields to integrable ones, Thurston initiated the use of triangulations. He demonstrated the effectiveness of his idea in \cite{Th1, Th2, Th3}. Eliashberg then adapted the techniques of \cite{Th2}
in \cite{El} to deform plane fields to contact structures, and extended them to families of plane fields depending on any number of parameters. In return, Proposition \ref{param} and its proof are modeled on part of \cite{El}, namely Lemma 3.2.1 and its proof, which relies on sections 2.3 and 2.4 of the same paper. Our aim here is mainly to detail and complete Eliashberg's arguments (see in particular Remark \ref{r:2.3.3}). We also refer the reader to the book \cite{Ge} by H. Geiges for further details about the complete argument of \cite{El}. 

We will now give an outline of the proof of Proposition \ref{param}, which takes up the entire Appendix. In particular, we will try to emphasize the difference between the nonparametric and multiparametric construction of almost-integrable plane fields (cf. ``Proof of Lemma \ref{l:R3}...''), and to motivate our choice to give a full proof of the multiparametric version, including a tiresome induction argument, rather than restrict to the one-parameter case which would convey most of the ideas.
\medskip

\noindent\textbf{Reduction to $\R^3$}. First, in Subsection \ref{ss:euclide}, we cover $M$ with finitely many charts to reduce to a problem in $\R^3$. The rest of the Appendix is devoted to the analogue of Proposition \ref{param} in $\R^3$, namely Lemma \ref{l:R3}, whose proof we now outline.\medskip

\noindent\textbf{The nonparametric case}. Before dealing with families of plane fields, we first recall Thurston's strategy to make \emph{one} plane field $\xi$ almost integrable. The starting point is to construct a triangulation in ``good position'' (or ``general position'' in Thurston's words) with respect to $\xi$, meaning basically that the direction of $\xi$ is ``almost constant'' on each $3$-simplex (this can be ensured simply by taking the triangulation fine enough) and that the faces and edges are transverse to $\xi$ (this is achieved by ``jiggling'' the previous triangulation).

Good position makes it ``easy'' to make $\xi$ integrable in a neighborhood of the $2$-skeleton \emph{and} to pick this neighborhood so that $\xi$ is almost horizontal on each ball of the complement. More precisely, one first makes $\xi$ integrable in a neighborhood of every vertex, then every edge and finally every face. The deformations near all simplices of a given dimension should be thought of as simultaneous, the tricky part being of course to guarantee the compatibility of deformations performed near adjacent simplices. This is made possible by the existence, near every simplex $\sigma$, of a vector field $\nu$ tangent to $\xi$ and transverse to $\sigma$. The deformation then consists in keeping $\xi$ unchanged \emph{on} $\sigma$ and making it invariant under $\nu$ in a neighborhood of $\Int\sigma$, covered by a flow box of $\nu$ with base $\Int\sigma$ (cf. Lemma \ref{l:defmod} for a generalized quantitative version of this process). Since $\xi$ is already integrable near $\partial \sigma$ by the previous step, it is already invariant under $\nu$ there and thus remains unchanged, which guarantees the global coherence of these local perturbations. 

Note the importance of the transversality condition on the triangulation. If the triangulation was not in good position with respect to $\xi$ (but still sufficiently fine), one could still find, for every face $\sigma$, a vector field $\nu$ tangent to $\xi$ with a flow box covering a neighbourhood of $\sigma$, and make $\xi$ invariant under this flow. \emph{But} any flow line leaving \emph{and} reentering the neighbourhood of $\partial \sigma$ would be a potential obstruction to keeping $\xi$ unchanged in this neighbourhood, which was our guarantee for the global coherence of the perturbations. So one would have to deal with these ``special faces'' first, like ``big vertices'',   \emph{before} carrying on with the other (actual) vertices, faces and edges. This is a problem we have to face in our parametric situation.\medskip

\noindent\textbf{Proof of Lemma \ref{l:R3}: meaning of the Key Lemma \ref{l:local} and of the ``curvature'' Lemma \ref{l:curv2}}. Indeed, what we want, in order to prove Lemma \ref{l:R3}, is to deform \emph{an entire family} $\xi_t$, $t\in K$, of plane fields to make them all integrable \emph{outside the same balls}. To that end, we must use the  \emph{same triangulation} for every value of the parameter. But we cannot expect a single triangulation to be in good position with respect to every plane field (the triangulation can be fine enough that every $\xi_t$ is almost constant near each $3$-simplex, but since the direction of $\xi_t$ varies with $t$, there is no hope to fulfil the transversality condition in general). The common triangulation we use is a rescaling by a small factor $d$ of a specific triangulation $\Delta$ of $\R^3$ defined in Section \ref{ss:triangulation} (its main features will be presented below). The choice, or rather the existence of a proper scaling factor $d$ plays an important and elaborate part in the proof. We try to clarify what underlies this choice by writing in bold the relevant parts of the following outline.

The Key Lemma \ref{l:local} (Section \ref{ss:triangulation}) claims that, \textbf{for any $d$ sufficiently small} (so that each plane field is ``almost constant'' near each simplex), our $K$-plane field can indeed be made integrable in a neighborhood of the $2$-skeleton of the rescaled triangulation $d\Delta$. As mentioned earlier, the fact that the triangulation is not in good position with respect to every plane field makes this already somewhat harder than in the nonparametric case: some $\xi_t$ might be tangent to some face $\sigma$ of some $3$-simplex $\tau$ and these parameters and ``special'' simplices (cf. Definition \ref{d:special}) have to be dealt with in a particular way. 

But perhaps more importantly, parameters (and in particular situations like the one above) make it harder to guarantee the almost horizontality of all plane fields on the complement of a common neighborhood of the $2$-skeleton: without further precautions, a plane field $\xit_t$ obtained after deforming a plane field $\xi_t$ as above might have infinitely many points of tangency with the boundary sphere of some randomly embedded ball in $\tau$. This will not happen, however, if the sphere is convex enough compared to the variations of $\xit_t$, as explained in Subsection \ref{ss:curv} (cf. Lemma \ref{l:curv2}). But remember this sphere must lie in the neighborhood of the $2$-skeleton where the $K$-plane field has been made integrable. This is why the Key Lemma \ref{l:local} has to quantify the size of this neighbourhood, namely a $\mu d$-neighborhood (cf. 2. in the Key Lemma), for some $\mu$ depending only on the initial $K$-plane field, but not on the scaling factor $d$. That way, once we have the Key Lemma, \textbf{choosing $d$ small enough once again}, we can make the spheres as curved as we like. 

At that point, we use a key property of the model triangulation $\Delta$: its $3$-simplices are all copies of a finite number of model simplices. One can embed a strictly convex sphere in the $\mu$-neighborhood of the boundary of each of them. The principal curvatures of these model spheres are bounded below by some positive number $k$. Then the spheres we embed in each $3$-simplex of $d\Delta$ are simply scaled copies of these model spheres and thus have principal curvatures bounded below by $k/d$, which can be made as big as we like by taking $d$ sufficiently small.

But the deformed $K$-plane field given by the Key Lemma depends on the triangulation, and thus on the choice of $d$. So when we shrink $d$ to increase the curvature of the spheres we can embed, we change $\xit_t$, and possibly its $C^1$ norm, which determines the minimal curvature guaranteeing the almost horizontality...  So, in short, we need to make sure that the variations of the plane fields resulting from the Key Lemma remain bounded \emph{regardless of the scaling factor $d$}, hence the need for point 3. in the Key Lemma. 

Now that we have explained the content of the Key Lemma and how it combines with Lemma \ref{l:curv2} to produce plane fields that are indeed almost integrable, and thus prove  Lemma \ref{l:R3}, we can go on with the outline of the proof of the Key Lemma itself, which is carried out in Sections \ref{ss:modele} and  \ref{ss:cle}.\medskip

\noindent\textbf{Proof of the Key Lemma I: the deformation model}. Again, the aim is roughly to show that, for any $d$ small enough, one can make a given $K$-plane field integrable on a neighbourhood of the $2$-skeleton of $d\Delta$ whose diameter does not depend on $d$ up to scaling, keeping control (again independent of $d$) on the $C^1$ norm of the resulting plane fields. To that end, the idea is, like in Thurston's process (cf. \emph{the nonparametric case}), to apply a local deformation model repeatedly (namely near each simplex of the triangulation, ``special'' or not). 

But we want quantitative control on the resulting object, that depends neither on the number of times we applied the model nor on the specific simplices to which we applied it. So we need a \emph{quantitative} deformation model which, for a sufficiently large (but necessarily restricted) class of plane fields, tells us how to make them integrable and gives uniform control on the size of the corresponding perturbation. This is the content of Lemma \ref{l:defmod} to which Section \ref{ss:modele} is devoted. No scaling parameter $d$ is involved there, since Lemma \ref{l:defmod} is precisely intended to provide bounds independent of $d$ in the end. This lemma is really a statement about plane fields defined near a simplex of the ``big'' triangulation $\Delta$ and will be applied to the plane fields of the Key Lemma defined near a simplex of $d\Delta$ only after a rescaling by a factor $1/d$.  

Now the difficulty of the proof of the Key Lemma consists in reducing to situations where the deformation model is indeed applicable (which, in particular, \textbf{requires the triangulation to be fine enough}) and to do things in the right order so that each step is compatible with the previous ones. \medskip

\noindent\textbf{Proof of the Key Lemma II: triangulation of the parameter space and induction}. 
Again, given a scaling factor $d$ (which will be chosen \emph{a posteriori}), the idea is to perturb every plane field $\xi_t$, $t\in K$, in a neighborhood of every simplex of the $2$-skeleton of $d\Delta$ to make it integrable there, and to do this continuously with respect to $t$. The problem is that, given $t$, the order in which one must deal with the simplices if one wants to avoid incompatibilities depends on the position of $\xi_t$ with respect to the triangulation. Indeed, as mentioned above (cf. \emph{Proof of \ref{l:R3}...}), one has to start with \emph{special} faces to which $\xi_t$ is ``almost tangent'' ($d$ being assumed small enough here that the direction of each $\xi_t$ on a given simplex is ``almost constant''). But this, of course, depends on $t$, so the order in which the deformation is conducted also does, which is a bad start if we want a deformation continuous in $t$... 

To deal with this issue, we start by triangulating \emph{the parameter space} itself finely enough so that if one plane field $\xi_{t_0}$ is almost tangent to some face of the ``spatial'' triangulation, then all other $\xi_t$'s are, for $t$ in the same simplex of the ``temporal'' triangulation as $t_0$. Thus, each simplex of the ``temporal'' triangulation has its own set of \emph{special} faces. All of this must of course be explicitly quantified. In particular, the ``almost tangency'' is defined so that a single plane field cannot be ``almost tangent'' to two adjacent faces (which would be a problem because we want to be able to deform it near all special faces simultaneously and independently). This uses a second key property of the triangulation $\Delta$: the existence of a uniform lower bound on the angles between adjacent simplices (not contained in one another), ``almost tangency'', and thus ``special faces'', being defined in terms of that bound. 

One then proceeds by induction on the successive skeleta of $K$ (cf. Lemma \ref{l:recurrence}). Our reasons for carrying out the induction explicitly for any $K$, rather than reducing to the case where $K$ is made of a single simplex (cf. \cite[p. 632]{El}) or where $K$ is one-dimensional, are the following.  At each step, we need the \emph{new} special faces (i.e. the ones associated to the plane fields obtained after the previous step), for any given simplex $K_*$ of $K$, to still be disjoint. Therefore, we must, at each step, keep control on the angle between the ``old'' plane fields and the new ones (cf. last condition in Lemma \ref{l:recurrence}). This (and simply the possibility to apply the deformation Lemma near each simplex) \textbf{puts constraints on the choice of $d$}. Furthermore, recall that we want a bound on the $C^1$ norm of the resulting plane fields independent from the choice of the scaling factor $d$. But as we will see in more detail in Subsection \ref{ss:modele}, the local deformation lemma \ref{l:defmod} only provides a bound on the $C^k$ norm of the resulting plane field in terms of the $C^{k+1}$ norm of the initial one. Thus the induction works basically as follows:
\begin{itemize}
\item We start with a $K$-plane field $\xi$ (automatically $C^{n+1}$-bounded on any compact set, where $n$ denotes the dimension of $K$).
\item \textbf{For any $d$ small enough}, one can perturb $\xi_t$ for $t$ in a neighborhood of the $0$-skeleton of $K$ (leaving $\xi_t$ unchanged for $t$ outside a bigger neighborhood) to make it integrable first near special $2$-simplices, then inductively over $0$-, $1$- and normal $2$-simplices of $d\Delta$. This is done by applying the deformation model to $\xi_t$ in a neighbourhood of each simplex (or rather to a rescaled version of it by a factor $1/d$). We then obtain a $K$-plane field $\xi^1$ and a $C^n$ bound, proportional to $d$, on the rescaling of $\xi^1$ by a factor $1/d$. \textbf{Provided $d$ was chosen small enough}, the deformed plane fields are arbitrarily $C^0$-close to the old ones. In particular, this implies that special simplices for the new ones are still disjoint.
\item \textbf{Shrinking $d$ in the previous step \emph{a posteriori} if necessary} (the above being valid \emph{for any $d$} below some given bound determined by the initial $K$-plane field), we continue with the $1$-skeleton of $K$, carrying out the previous step relative to the boundary of the $1$-skeleton, i.e. to the $0$-skeleton of $K$. We obtain a $K$-plane field $\xi^2$ with a $C^{n-1}$ bound on the rescaling of $\xi^1$ by a factor $1/d$, and again a bound on the angle by which the plane fields have been modified.
\item We continue on $K^2$, $K^3$, \dots, $K^n$ and obtain $\xi^3$, $\xi^4$, \dots, $\xi^n=:\xibar$ with $C^{n-2}$, $C^{n-3}$,\dots, $C^1$ bounds proportional to $d$ on their rescalings. Since the $C^1$ norm of $\xibar$ is $1/d$ times the $C^1$ norm of its rescaling (cf. Subsection \ref{sss:induction}), this gives the desired $C^1$ control on $\xibar$. 
\end{itemize}
For simplicity, we have not mentioned here the matter of the size of the neighbourhood of the $2$-skeleton on which the plane fields are made integrable, but this too must be controlled from one step to the next. 
\medskip

Throughout the appendix, $\norm{.}$ will denote both the euclidean norm on $\R^3$ and the operator norm on the space of $k$-linear maps $L^k(\R^3,\R^3)$ associated to it. When there is no ambiguity on the domain of definition $U\subset \R^3$ of a $C^m$ map $f:U\to\R^3$ (resp. $U\to L^k(\R^3,\R^3)$), we will write
\begin{equation*}
\norm{f}_0=\sup_{p\in U}\norm{f(p)}\quad\in[0,+\infty]
\end{equation*}
and
\begin{equation}\label{e:normm}
\norm{f}_m=\max_{1\le k\le m }\norm{D^kf}_0.
\end{equation}

\subsection{Reduction to open sets of euclidian space}
\label{ss:euclide}
The statement we will need in $\R^3$ is the following. 

\begin{lemma} \label{l:R3}
Let $U$ be an open subset of $\R^3$, $F$ a closed subset of $U$ and $\xi$ a $K$-plane field on $U$ integrable on $(K \times \Op F) \cup (L \times U)$. Given a compact subset $A \subset U$, there exists a $K$-plane field $\xibar$ on $U$
satisfying the following properties:
\begin{enumerate}
\item
there is a compactly supported homotopy from $\xi$ to $\xibar$ relative to $(K \times \Op F) \cup
(L \times U)$;
\item
$\xibar$ is integrable on $K \times (A_* \setminus B)$ and almost horizontal on $K \times B$, where $A_*$ is a compact neighbourhood of $A$ and $B$ a collection of balls in $\Int A_* \setminus F$. 
\end{enumerate}
\end{lemma}

\begin{figure}[htbp]
\centering
\includegraphics[height=4cm]{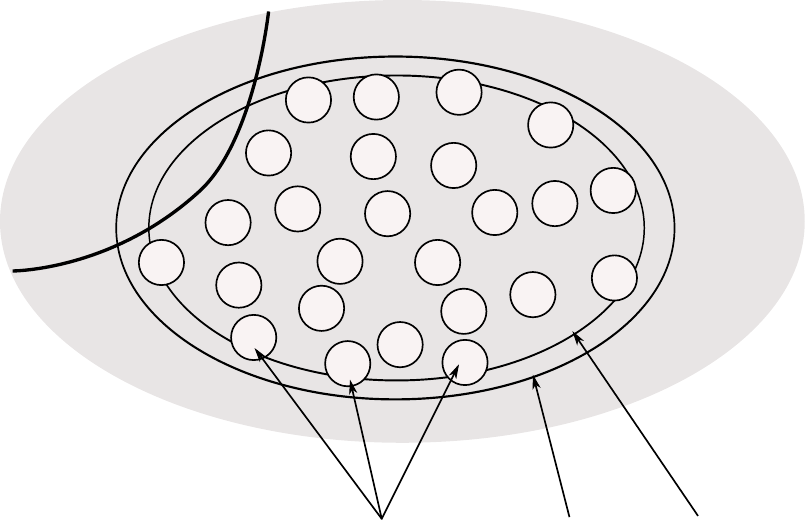}
\put(-54,-8){$\scriptstyle \partial A_*$}
\put(-27,-7){$\scriptstyle \partial A$}
\put(-93,-8){$\scriptstyle B$}
\put(-160,70){$\scriptstyle F$}
\put(-17,63){$\scriptstyle U$}
\caption{Setup in Lemma \ref{l:R3}}
\label{fig:5.2}
\end{figure}
\begin{proof}[Proof of Proposition \ref{param} assuming Lemma \ref{l:R3}]
\smallskip
Let $M$, $B$ and $\xi$ be as in Proposition \ref{param} and $A_{0*}$ be a compact neighbourhood of
$B$ such that $\xi$ is integrable on $K \times (\Op A_{0*} \setminus B)$. Consider open charts $V_i \subset M$, $1 \le i \le p$, and compact subsets $W_i \subset V_i$ such that $M = \bigcup W_i$. Lemma \ref{l:R3} applied to
$$U_1 = V_1 \setminus B,\quad F_1 = U_1 \cap A_{0*}, \quad A_1 = W_1 \setminus \Int
A_{0*}$$
and to the $K$-plane field $\xi$ restricted to $U_1$, provides a compact set $A_{1*}$, a collection of balls $B_1 \subset \Int A_{1*}$ and a new $K$-plane field $\xi_1$ on $U_1$ equal to $\xi$ on $(K\times \Op F_1)\cup (L\times U_1)$, and which extends to $M$ by $\xi_1 = \xi$ on $M \setminus U_1$. We then apply Lemma \ref{l:R3} to 
$$U_2 = V_2 \setminus (B \cup B_1),\quad F_2 = U_2 \cap
(A_{0*} \cup A_{1*}), \quad A_2 = W_2 \setminus \Int (A_{0*} \cup A_{1*})$$
and to the $K$-plane field $\xi_1$ restricted to $U_2$. We then iterate this construction and after finitely many steps we are done.
\end{proof}

\subsection{Almost horizontality and curvature}
\label{ss:curv}

The following lemma will be used to make sure that the plane fields we construct in the next sections have the desired almost horizontality property. It corresponds to Lemmas 2.4.1 and 2.4.2 in \cite{El} (stated without proof in \cite{El} and proved in \cite{Ge}, cf. 4.7.17 and 4.7.18). Figure \ref{fig:pinceau} below gives a schematic picture of its content.

Let $\xi$ be a transversely oriented plane field on an open subset $U$ of $\R^3$. For every $p \in U$, we denote by $\xi^+(p)$ the open half-space of $T_p\R^3$ lying on the positive side of $\xi(p)$ and by $\xi^\perp(p) \in \xi^+(p)$ the positive unit normal vector. In other words, $\xi^\perp \from U \to \sphere^2$ is the Gauss map of $\xi$. For every integer $m \ge 1$, we define the $C^m$ norm of $\xi$ to be the $C^m$ norm (as defined in \eqref{e:normm}) of its Gauss map, and denote it simply by $\norm{ \xi }_m$:
\begin{equation}
\label{e:normm2}
 \norm{\xi}_m:=\norm{\xi^\perp}_m. 
\end{equation}

Now given two points $p,  q \in U$, the affine planes $P_p$ and $P_q$ tangent to $\xi(p)$ and $\xi(q)$ respectively, determine
a pencil, namely the set of planes containing the straight line $P_p \cap P_q$, called the axis of the pencil. Note that this axis can be at infinity, in which case the planes of the pencil are all parallel.

\begin{lemma} \label{l:curv2}
Let $U$ be an open subset of $\R^3$, $\xi$ a $\CC^1$-bounded plane field on $U$ and $S_* \subset \R^3$ a strictly convex sphere. For $d_0 > 0$ sufficiently small, every image $S \subset U$ of $S_*$ by a dilation by a factor $d \le d_0$ has the following properties:
\begin{enumerate}
\item
$\xi$ is tangent to $S$ at exactly two points, a north pole $p_+$ where their coorientations coincide and a south pole $p_-$ where they are opposite; we denote by
$\eta$ the distribution of tangent planes to the pencil defined by $\xi_{p_-}$ and 
$\xi_{p_+}$ (the coorientation of $\xi$ naturally endows $\eta$ with a coorientation);
\item
For every $\eps>0$, there exists a nonsingular vector field $\nu$ on the ball $B$ bounded by $S$, which lies in the dihedral cone $\Omega_p = 
\xi_p^+ \cap \eta_p^+$ at every $p\in B$, and which is tangent to $S$ outside the $\eps$-neighbourhood of the poles.
\end{enumerate}
\end{lemma}

\begin{figure}[htbp]
\centering
\includegraphics[height=3cm]{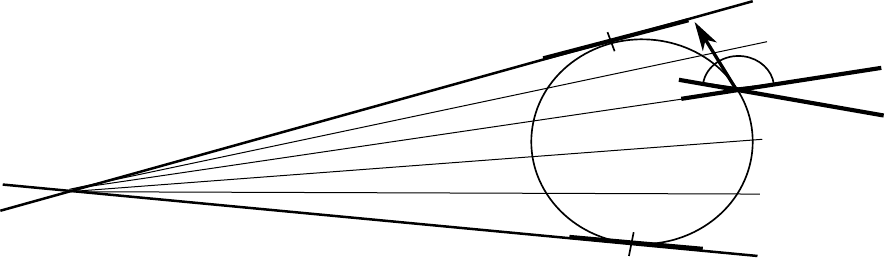}
\put(-273,11){$\scriptstyle A$}
\put(-52,8){$\scriptstyle S$}
\put(-83,-5){$\scriptstyle p_-$}
\put(-52,47){$\scriptstyle p$}
\put(-55,74){$\scriptstyle \nu_p$}
\put(-37,63){$\scriptstyle\Omega_p$}
\put(-93,78){$\scriptstyle p_+$}
\put(3,63){$\scriptstyle \eta_p$}
\put(3,45){$\scriptstyle \xi_p$}
\caption{$2$-dimensional schematic picture}
\label{fig:pinceau}
\end{figure}

\begin{proof}
Let $c = \norm{ \xi }_1$ and let $k>0$ be a (uniform) lower bound on the principal curvatures of
$S_*$ --~so the principal curvatures of $S$ are everywhere at least $k/d$.  

Let $\gamma \from S \to \sphere^2$ be the Gauss map of $S$. The curvature hypothesis means that $\gamma$ is a diffeomorphism and that its inverse satisfies $\norm{ D\gamma^{-1} }_0 \le d/k$. Thus,
$$ \norm{ D (\xi^\perp \circ \gamma^{-1}) }_0\le \norm{D\xi^\perp}_0\norm{D\gamma^{-1} }_0  \le cd/k . $$
For all $d < k/c$, the maps $\pm\xi^\perp \circ \gamma^{-1} \from \sphere^2 \to \sphere^2$ are contractions and each of them has a unique fixed point denoted by $\gamma(p_\pm)$. The points $p_\pm$ are the poles we are looking for. 
As for the vector field $\nu$, it is easily obtained with a partition of unity, provided $\Omega_p$ (resp. $T_pS \cap \Omega_p$) is nonempty for every $p$ in $B$ (resp. in $S \setminus \{p_-,p_+\}$). 

Let $p\in B$. Clearly, the angle between $\xi_p^\perp$ and $ \eta_p^\perp$ satisfies
$$ \angle (\xi_p^\perp, \eta_p^\perp) \le \angle (\xi_p^\perp, \xi_{p_+}^\perp)
 + \angle (\xi_{p_+}^\perp, \xi_{p_-}^\perp) \le 2 \norm{ \xi }_1 d \delta_* $$
where $\delta_*$ denotes the diameter of $S_*$. Thus, for $d < \pi/(2c\delta_*)$, the planes $\xi_p$ and $-\eta_p$ are distinct, and hence $\Omega_p$ is nonempty.

Now let $p \in S \setminus \{p_-,p_+\}$. The plane $T_pS$ is transverse to both $\xi_p$ (by definition of $p_\pm$) and $\eta_p$ (by convexity of $S$), and it is easy to see that $T_pS \cap \Omega_p$ is empty if and only if $\pm\gamma(p)$ belongs to the minimizing geodesic segment of $\sphere^2$ joining $\xi_p^\perp$ to $\eta_p^\perp$. Here we will discuss the case of $\gamma(p)$; for $-\gamma(p)$, replace $p_+$ by $p_-$.

Let $\rho$ be the distance in $B$ between $p$ and $p_+$. On $\sphere^2$, the disk $D$ of radius $c\rho$ centered at $\xi_{p_+}^\perp$ contains $\xi_p^\perp$ but not $\gamma(p)$ if $d < k/c$, for the principal curvatures of $S$ are then greater than $c$. Moreover, since $d < \pi / (2c\delta_*)$, the disk $D$ is geodesically convex: $c\rho \le c d \delta_* < \pi/2$. To conclude, all we need to check is that if $d$ is small enough, $D$ contains $\eta_p^\perp$. This is done below, by showing that
$$ \bigl\lVert \eta \res B \bigr\rVert_1 \le \kappa c $$ 
for some constant $\kappa$ given by the geometry of $S_*$.

First note that the norm of $D\eta^\perp$ at any point $p$ is the inverse of the distance from $p$ to the axis $A$ of the pencil. Actually, in euclidian coordinates in which $A$ is the $z$-axis, the map $\eta^\perp$ is of the form 
$$ (x,y,z) \longmapsto (x^2 + y^2)^{-\frac 1 2} (-y,x,0), $$
 so we can calculate the differential and its norm. 

Now observe that the axis $A$ remains distant from $B$. This is because $B$ contains a euclidian (round) ball $B'$ of radius $d r_*$, where $r_*$ only depends on the geometry of $S_*$. The angle of the sector of the pencil between $P_-$ and
$P_+$ (the affine planes tangent to $S$ at $p_-$ and $p_+$) is bounded above by $c d \delta_*$. The fact that this sector contains $B'$ implies that the distance $l$ from the center of $B'$ to $A$ satisfies $d r_* / l \le \sin (c d \delta_* / 2)$. The desired estimate follows, provided $d$ is sufficiently small.
\end{proof}

\subsection{Triangulation and Key Lemma} \label{ss:triangulation}

The following result, combined with Lemma \ref{l:curv2}, is the key to Lemma \ref
{l:R3} (cf. Introduction of the Appendix) and is an adaptation of Lemma 2.3.4 in \cite{El}. It involves a specific triangulation $\Delta$ of $\R^3$ defined as follows.

The unit cube $[0,1]^3 \subset \R^3$ decomposes into six tetrahedra intersecting along the diagonal from $(0,0,0)$ to $(1,1,1)$. This subdivision of the cube gives rise to an infinite triangulation of $\R^3$ invariant under $\Z^3$, sometimes called \emph{crystalline}, whose vertices are the integer points. We take the first barycentric subdivision of this triangulation (whose simplices have a diameter less than or equal to $\sqrt{3}/2$) and, as in Thurston's ``Jiggling Lemma'' \cite{Th1}, we ``jiggle'' it in a $(2\Z^3)$-periodic way so that any three edges sharing a vertex have linearly independent directions. One can make the jiggling small enough that the diameters of the simplices remain less than $1$.  We denote by $\Delta$ the resulting triangulation. By periodicity, the distance between disjoint simplices of $\Delta$ is bounded below by some positive number $\delta>0$. For any $d>0$, we denote by $d\Delta$ the image of $\Delta$ under a dilation by a factor $d$. This way, we obtain arbitrarily fine triangulations of $\R^3$ whose $3$-simplices are all small (similar) copies of a finite number of model simplices.

We denote by $N_\eps(V)$, $\eps > 0$, the (closed) $\eps$-neighbourhood of a subset $V$ of $\R^3$.

\begin{keylemma} \label{l:local}
Let $U$ be an open subset of $\R^3$, $F$ a closed subset of $U$ and $\xi$ a $K$-plane field on $U$ which is integrable on $(K \times \Op F) \cup (L \times U)$. Given a compact subset $A \subset U$, one can find positive numbers $d_*$, $\mu$ and
$c$ such that, for every $d<d_*$, there exists a $K$-plane field $\xibar$ on $U$ with the following properties: 
\begin{enumerate}
\item
there is a compactly supported homotopy from $\xi$ to $\xibar$ relative to $(K \times \Op N_d(F)) \cup (L \times U)$;
\item
$\xibar$ is integrable on $K \times N_{\mu d}(A_d^2)$ where $A_d$ is a compact polyhedral neighbourhood of $A$ in $d \Delta$ and $A_d^2$ is the $2$-skeleton of $A_d$;
\item
$\norm{ \xibar_t \res{ N_{\mu d} (A_d)}}_1 \le c$ for all $t \in K$. 
\end{enumerate}
\end{keylemma}

\begin{proof}[Proof of Lemma \ref{l:R3} assuming Lemma \ref{l:local}]
Let $U$, $F$, $\xi$ and $A$ be as in Lemma \ref{l:R3}, and let $d_*$, $\mu$ and $c$ be the positive numbers given by Lemma \ref{l:local}. Denote by $\sigma_i$, $1 \le i \le p$, the model $3$-simplices of the triangulation $\Delta$. Each of them contains a strictly convex sphere $S_i$ in the $\mu$-neighbourhood of its boundary. For every $d < d_*$, Lemma \ref{l:local} provides a
$K$-plane field $\xibar$ and a polyhedral neighbourhood $A_d$ of $A$. Every $3$-simplex $\sigma$ of $A_d$ contains a ball $B_{\sigma}$ whose boundary is the image under a dilation of factor $d$ of one of the model spheres $S_i$. Now for every $d$, the plane field $\xibar$ given by Lemma \ref{l:local} satisfies $\Bars{\xibar_t \res{ N_{\mu d} (A_d)}}_1 \le c$ for all $t \in K$. So according to Lemma \ref{l:curv2}, if $d$ is chosen small enough (with respect to the geometry of the model spheres $S_i$) $\xibar$ is almost horizontal on every ball $B_{\sigma}$. If $G\subset U\setminus N_d(F)$ denotes a compact subset such that the support of the deformation from $\xi$ to $\xibar$ is contained in $K\times G$, the $K$-plane field $\xibar$, the neighbourhood $A_* = A_d$ of $A$ and the collection of balls $B$ made of the $B_\sigma$ meeting $A_d\cap G$ (so that $B \subset
(\Int A_*) \setminus F$) satisfy all the properties of Lemma \ref{l:R3}.
\end{proof}

Lemma \ref{l:local} is by far the most technical result in this appendix. Its proof takes up the next two subsections.

\subsection{Deformation model} \label{ss:modele}

Lemma \ref{l:defmod} and its proof describe the properties of the deformation model we will use in the next subsection to make plane fields integrable in a neighbourhood of each simplex of the $2$-skeleton of some subcomplex of the triangulation $d\Delta$, for some small enough $d$. More precisely, the model will be applied after rescaling the plane fields by a factor $1/d$, so our model here deals with plane fields defined near a simplex $\sigma$ of the ``big'' triangulation $\Delta$; no scaling factor $d$ is involved in this subsection. Our construction is directly inspired by that of Eliashberg in Lemma 2.3.2 of \cite{El}, and simply consists in flowing the restriction of the given plane field to some transverse surface under a flow tangent to it, whose orbits cover a neighbourhood of the simplex $\sigma$ (cf. first paragraph of the proof below). (In other words, here, we \emph{untwist} the plane field around a line field tangent to it, while Eliashberg \emph{twists} it to make it contact). But, as we explained in the introduction to the appendix, in the next subsection (proof of the Key Lemma), we will also need a bound on the $C^1$ norm of the resulting plane field in terms of the geometric setting and of the norm of the initial plane field (but not of the plane field itself), and this is actually the main issue of this subsection:\medskip

\begin{remark} \label{r:2.3.3}
Despite Eliashberg's claim in \cite[Note 2.3.3]{El}, the $\CC^1$ norm of the plane field $\xi^1$ given by our deformation model is not controlled by the $\CC^1$ norm of the initial plane field $\xi$ but only by its $\CC^2$ norm. More generally, the $\CC^m$ norm of $\xi^1$ is controlled by the
$\CC^{m+1}$ norm of $\xi$. This ``consumption'' of one derivative, which comes from the ``pull-back'' construction of $\xi^1$, complicates the statement and proof of Lemma \ref{l:defmod} and its application in the next subsection but does not affect the result:  though the number of simplices of $d\Delta$, and thus the number of times one applies the deformation model, grows with $d$, the model is actually applied simultaneously to many simplices, in a finite number of steps (independent of $d$), so knowing that the initial plane field is $\CC^m$-bounded with $m$ sufficiently large (independent of $d$) will give a $\CC^1$ bound on the final plane field regardless of the chosen scaling factor. 
\end{remark}

The statement of Lemma \ref{l:defmod} is already quite elaborate so let us introduce part of the setting beforehand. 
In this subsection, we work in $\R^3$ endowed with the triangulation $\Delta$ and with an affine orthonormal frame $(O,x,y,z)$ which is not necessarily the canonical one. We denote by $V$ the $\delta/2$-neighbourhood of a simplex $\sigma$ of $\Delta$, where $\delta$ is the minimal distance between two disjoint simplices of $\Delta$. We endow $V$ with the horizontal foliation $\eta$ defined by $dz=0$, and the plane fields $\xi$ we deform below satisfy the following condition:
\begin{description}
\item[($*$)]
the angle between the vectors $\xi^\perp$ and $\partial_x$ is everywhere less than some fixed number $\tilde\theta \in (0, \pi/2)$.
\end{description}
In particular, $\xi$ is transverse to $\partial_x$, and \emph{a fortiori} to $\eta$, and the angle between the line field $\xi \cap 
\eta$ and $\partial_y$ is everywhere less than $\tilde\theta$. 

Given a plane field $\xi$, all the deformations of $\xi$ we will define consist in ``straightening'' $\xi$ by (un)rotating it around $\xi \cap \eta$ and have compact support in $\Int V$.  We will thus refer to a plane field as \emph{admissible} if it contains $\xi
\cap \eta$ and coincides with $\xi$ near the boundary $\partial V$.

\begin{lemma} \label{l:defmod}
Let $\xi$ be a plane field on $V$ satisfying Condition $({*})$ and $\norm{\xi}_1<1$, and let $S$ be a properly embedded surface in $V$. Given positive numbers $\mu$ and $\kappa$, assume $S\setminus\partial S$ contains a disk $D$ in $V$ transverse to $\xi \cap \eta$ whose orbit segments under $\xi \cap \eta$ cover the $2 \mu$-neighbourhood of $\sigma$ and whose intersection $D \cap P$ with any leaf $P$ of $\eta$ is a connected curve whose angle with $\xi \cap \eta$ is greater than $\kappa > 0$. Then one can deform $\xi = \xi^0$ by a homotopy $\xi^u$, $u \in [0,1]$, of admissible plane fields satisfying the following properties:
\begin{itemize}
\item
$\xi^1$ coincides with $\xi$ along $D$ and is integrable on the $\mu$-neighbourhood of $\sigma$ ;
\item
$\Bars{ \xi^u }_{m} \le \chi_m \bigl( \Bars{ \xi }_{m+1} \bigr)$ for all $u \in [0,1]$ and all $m \ge 1$, where $\chi_m$ is a polynomial without constant term and with positive coefficients depending only on $\tilde\t$, $\kappa$, $\mu$ and $S$ (but not on $D$).
\end{itemize}
Moreover, the homotopy $\xi^u$ varies continuously with $\xi$.
\end{lemma}

\begin{figure}[htbp]
\centering
\begin{tabular}{ccc}
\includegraphics[height=4cm]{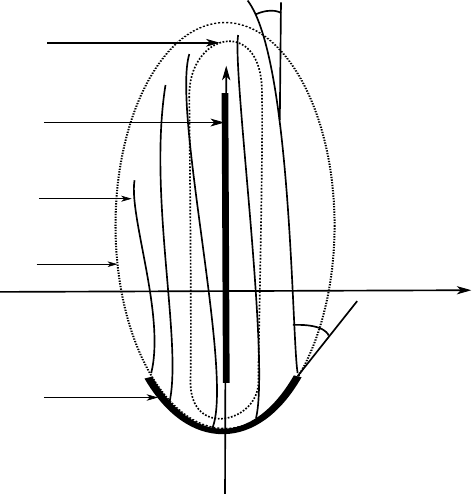} 
\put(-52,117){$\scriptstyle\le\tilde\theta$}
\put(-40,41){$\scriptstyle\ge\kappa$}
\put(-125,102){$\scriptstyle N_{2\mu}(\sigma)$}
\put(-106,84){$\scriptstyle \sigma$}
\put(-116,67){$\scriptstyle \xi\cap \eta$}
\put(-107,50){$\scriptstyle S$}
\put(-107,20){$\scriptstyle D$}
\put(-5,38){$\scriptstyle \partial_x$}
&\hspace{3cm}& \includegraphics[height=4cm]{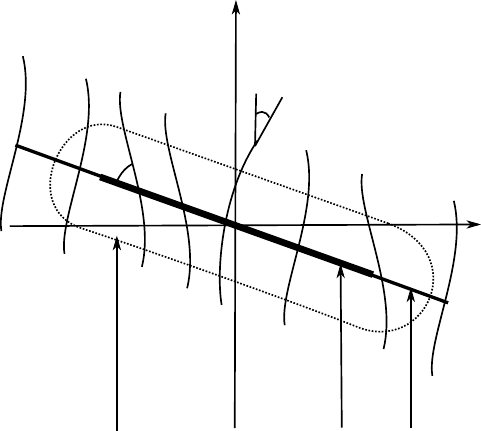} 
\put(-60,88){$\scriptstyle\le\tilde\theta$}
\put(-105,-10){$\scriptstyle N_{2\mu}(\sigma)$}
\put(-105,72){$\scriptstyle \ge \kappa$}
\put(-40,-9){$\scriptstyle \sigma\subset D=S$}
\put(-5,45){$\scriptstyle \partial_x$}
\put(-70,115){$\scriptstyle \partial_y$}
\end{tabular}
\caption{Two situations to which Lemma \ref{l:defmod} will be applied \label{f:5-6}}
\small (near special faces and non-special faces respectively)

Each picture represents the intersections of the objects of 

Lemma  \ref{l:defmod} with a leaf $P$ of the horizontal foliation $\eta$
\end{figure}

\begin{remark}
\label{r:roleofS}
 The role of $S$ is unclear at that stage. The point is that we are going to apply Lemma \ref{l:defmod} to a family of plane fields $\xi_t$, using a different $D_t$ for each $t$, but all these disks will be part of the same surface $S$, so the bound in the second item will be uniform in $t$, since $\chi_m$ depends only on $S$ and not on $D$. 
\end{remark}

\begin{remark}\label{r:poly}
What really matters to us concerning $\chi_m$ is that it is nondecreasing and that $\chi_m(x)/x$ is bounded on any bounded set of $\R_+^*$.  
\end{remark}

\begin{proof} Let us first describe $\xi^u$ geometrically. By condition $(*)$, there exists a nonvanishing vector field 
$\nu$ tangent to $\xi \cap \eta$ (and a unique one if we impose the additional condition $\nu \cdot \partial_y = 1$).
Let $C$ denote the flow ``cylinder'' of $D$ under $\nu$. Note that an integral curve of $\nu$ starting from $D$ cannot return to $D$. Indeed, otherwise, some subarc of this curve intersects $D$ exactly at its endpoints, which necessarily belong to the same leaf $P$ of $\eta$ and are thus connected by an arc of $P\cap D$ by assumption. Then the union of these two arcs forms a simple closed curve in $P\cap V$ (which is convex) and thus bounds a disk in $P\cap V$, which is either positively or negatively stable under $\nu$. Then by a corollary of Poincar\'e-Bendixson's Theorem, this disk must contain a singularity of $\nu$, which is impossible.

Hence, $C$ is an interval bundle over $D$. Now the key observation is that there is a unique integrable plane field $\xibar$ on $C$ containing $\xi \cap \eta$ and coinciding with $\xi$ at every point of $D$: the unique plane field invariant under the holonomy of $\xi \cap \eta$ and equal to $\xi$ along $D$. For $\xi^1$, we will take a plane field coinciding with $\xibar$ on the $\mu$-neighborhood of $\sigma$, with $\xi$ outside a $2\mu$-neighbourdhood of $\sigma$ (which is contained in $C$) and with both on $D$, and $\xi^u$ will be a linear homotopy connecting $\xi$ to $\xi^1$. \medskip

Let us now briefly explain why the control on derivatives of the second part of the statement is natural. One can easily believe that the $C^m$ norm of $\xi^u$, as ``interpolation'' between  $\xi$ and $\xibar$ in a fixed region ($N_{2\mu}(\sigma)\setminus N_\mu(\sigma)$), is bounded in terms of the $C^m$ norms of these two plane fields, so the main object to control is in fact $\xibar$. Now $\xibar$ is obtained from $\xi$ by some kind of pull-back/push-forward construction as follows. Denote by 
$\nu$ the vector field spanning $\xi \cap \eta$ and satisfying $\nu \cdot \partial_y = 1$, by $\phi \from \Omega \subset
\R \times V \to V$ its flow, and by $\tau\from C \to \R$ the map uniquely determined by:
$$\forall p\in C, \;\phi^{-\tau(p)}(p)\in D$$  
(observe that since $\sigma$ is of diameter less than $1$, and $\delta<1$, the condition $dy(\nu) = 1$ implies that $\Omega$ is contained in $[-2,2]\times V$ and that the function $|\tau|$ is bounded by $2$). Then $\xibar$ is defined by:
\begin{equation}
\label{e:xibar}
\xibar(p) = \left(\left(\phi^{\tau(p)}\right)_*\xi\right)(p) = D\phi^{\tau(p)}\left(\phi^{-\tau(p)}(p)\right)\cdot \xi\left(\phi^{-\tau(p)}(p)\right).
\end{equation}
Thus, in short, by composition, the derivatives of $\xibar$ are polynomials in the derivatives of the involved objects, i.e. $\xi$, but also $\tau$ and the flow of $\nu$, whose $C^m$ norms, as we will see, are controlled by that of $\nu$, which are in turn controlled by that of $\xi$.  As a result, since the expression of $\xibar$ involves \emph{the differential} of the flow, the $C^m$ norm of $\xibar$ will be controlled by the $C^{m+1}$ norm of $\xi$.

The uniform polynomial bound on the variations of the flow $\phi$ in terms of its generating vector field $\nu$ (cf. Claims \ref{a:Dmt} and \ref{a:Dm}) is rather natural; the proofs, which we only include for the sake of completeness, rely on Gronwall's Lemma and on general formulas (cf. \eqref{e:faapar} below, for example) for the derivatives of composed maps of several variables (generalizing the so-called \emph{Fa\`a di Bruno Formula} for maps of one variable). The \emph{uniformity} of the control of $\nu$ and $\tau$ on the other hand  (cf. Claim \ref{a:nuxi} and Claim \ref{a:Dmhbar}) depends in a crucial way on the angle bounds $\tilde\theta$ and $\kappa$ and on the geometry of $S$. The calculations carried out in the respective proofs, also based on Fa\`a di Bruno-like formulas, are only intended to clarify this dependency.

Let us now make the previous paragraph more precise. \medskip

First of all, we defined the $C^m$ norm of a plane field as the $C^m$ norm of its Gauss map, so here, rigorously speaking, the derivatives we want to control are those of $\xibar^\perp$, which, from \eqref{e:xibar}, has the following expression:
$$\xibar^\perp(p)=\frac{{}^t\!\!\left(D\phi_p^{-\tau(p)}\right)\cdot\xi^\perp\left(\phi^{-\tau(p)}(p)\right)}{\norm{{}^t\!\!\left(D\phi_p^{-\tau(p)}\right)\cdot\xi^\perp\left(\phi^{-\tau(p)}(p)\right)}}.$$
Let us denote the numerator by $X(p)$ and write $Y(t,p)={}^t(D\phi_p^{t})\cdot\xi^\perp(\phi^t(p))$, so that $X(p)=Y(-\tau(p),p)$. We are going to control the variations of $\nu$, then its flow $\phi^t$, then $Y$, $\tau$, $X$, $\xibar^\perp=\frac{X}{\norm{X}}$ and finally $(\xi^u)^\perp$.

Throughout the calculations, given $m\in\N$, the symbol $\chi_m$ denotes some universal polynomial with positive coefficients depending, as in the lemma, only on $m$, $\tilde\t$, $\mu$, $\kappa$ and $S$, and \emph{which will change in the course of the argument}. When we write $\chi_m^0$ rather than $\chi_m$, we mean that, in addition, $\chi_m$ has no constant term.

\begin{affirmation}\label{a:nuxi}
For all $m \ge 1$, 
$$ \Bars{ \nu }_{m} \le  \chi_m^0 \bigl(\Bars{ \xi }_{m}\bigr). $$
\end{affirmation}
\begin{remark}\label{r:nu1}
For $m=1$, with the assumption $\norm{\xi}_1<1$, this implies: $\norm{\nu}_1\le c\norm{\xi}_1$, for some constant $c$ independent of $\xi$.
\end{remark}
\begin{proof}
If $\xi^\perp=u\partial_x+v\partial_y+w\partial_z$, the maps $u,v,w$ satisfy $u^2+v^2+w^2=1$ and $\frac{v^2+w^2}{u^2}<\tan^2\tilde\theta$ by condition $(*)$, so $\frac1{u^2}=1+\frac{v^2+w^2}{u^2}<1+\tan^2\tilde\theta$. Now $\nu=-\frac v u \partial_x+\partial_y$, so its $C^m$ norm is that of $\frac v u$. Now the $m$-th derivative of $\frac v u$ is a fraction with numerator a (universal) polynomial in the derivatives of $v$ and $u$ of order $k\in[\![0,m]\!]$, each monomial containing a derivative of order at least $1$, $u$ and $v$ are bounded above by $1$, and the denominator is a power of $u$, for which we have a lower bound independent of $\xi$. Hence the required polynomial bound on $\norm{\nu}_m$.
\end{proof}

This leads to a similar bound on the (space) differential of the flow of $\nu$:
\begin{affirmation} \label{a:Dmt}
$ \Bars{ \phi^t }_1 = \Bars{D \phi^t }_0 \le
\chi_1 (\Bars{ \xi }_{1}) $, and for all $m \ge 2$, and all $|t| \le 2$,
$$ \Bars{D^m \phi^t }_0 \le
\chi_m^0 \bigl(\Bars{ \xi }_{m} \bigr) $$
(with $\chi_m$ independent of $t$).
\end{affirmation}
To prove this, we will use the following version of Gronwall's Lemma:

\begin{gronwall} Let $(E,\norm{\cdot})$ be a finite dimensional normed vector space and $x:I\to E$ a differentiable curve satisfying for some positive $a$ and $b$
$$\forall t\in I, \norm{x'(t)}< a\norm{x(t)}+b.$$
Then for all $t_0$ and $t$ in $I$,
$$\norm{x(t)}\le \norm{x(t_0)}e^{a|t-t_0|}+\frac b a \left(e^{a|t-t_0|}-1\right).$$
\end{gronwall}

\begin{proof}[Proof of Claim \ref{a:Dmt}]
We proceed by induction on $m$. For $m=1$, the differential $D\phi^t(p)$ at any point $p$ satisfies the differential equation
\begin{equation}
\label{e:variation1}
\frac d{dt} D\phi^t(p) = D\nu \bigl(\phi^t(p)\bigr) \, D\phi^t(p) 
\end{equation}
with initial condition $D\phi^0(p) = \id$. In particular,
$$\norm{D\phi^0(p)}= 1\quad\text{and}\quad \norm{ \frac d{dt} D\phi^t(p)} \le \norm{\nu}_1 \, \norm{D\phi^t(p)} $$
so if $\norm{\nu}_1>0$, by Gronwall's Lemma with $t_0=0$ and $a=b=\norm{\nu}_1$ we have for all $|t|\le 2$ that
$$\norm{D\phi^t(p)}\le 2e^{2\|\nu\|_1}-1.$$
Let $c$ be the constant given by Remark \ref{r:nu1} and $C$ such that $e^x-1\le Cx$ for all $x\le 2c$. Since $2\norm{\nu}_1\le 2c\norm{\xi}_1\le 2c$, 
$$2e^{2\|\nu\|_1}-1 = 1+2(e^{2\|\nu\|_1}-1)\le 1+4C\|\nu\|_1\le 1+4Cc\|\xi\|_1,$$
which gives the desired bound on $\norm{D\phi^t}_0$ (if $\norm{\nu}_1=0$, $\norm{D\phi^t(p)}\equiv1$).

Now let $m\ge 2$ and assume Claim \ref{a:Dmt} has been proved for every $k\le m-1$. For all $(t,p)$ where it makes sense,
\begin{equation}
\label{e:variationm}
 \frac d {dt} D^{m}\phi^t(p) = D^m(\nu\circ \phi^t)(p)
\end{equation}
with initial condition $D^{m}\phi^0(p) = 0$. Given a smooth function $f$ on an open subset of $\R^3$ and a multi-index $I=(i_1,\dots,i_k)\in \{1,2,3\}^k$, $1\le k \le m$, we denote by $\partial_If$ the partial derivative $\frac{\partial^kf}{\partial x_{i_1}\dots\partial x_{i_k}}$. Given a subvector $J=(i_{n_1},...,i_{n_l})$ of $I$, i.e. given a subset $B=\{n_1,...,n_l\}$ of $\{1,...,k\}$ with $n_1<...<n_l$, we will abusively write $\partial_Bf$ instead of $\partial_Jf$. By induction, one gets the following formula for partial derivatives of the composed map $\nu\circ\phi^t$:
\begin{equation}
\label{e:faapar}
\partial_I(\nu\circ\phi^t)(p) = \sum_{\pi\in\Pi_k}D^{|\pi|}\nu(\phi^t(p))\cdot \left(\prod_{B\in\pi}\partial_B\phi^t(p)\right)
\end{equation}
where $\Pi_k$ denotes the set of partitions $\pi$ of $\{1,...,k\}$ (recall $k$ is the length of $I$ here), and $|\pi|$ the number of ``blocks'' of such a partition. If the blocks of $\pi$ are $B_1$, ..., $B_k$, the parenthesis $\left(\prod_{B\in\pi}\partial_B\phi^t(p)\right)$ must be understood as the $k$-tuple of vectors $\partial_{B_1}\phi^t(p)$,..., $\partial_{B_k}\phi^t(p)$ to which $D^{k}\nu(\phi^t(p))$ is applied. If $I$ is of size $m$, isolating  the partition $\pi$ with one block of size $m$, we get:
$$\partial_I(\nu\circ\phi^t)(p) = D\nu(\phi(p))\cdot \partial_I\phi^t(p)+\sum_{{\pi\in\Pi_m}\atop{|\pi|\ge 2}}D^{|\pi|}\nu(\phi^t(p))\cdot \left(\prod_{B\in\pi}\partial_B\phi^t(p)\right)$$
and thus
\begin{align*}
\norm{\partial_I(\nu\circ\phi^t)(p)}\le \|\nu\|_{1} \norm{\partial_I\phi^t(p)} + \sum_{{\pi\in\Pi_m}\atop{|\pi|\ge 2}}\|\nu\|_{|\pi|}\prod_{B\in\pi}\norm{D^{|B|}\phi^t}
\end{align*}
where the last term is a $\chi_m^0(\|\xi\|_{m})$ by induction and Claim \ref{a:nuxi}. So according to \eqref{e:variationm},
$$\norm{ \frac d {dt} \partial_I\phi^t(p) }\le  \|\nu\|_{1}  \norm{\partial_I\phi^t(p)}+ \chi_m^0(\|\xi\|_{m})$$
and once again we conclude using Gronwall's Lemma (and the fact that $\partial_I\phi^0(p)=0$ for all $p$).
\end{proof}

Recall  that $Y$ is defined by $Y(t,p)={}^tD\phi^t(p)\cdot \xi^\perp (\phi^t(p))$. 

\begin{affirmation}\label{a:Dm}
For all $m\ge 1$, $0\le k\le m$, and every multi-index $I=(i_1,\dots,i_{m-k})\in \{1,2,3\}^{m-k}$,
$$  \norm{\partial_I(\partial_t)^k Y }_0 \le
\chi_m^0 \bigl(\Bars{ \xi }_{m+1} \bigr) . $$
\end{affirmation}

\begin{proof} This follows easily from Claims \ref{a:nuxi} and \ref{a:Dmt}, by product and composition. Let us explain how, starting with the time derivatives:
\begin{align*}
\partial_t Y(t,.) &= \partial_t({}^tD\phi^t)\cdot ( \xi^\perp \circ\phi^t) + {}^tD\phi^t\cdot \partial_t( \xi^\perp \circ\phi^t)\\
&={}^t(D\phi^t) ({}^tD\nu\circ\phi^t)\cdot ( \xi^\perp \circ\phi^t)+ {}^t(D\phi^t)\cdot (D \xi^\perp \circ\phi^t)\cdot (\nu\circ\phi^t)\\
&={}^t(D\phi^t)({}^tD\nu\cdot \xi^\perp+ D \xi^\perp \cdot \nu)\circ\phi^t\\
&={}^t(D\phi^t)(\nu *  \xi^\perp)\circ\phi^t
\end{align*}
where $\nu *$ denotes the Lie derivative-like differential operator $X\mapsto {}^tD\nu\cdot X+ D X \cdot \nu$. By induction,
\begin{align*}
(\partial_t)^k Y(t,.) ={}^t(D\phi^t)\cdot((\nu *)^k  \xi^\perp)\circ\phi^t
\end{align*}
There is a general polynomial formula for $(\nu *)^k  \xi^\perp$ in terms of the derivatives of $\nu$ and $\xi^\perp$ of order $l\in[\![0,k]\!]$, each monomial containing a \emph{real} derivative (\emph{i.e} of order at least one). Now Formula \eqref{e:faapar} applied to $(\nu *)^k  \xi^\perp$ instead of $\nu$ gives a polynomial expression for any partial derivative of order $l$ of $((\nu *)^k  \xi^\perp)\circ\phi^t$ in terms of that of $(\nu *)^k  \xi^\perp$ and $\phi^t$ of order $\le l$, so in terms of  the derivatives of $\nu$ and $\xi^\perp$ of order $\le k+l$ and that of $\phi^t$ of order $\le l$. So in the end, any partial derivative $\partial_I(\partial_t)^k Y$ of order $l$ of the product $(\partial_t)^k Y={}^t(D\phi^t)\cdot((\nu *)^k  \xi^\perp)\circ\phi^t$ is given by a general polynomial formula in terms of the derivatives of $\nu$ and $\xi^\perp$ of order $\le k+l$ and that of $\phi^t$ of order $\le l+1$ (again, each monomial containing a \emph{real} derivative of $\nu$ or $\xi^\perp$), and we conclude using Claims \ref{a:nuxi} and \ref{a:Dmt}.
\end{proof}

\begin{affirmation}\label{a:Dmhbar}
For every $m\ge 1$ and every multi-index $I=(i_1,\dots,i_{m})\in \{1,2,3\}^{m}$,
$$  \norm{\partial_I\tau }_0 \le
\chi_m \bigl(\Bars{ \xi }_{m+1} \bigr) . $$
\end{affirmation}
Here $\norm{\partial_I\tau }_0$ simply means $\sup_{p\in V}|\partial_I\tau (p)|$. 
\begin{proof}
Since $\tau \circ \phi^t = \tau+ t$, given Claim \ref{a:Dmt}, we only need to estimate the derivatives of $\tau$ along $D$. To that end, we need to understand the relation between $\tau$ and the geometry of $D$. 

Therefore, let us introduce the function $\tau_0$, defined in a neighbourhood of $D$, whose restriction to every plane
$P$ of $\eta$ is the euclidean distance to $S\cap P$ multiplied by the sign of $\tau$. In other words $\tau_0$ is the algebraic distance to $S\cap P$ --~where $S$ is cooriented so that $\tau$ and $\tau_0$ have the same sign~-- and is thus smooth. Moreover, it depends only on the geometric setting (not on $\xi$). The idea is to compare $\tau$ to $\tau_0$ and to deduce a bound on the first from one on the latter. To that end, we consider the unique multiple $\nu_0=f\nu$ of $\nu$ near $D$ whose flow $\phi_0^t$ satisfies:
$$\forall  p, \phi_0^{-\tau_0(p)}(p)\in D.$$
Such a vector field must satisfy $\tau_0\circ \phi_0^t=\tau_0+t$ which, differentiating with respect to $t$, gives $f= \frac{1}{\nu \cdot \tau_0}$. The flows $\phi_0^t$ and $\phi^t$ satisfy the relation 
$$\phi_0^t(p) = \phi^{s(t,p)}(p)$$
where the function $s$ satisfies $s(0,p) = 0$ for all $p$ and the differential equation
\begin{equation}\label{e:edo-s}
\frac d {dt} s(t,p) = f \left(\phi^{s(t,p)}(p)\right).
\end{equation}
Since
$$\phi_0^{-\tau_0(p)}(p) = \phi^{-\tau(p)}(p) = \phi^{s(-\tau_0(p),p)}(p),$$
we have $-\tau(p) = s(-\tau_0(p),p)$. Since $s(0,p) = 0$ for all $p$ close to $D$ in $C$, the spatial derivatives $\partial_Is(0,p)$ are all zero. As a consequence, for $p \in D$, the general formula (which we will omit here) expressing the derivatives of $\tau$ in terms of that of $s$ and $\tau_0$ becomes
\begin{equation}
\label{e:dDmh}-\partial_I\tau(p) = \sum_{\pi \in \Pi_k} (\partial_t)^{|\pi|} s(0,p)
\left(\prod_{B \in \pi}  \partial_{B} \tau_0(p) \right),
\end{equation}
where $k$ is the length of the multi-index $I$. The quantities $\Bars{\partial_{B} \tau_0}_0$ depend only on the geometry of $S$. We are now going to control $(\partial_t)^{k+1} s(0,.)$ by induction.  For $k=0$, according to Equation \eqref{e:edo-s}, we need a uniform bound on $f= \frac 1{\partial_\nu \tau_0}$. But $\partial_\nu \tau_0(p)$, for all $p \in D$, is the scalar product of $\nu(p)$ with the unit normal vector to $S \cap P$ in $P$, where $P$ is the horizontal plane containing $p$. The function $\partial_\nu \tau_0$ is thus bounded below along $D$ by some constant depending only on $\kappa$ and $\tilde\t$. 

Now for $k\ge 1$, differentiating Equation \eqref{e:edo-s}, one gets, for all $p \in D$,
\begin{equation}\label{e:dks}
(\partial _t )^{k+1}s(0,p) = \sum_{\pi \in \Pi_k} \left(
\partial_{\nu}^{|\pi|} f(p)\right) \prod_{B \in \pi} \left(
\partial_t^{|B|} s(0,p)\right),
\end{equation}
where $\partial_{\nu}$ denotes the derivative in the direction of $\nu$ (in other words, $\partial_\nu f = \nu \cdot f$). We saw that $\frac 1{\partial_\nu \tau_0}$ is bounded above independently of $\xi$. Moreover, every quantity $\Bars{\partial_\nu^{l+1} \tau_0}_0$ is bounded above by a function of $\Bars{\nu}_{l}$ (which depends only on $S$). Thus, every quantity $\Bars{\partial_\nu^{|\pi|} f \res D}_0$ is itself controlled by $\Bars{\nu}_{|\pi|}$, and Relation \eqref{e:dks} shows by induction that the quantities $|\partial^{k}_t s(0,p)|$ are controlled by $\Bars{\nu}_{k}$. Formula \eqref{e:dDmh} and Claim \ref{a:nuxi} thus imply Claim \ref{a:Dmhbar}.
\end{proof}

\begin{affirmation}\label{a:Dmbar}
For all $m\ge 1$,
$$ \norm{ X }_m \le
   \chi_m^0 \bigl(\norm{ \xi }_{m+1} \bigr). $$
\end{affirmation}

\begin{proof} Since $X=Y\circ(-\tau,\id)$, this follows by composition from Claims \ref{a:Dm} and \ref{a:Dmhbar}.   
\end{proof}

Since $\xibar^\perp=X/\|X\|$, in order to deduce that $\norm{\xibar}_{m} \le \chi_m^0(\Bars{ \xi}_{m+1} )$, we should still check that $\|X\|$ is bounded below (independently of $\xi$). We will simply say here that this follows from the fact that $\xi^\perp$ is unitary and that $D\phi_p^{-\tau(p)}$ is close enough to the identity in our setting.
\medskip

Now let $\rho : V \to [0,1]$ be a function equal to $1$ on $N_{\mu}(\sigma)$, with support in $N_{2 \mu}(\sigma)$. For all $u \in [0,1]$, define $\xi^u$ by
$$(\xi^u)^\perp =(1 - u \rho) \xi^\perp + u \rho \xibar^\perp.$$
Observe that these vector fields are all nonsingular, for the $\partial_x$ components of $ \xi^\perp$ and  $\xibar^\perp$ are both positive. The plane fields $\xi^u$ have all the desired properties. Let us simply stress that the bound on their derivatives is independent of the particular $\sigma$ under scrutiny. This is due to the $(2\Z)^3$-invariance of $\Delta$: all simplices are copies of a finite number of model ones, and we only need one cut-off function per isometry class of simplices (note that $\sigma$ plays no role in the rest of the proof). 
\end{proof}

\begin{remark}\label{r:relatif}
Note that, if the plane field $\xi$ is already integrable on a region of the form
$C' = \{ \phi^t(p), \; p \in D', \; t \in [a(p),b(p)] \}$ for some domain
$D' \subset D$ and some functions $a, b \from D' \to \R$ satisfying $a \le 0 \le b$,
then the homotopy $\xi^u$ of Lemma \ref{l:defmod} is stationary on $C'$.
\end{remark}

\subsection{Proof of the Key Lemma \ref{l:local}}
\label{ss:cle}

We start with the data $U$, $F$, $A$ and $\xi$ of Lemma \ref{l:local} and use the notations of Section \ref{ss:triangulation}. As explained in the introduction to the appendix, the Key Lemma is proved by induction on the successive skeleta of some triangulation of the parameter space $K$. This induction is formalized in Lemma \ref{l:recurrence} below (cf. subsection \ref{sss:induction}), whose proof takes up the last subsection of the article. Before stating this Lemma, we need to prepare its setting, \emph{i.e.}:
\begin{itemize}
\item to define, for any scaling factor $d$ (less than some $d_0$ defined below), polyhedral neighbourhoods $A_d$ and $F_d$ of $A$ and $F$, the support of the future deformation(s) being contained in $A_d$ and disjoint from $F_d$ (cf. subsection \ref{sss:polyhedral});
\item to fix the triangulation of the parameter space (cf \ref{sss:parameter});
\item to define special simplices precisely (cf. \ref{sss:special}).
\end{itemize}
Recall that $\Delta$ is the $(2\Z)^3$-invariant triangulation of $\R^3$ obtained by ``jiggling'' the barycentric subdivision of the crystalline triangulation with integer vertices. By periodicity of the construction, $\Delta$ has a finite number of \emph{model} simplices, meaning that every simplex of $\Delta$ is the image of one of those by a translation. The diameter of the simplices of $\Delta$ is less than $1$. Moreover, the distances between two disjoint simplices and the angles between two intersecting
$1$~or~$2$-simplices (not contained into one another) are uniformly bounded below by numbers denoted by $\delta > 0$ and $\gamma \in (0, \pi/2]$ respectively (the angle between a straight line and a plane is the angle between the straight line and its orthogonal projection on the plane).

Fix an angle $\theta < \gamma/2$. 

\subsubsection{Polyhedral neighbourhoods}
\label{sss:polyhedral}

Given $d>0$, we will still (improperly) call every subcomplex coming from a cube which has been subdivided, ``jiggled'' and scaled, a ``cube'' of $d\Delta$. Since $A$ is a compact subset of $U$, for $d_0 > 0$ sufficiently small, $N_{2d_0}(A)$ is contained in $U$ and $\xi$ is integrable on $K \times N_{2d_0} (F\cap A)$. Fix such a $d_0$. 

Given $d < d_0/4$, we denote by $A_d$ and $F_d$ the subcomplexes of $d\Delta$ made up of all the ``cubes'' meeting $N_{d_0} (A)$ and $N_{d_0} (F \cap A)$ respectively. Thus, since $d_0 + 3d < 2d_0$,
\begin{align*}
   N_{d_0} (A) \subset A_d & \subset N_d(A_d) \subset N_{2d_0} (A) \\
   \llap{\text{and}\quad} N_{d_0} (F \cap A) \subset F_d & \subset N_{d}(F_d)
   \subset N_{2d_0} (F \cap A)
\end{align*}
so, in particular, $\xi$ is integrable on $K \times N_{d}(F_d)$.

\begin{remark} \label{r:arete} 
The combinatorial structure of $\Delta$ will be important when it comes to perturb the plane fields near (special) simplices whose boundary meets $F_d$, where the homotopy must be stationary. More precisely, we will need to know that every $2$-simplex of $A_d$ not contained in $F_d$ has at most one edge in $F_d$. Indeed, let $\sigma$ be such a $2$-simplex and $Q$ the cube of $A_d$ containing it. By assumption, this cube is not contained in $F_d$, so $\sigma \cap F_d \subset \sigma \cap \partial Q$. Since the triangulation $\Delta$ is obtained by barycentric subdivision, there are two cases:
\begin{itemize}
\item 
if $\sigma$ has a vertex in the interior of $Q$, it has at most one edge in $\partial Q$ ;
\item 
otherwise, $\sigma \subset \partial Q$ has a vertex $q$ in the interior of some ``square face'' of $Q$; either $q \in F_d$ and then $\sigma \subset F_d$ (for $F_d \cap Q$ is a union of ``square faces''), or $q \notin F_d$ and then $\sigma$ has at most one edge in $F_d$.  
\end{itemize} 
\end{remark}

\subsubsection{Subdivision of the parameter space}
\label{sss:parameter}

We fix a subdivision of the parameter space $K$ compatible with $L$ and so fine that the following inequality holds on every simplex $K_*$:
\begin{equation} \label{e:vp} \tag{$\ma_0$}
   \angle( \xi_s(p), \xi_t(p) ) < \frac \theta {16} \quad 
   \text{for all $s, \, t \in K_*$ and $p \in N_{2d_0} (A)$.}
\end{equation}
For $0 \le i \le n$, where $n = \dim K$, denote by $K^i$ the union of the $i$-skeleton of the triangulation with the subcomplex $L$. We also write $K^{-1} = L$.

\subsubsection{Special simplices}
\label{sss:special}

Here, $d$ is any positive number less than $d_0$.

\begin{definition} \label{d:special}
Given a simplex $K_*$ of $K$ and a $K_*$-plane field $\xi^*$ on $U$, we will call a $2$-simplex $\sigma$ of $A_d$ \emph{special} (for $\xi^*$) if it is not contained in $F_d$ and if there exists $(s, q) \in K_* \times \sigma$ such that
$\angle( \sigma, \xi^*_s(q) ) < \theta/2$.
\end{definition}

\begin{affirmation} \label{a:speciaux}
If $\xi^*$ satisfies: 
\begin{equation*}
\forall (s,t,p)\in(K_*)^2\times U, \quad\angle(\xi^*_s(p),\xi^*_t(p))<\frac\theta8
\end{equation*}
and
\begin{equation*}
\forall (t,p,q)\in K_*\times (A_d)^2 \text{ s.t. } |p-q|<2d, \quad \angle(\xi^*_t(p),\xi^*_t(q))<\frac\theta8,
\end{equation*}
then for any special simplex $\sigma$, 
\begin{equation*} 
\angle( \sigma, \xi^*_t(p) ) < \theta \quad 
\text{for all $(t,p) \in K_* \times N_d (\sigma)$.}
\end{equation*}
In particular, special simplices are disjoint.
\end{affirmation}

\begin{proof}[Proof]
If $(s,q) \in K_* \times \sigma$ is such that $\angle( \sigma, \xi^*_s(q) ) < 
\theta/2$, then for all $(t,p) \in K_* \times N_d (\sigma)$,
\begin{align*} 
   \angle( \sigma, \xi^*_t(p) )
 & \le \angle( \sigma, \xi^*_s(q) ) + \angle( \xi^*_s(q), \xi^*_s(p) ) 
   + \angle( \xi^*_s(p), \xi^*_t(p) ) \\
 & < \frac \theta 2 + \frac \theta 8  + \frac \theta 8 <\theta.
\end{align*}
Now assume that $\sigma$ and $\sigma'$ are non disjoint special simplices. Let $s, s' \in K_*$ and $q \in \sigma$, $q' \in \sigma'$ be such that
$$ \angle( \sigma, \xi^*_s(q) ) < \frac \theta 2 \quad \text{and} \quad
   \angle( \sigma', \xi^*_{s'} (q') ) < \frac \theta 2 . $$
Then
\begin{align*}
\angle( \sigma, \sigma') 
& \le \angle( \sigma, \xi^*_s(q) ) + \angle( \xi^*_s(q), \xi^*_s(q') ) 
 + \angle( \xi^*_s(q'), \xi^*_{s'}(q') ) + \angle( \xi^*_{s'}(q'), \sigma' ) \\
& < \frac \theta2 +  \frac \theta 8 + \frac \theta 8 + \frac
\theta 2 < 2\theta < \gamma .
\end{align*}
By definition of $\gamma$, the simplices $\sigma$ and $\sigma'$ coincide.
\end{proof}

\subsubsection{Induction Lemma}\label{sss:induction}

Given a plane field $\xi$, we denote by $\|\xi\|_m$ the norm of its restriction to $N_{2d_0}(A)$ and for all $d>0$, we define
\begin{equation*}
\norm{\xi}_{d,m} =\norm{(h_d)^*\xi_{|N_{2d_0}(A)}}_m
\end{equation*} 
where $h_d$ denotes any homothety of factor $d$. Note that 
$$((h_d)^*\xi)^\perp = d\cdot (h_d)^*(\xi^\perp)$$
so
\begin{equation*}
\norm{\xi}_{d,m} = d\norm{(h_d)^*(\xi^\perp)}_m 
\end{equation*} 
and in particular 
\begin{equation*}
\norm{\xi}_{d,1} =d \norm{\xi}_1 \quad\text{and}\quad \norm{\xi}_{d,m} \le d \norm{\xi}_m\quad\forall m\ge 1.
\end{equation*} 

The Key Lemma \ref{l:local} is a consequence of the following result. 

\begin{lemma} \label{l:recurrence}
For every $0 \le i \le n+1$, there are positive numbers $d_i$, $\mu_i$ and $(c_{i,m})_{m \ge 1}$ such that, for every $d < d_i$, there exists a compactly supported homotopy $\xi^u$, $u \in [0,i]$, of $K^{i-1}$-plane fields on $U$ with the following properties
\begin{itemize}
\item
$\xi^0$ coincides with $\xi$ (or more accurately with its restriction to $K^{i-1}
\times U$) and the homotopy is relative to $(K^{i-1} \times \Op N_{\mu_id} (F_d)) 
\cup (L \times U)$;
\item
$\xi^i$ is integrable on $K^{i-1} \times N_{\mu_id} (A_d^2)$;
\item
for every $(t, u) \in K^{i-1} \times [0,i]$, 
\begin{equation} \label{e:vsi} \tag{$\dagger_i$}
\lrBars{ \xi_t^u }_{d,m} \le c_{i,m} d\; ;
\end{equation}
\item for every $(t, u) \in K^{i-1} \times [0,i]$ and every $p\in U$,
\begin{equation*}
\angle( \xi_t^u(p), \xi_t(p) ) < \frac \theta {32}.
\end{equation*}
\end{itemize}
\end{lemma}

\begin{proof}[Proof of the Key Lemma \ref{l:local}]
For $i = n+1$, the above lemma implies the Key Lemma with $d_* = d_{n+1}$, $\mu = \mu_{n+1}$, and $c = c_{n+1,1}$.
\end{proof}

\subsubsection{Proof of Lemma \ref{l:recurrence}}

We proceed by induction. Step $i = 0$ is trivial (with $\mu_0 = 1$) since the plane fields $\xi_t$, $t \in K^{-1} = L$, are integrable on all of $U$ and uniformly $C^m$-bounded on $N_{2d_0}(A)$ for every $m$. 

Assume now that step $i \ge 0$ has been completed and let us describe the structure of step $i+1$ (each of the following sentences will be detailed afterwards). First, given \emph{any} $d<d_i$, we take the homotopy of $K^{i-1}$-plane fields given by step $i$ and extend it to a homotopy of $K^i$-plane fields stationary outside a neighbourhood of $K^{i-1}$. Taking $d_i$ smaller if necessary, we arrange that for any $d<d_i$, the restriction of the resulting $\xi^i$ to any simplex $K_*$ of $K^i$ satisfies the hypothesis of Claim \ref{a:speciaux}, so that, for any $K_*$, special simplices are disjoint. 

Then we build a homotopy $\xi^{i+u}$, $u \in [0,1]$, of $K^i$-plane fields relative to $K^{i-1} \times U$. This is done $i$-simplex by $i$-simplex of $K^i$ independently, applying Lemma \ref{l:defmod} to a neighbourhood of each simplex of the $2$-skeleton of $A_d$ not contained in $F_d$, after a rescaling by a factor $1/d$, and taking these simplices in a suitable order: first special simplices, then vertices contained in no such simplex, then edges, and finally non-special faces. For each $i$-simplex of $K^i$, each of these four substeps yields ``one quarter'' of the desired homotopy from $\xi^i$ to $\xi^{i+1}$, namely $\xi^{i+u}$, $u \in [0,1/4]$, $[1/4,1/2]$, $[1/2,3/4]$, and $[3/4,1]$. 

More precisely, these four substeps follow the same pattern: we first define coordinates, constants ($\mu$, $\kappa$) and surfaces ($S$, $D$) to which Lemma \ref{l:defmod} and Remark \ref{r:relatif} (to make the deformation relative to the already foliated region)   are applicable (provided $d$ is small enough). Then (taking $d$ even smaller if necessary), we use the quantitative part of Lemma \ref{l:defmod} to bound the variations of the resulting plane fields (in space and time), to ensure both the applicability of the deformation model in the next substep and the disjointness of special simplices in the next inductive step ($i+2$). In particular, the set of $d$'s for which the construction can actually be carried on decreases at each substep.

This similarity of pattern makes this last subsection somewhat repetitive, and generates references to a large number of very similar inequalities. The cases of special simplices and vertices are probably enough for the reader to get the whole picture. There are no new ideas in the rest; we mainly include it for the sake of completeness and to show how Remark \ref{r:relatif} is used in each particular case to guarantee the global coherence of the construction.

Before detailing these four substeps, let us be a little bit more specific about the $K^i$-plane field $\xi^i$ defined after step $i$. Taking $d_i$ smaller if necessary, we assume that
\begin{equation}\label{e:di}\tag{$\diamond_i$}
d_i c_{i,1} < \frac{\t}{16}.
\end{equation}
Given any $d<d_i$, we take a homotopy $\xi^u$, $u \in [0,i]$, of $K^{i-1}$-plane fields given by step $i$ and more specifically satisfying for all $(t, u) \in K^{i-1} \times [0,i]$ and all $p\in U$
\begin{equation}\label{e:vhi}\tag{$\ddagger_i$}
\angle( \xi_t^u(p), \xi_t(p) ) < \frac \theta {32} - \beta \quad \text{with $\beta>0$. }
\end{equation}
By the homotopy extension property, we can extend this homotopy to $K^i$. According to \eqref{e:vp} and \eqref{e:vhi}, for all $s, t$ in the same simplex $K_*$ of $K^i$ and all $p \in U$,
\begin{equation} \label{e:vpi}\tag{$\ma_i$}
\angle (\xi^i_t(p), \xi^i_s(p)) < \frac \theta {8}.
\end{equation}
Furthermore, according to \eqref{e:vsi} and \eqref{e:di}, for all $t\in K^i$ and all $(p,q)\in(A_d)^2$ such that $|p-q|<2d$,
\begin{align*}
 \angle(\xi^i_t(p),\xi^i_t(q))\le  \lrBars{ \xi^i_t }_1 |q-p| \le \frac \theta {16 d_i} \times 2d  < \frac{ \theta }{ 8 }.
\end{align*}
So $\xi^i$ satisfies the hypothesis of Claim \ref{a:speciaux}. Thus, for any $K_*$, special simplices (for $\xi^i$) are disjoint. \medskip

From now on we fix an $i$-simplex $K_*$ (not contained in $L$).

\subsubsection*{Deformation near special simplices}
We first define the coordinates in which we want to apply the deformation model (after rescaling), then Claim \ref{a:Dt} specifies the setting to which we can actually apply it, and Claim \ref{a:D't} makes sure the deformation can be made relative to $F_d$. These two assertions are fairly straightforward for constant plane fields, and the point is that the more we shrink the scaling factor $d$, the more the rescaled plane fields are close to being constant. Finally we keep track of the amplitude of the deformation to make sure that, at the end of step $i+1$, the new plane fields still satisfy the angle conditions necessary for their special simplices to be disjoint. \medskip

Let $\sigma_d$ be a special simplex (for $\xi^i$), $V_d$ its $d\delta/2$-neighbourhood and $\alpha_d$ any edge of $\sigma_d$ if $\sigma_d \cap F_d=\varnothing$ and the (single) edge $\sigma_d \cap F_d$ otherwise (cf. Remark \ref{r:arete}). We now choose adapted coordinates in the following way: the origin is the midpoint $q$ of $\alpha_d$, the vector $\partial_y(q)$ is tangent to $\sigma_d$ and points to the vertex opposite $\alpha_d$, and the vector $\partial_x(q)$ is orthogonal to $\sigma_d$. 
\begin{figure}[htbp]
\centering
\includegraphics[width=5cm,height=3cm]{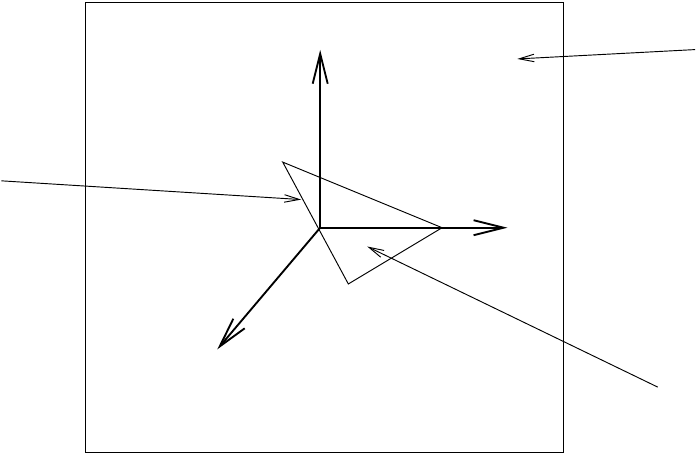}
\put(-5,8){$\scriptstyle \sigma_d$}
\put(-80,32){$\scriptstyle q$}
\put(-105,15){$\scriptstyle \partial_x$}
\put(-38,38){$\scriptstyle \partial_y$}
\put(-80,78){$\scriptstyle \partial_z$}
\put(-147,50){$\scriptstyle \alpha_d$}
\put(3,75){$\scriptstyle \text{plane spanned by } \sigma_d$}
\label{fig:speciaux}
\caption{Choice of coordinates near a special simplex $\sigma_d$}
\end{figure}

Now denote by $h_d$ the homothety of ratio $d$ and center $q$, by $V$, $\sigma$, $\alpha$ and $\bar F$ the inverse images of  $V_d$, $\sigma_d$, $\alpha_d$ and $F_d$ under $h_d$, and by $\zeta_t^i$, $t \in K_*$, the pull-back $(h_d)^*\xi^i_t$ defined on $V$. With the above choices, $\zeta^i_t$, $t \in K_*$, satisfies condition ($*$) of Section \ref{ss:modele} with $\tilde\theta = \theta$ since, for all $(t,p) \in K_* \times N_1 (\sigma) \supset K_* \times V$, 
\begin{equation*} \label{e:transverse2}
\angle({\zeta_t^i}^\perp(p), \partial_x(p))=\angle({\xi_t^i}^\perp(h_d(p)), \partial_x(h_d(p))) = \angle ( \xi^i_t(h_d(p)), \sigma) < \theta
\end{equation*}
according to Claim \ref{a:speciaux}. In addition, according to \eqref{e:vsi} and \eqref{e:di},
$$\norm{\zeta_i^t}_1 = \norm{\xi^i_t}_{d,1}\le d_i c_{i,1} <\frac\theta{16}<1.$$
As in Section \ref{ss:modele}, $\eta$ denotes the plane field defined by $dz=0$. By the induction hypothesis, the plane field $\xi^i_t$ is integrable on $N_{\mu_i d}(\sigma_d)$ for $t \in \partial K_*$ and on $N_{\mu_i d}(F_d)$ for $t \in K_*$, so  $\zeta^i_t$ is integrable on $N_{\mu_i}(\sigma)$ for $t \in \partial K_*$ and on $N_{\mu_i}(\bar F)$ for $t \in K_*$. We denote by $S$ the smooth boundary of some stricly convex domain containing $N_{9 \mu_i /10}(\sigma)$ and contained in $N_{\mu_i }(\sigma)$.

\begin{affirmation}\label{a:Dt}
There are positive numbers $d_{i+1/4}$, $\mu$ and $\kappa$ such that if $d < d_{i+1/4}$, for all $t\in K_*$, the surface $S$ contains a disk $D_t$ varying continuously with $t$ and satisfying the following properties: 
\begin{itemize}
\item $D_t$ is transverse to $\zeta^i_t \cap \eta$ and its orbit segments under $\zeta^i_t \cap \eta$ cover the $2\mu$-neighbourhood of $\sigma$ ; 
\item $D_t \cap P$ is connected for every leaf $P$ of $\eta$ and its angle with $\zeta^i_t \cap \eta$ is at least $\kappa$.
\end{itemize}
\end{affirmation}

\begin{remark}\label{r:const} The constants $d_{i+1/4}$, $\mu$ and $\kappa$ depend only on $\theta$ and the geometry of $S$, \emph{i.e.} on $\theta$, $\mu_i$ and the geometry of the model simplices of $\Delta$. They do not depend on the specific simplex $\sigma$ itself.
\end{remark}

\begin{figure}[htbp]
\centering
\includegraphics[height=3.3cm]{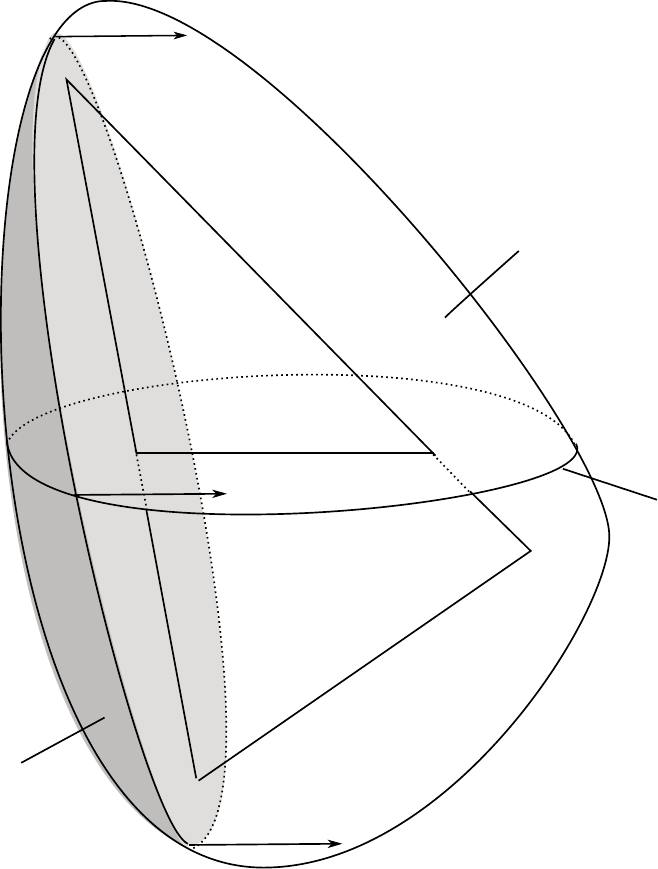}
\put(-80,8){$\scriptstyle D_t$}
\put(-59,84){$\scriptstyle \nu_t$}
\put(2,36){$\scriptstyle S\cap P$}
\put(-13,67){$\scriptstyle S$}
\put(-35,30){$\scriptstyle \sigma$}
\label{fig:Dt3d}
\end{figure}

\begin{proof}
For every $t\in K_*$, denote by $\bar\zeta^i_t$ the constant plane field equal to $\zeta^i_t(q)$ on $V$. Then $\bar\zeta^i_t\cap \eta$ is tangent to $S$ along a simple closed curve which divides $S$ into an ``entrance face''  $S_t^-$ and an ``exit face'' $S_t^+$, which depend continuously on $t$. The flow lines of $\bar\zeta^i_t\cap \eta$ through $S_t^-$  cover the whole domain $B$ bounded by $S$ and hence $N_{9 \mu_i /10}(\sigma)$. So if one defines $D_t$ as $S_t^-$ with an $\eps$-neighborhood of its boundary removed, for some small fixed $\eps$ (independent of $t$, $\sigma$...), the flow lines through $D_t$ still cover a $2\bar\mu$-neighbourhood of $\sigma$ (for some $\bar\mu$ depending only on $\mu_i$ and $\eps$) and for every leaf $P$ of $\eta$, $P\cap D_t$ is connected and its angle with $\bar\zeta^i_t\cap \eta$ is bounded below by some constant $\bar\kappa$ depending only on $\eps$ and some lower bound on the curvatures of $S$ (this lower bound depending only on $\mu_i$ and the global geometry of $\Delta$, not on the simplex $\sigma$ under scrutiny). 

Now taking $\mu$ and $\kappa$ slightly smaller than $\bar\mu$ and $\bar\kappa$, the two properties of Claim \ref{a:Dt} are satisfied provided $\zeta_t^i$ is sufficiently $C^0$-close to $\bar\zeta^i_t$ (in terms of the geometric constants), which can be guaranteed by assuming $d$ to be small enough, since $\norm{\zeta^i_t}_1\le d c_{i,1}$. \end{proof}

In the following claim, we consider the (hard) case where $\sigma_d\cap F_d$ is nonempty, and thus $\alpha_d=\sigma_d\cap F_d$.

\begin{affirmation}\label{a:D't}
Taking $\mu$ and $d_{i+1/4}$ smaller if necessary, the orbit segments of $\zeta^i_t \cap \eta$ starting from $D_t$ and entirely contained in $N_{\mu_i}(\alpha)$ cover $N_{2 \mu}(\sigma) \cap N_{\mu}(\bar F)$, for every $t \in K_*$. 
\end{affirmation}

\begin{proof}
First observe that given $\tilde\mu \in [0,\mu_i]$, there exists $\mu$ (depending only on $\tilde\mu$ and the geometry of $\Delta$) such that 
$$N_{2 \mu}(\sigma) \cap N_{\mu}(\bar F) \subset N_{\tilde\mu}(\alpha).$$
So what we actually need to find is $\tilde\mu\in [0,\mu_i]$ such that $\zeta^i_t$ satisfies the following property (provided $d$ has been chosen small enough):\smallskip

\emph{The union of orbit segments of $\zeta^i_t \cap \eta$ starting from $D_t$ and entirely contained in $N_{\mu_i
}(\alpha)$ contains $N_{\tilde\mu }(\alpha)$.}\smallskip

As in the proof of Claim \ref{a:Dt}, denote by $\bar\zeta^i_t$ the constant plane field equal to $\zeta^i_t(q)$ on $V$. The union $U_t$ of the orbit segments of $\bar\zeta^i_t \cap \eta$ starting from $D_t$ and entirely contained in $N_{\mu_i}(\alpha)$ is simply the intersection of the orbits of $\bar\zeta^i_t \cap \eta$ starting from $D_t$ with $N_{\mu_i}(\alpha)$, since these orbits are straight lines and $N_{\mu_i}(\alpha)$ is convex. In particular, according to Claim \ref{a:Dt}, $U_t$ contains $N_{2\mu}(\sigma)\cap N_{\mu_i}(\alpha)$, which contains $N_{\bar\mu}(\alpha)$ provided $\bar\mu<\min(2\mu,\mu_i)$. 

Now taking $\tilde\mu$ slightly smaller than $\bar\mu$, $\zeta^i_t$ satisfies the above property in italics provided $\zeta_t^i$ is sufficiently $C^0$-close (in terms of the geometric constants) to $\bar\zeta^i_t$, which again can be guaranteed by assuming that $d$ is small enough.
\end{proof}

From now on we assume $d \le d_{i+1/4}$. Claim \ref{a:Dt} shows that the hypotheses of Lemma \ref{l:defmod} are satisfied by every plane field $\zeta^i_t$, $t \in K_*$, for the constants $\mu$ and $\kappa$. Setting $\mu_{i+1/4} = \mu$, we thus obtain a homotopy $\zeta^u$, $u \in [i,i+1/4]$, of $K_*$-plane fields, so that:
\begin{itemize}
\item $\zeta_t^u$ coincides with $\zeta_t^i$ outside of $N_{2 \mu_{i+1/4} }(\sigma)$; 
\item $\zeta_t^{i+1/4}$ is integrable on the $\mu_{i+1/4}$-neighbourhood of
$\sigma$ ;
\item for every $m\ge 1$, $\norm{\zeta^u_t}_m\le\chi_m(\norm{\zeta^i_t}_{m+1})$ for all $(t,u) \in K_* \times [i,i+1/4]$, for some universal polynomial $\chi_m$ with positive coefficients and no constant term. Now 
$$\norm{\zeta^i_t}_{m+1} =  \norm{\xi^i_t}_{d,m+1}\le c_{i,m+1}d$$ 
and there exists a constant $c_{i+1/4,m}\ge c_{i,m}$ such that $\chi_m(c_{i,m+1}x)\le c_{i+1/4,m}x$ for all $x\le d_{i+1/4}$ (cf. Remark \ref{r:poly}), so in the end $\norm{\zeta^u_t}_m\le c_{i+1/4,m} d$. 
\end{itemize}
\begin{remark}\label{r:constantes}
The constants $c_{i+1/4,m}$ depend on $\theta$, $\kappa$, $\mu$ and the geometry of $\Delta$, but not on $\sigma$ itself.
\end{remark}
Hence the rescaling $\xi^u$, $u \in [i,i+1/4]$, of $\zeta^u$ by a factor $d$ has the following properties: 
\begin{itemize}
\item $\xi_t^u$ coincides with $\xi_t^i$ outside of $N_{2 \mu_{i+1/4} d}(\sigma_d)$; 
\item $\xi_t^{i+1/4}$ is integrable on the $\mu_{i+1/4} d$-neighbourhood of
$\sigma_d$ ;
\item for every $m\ge 1$, there exists a constant $c_{i+1/4,m}$ such that $\Bars{\xi_t^u}_{d,m}\le c_{i+1/4,m}d$ for all $(t,u) \in K_* \times [0,i+1/4]$. 
\end{itemize}
Reducing $d_{i+1/4}$ if necessary so that 
\begin{equation}\label{e:di1}\tag{$\diamond_{i+1/4}$}
2 d_{i+1/4} c_{i+1/4,1} < \frac\beta{4},
\end{equation}
we can assume the angle variation of each $\xi^u_t$ on $V_d =
N_{d \delta/2}(\sigma)$ is less than $\beta/8$. Then for all $(t,u) \in
K_* \times [i,i+1/4]$ and all $p \in V_d$,
 \begin{equation}\label{e:ddi1}\tag{$\diamond\diamond_{i+1/4}$}
\angle(\xi^u_t(p), \xi^i_t(p)) < \frac\beta{4}.
\end{equation}
Indeed, if $q \in V_d \setminus  N_{2 \mu_{i+1/4} d}(\sigma)$, 
$$\angle(\xi^u_t(p), \xi^i_t(p)) \le \angle(\xi^u_t(p), \xi^u_t(q)) +
\angle(\xi^u_t(q), \xi^i_t(q)) +
\angle(\xi^i_t(q), \xi^i_t(p)) < \frac\beta{8} + 0 + \frac\beta{8}.$$
Inequalities \eqref{e:vhi} and \eqref{e:ddi1} imply, for all $(t,u) \in
K_* \times [0,i+1/4]$ and all $p \in U$, 
\begin{equation}\label{e:vhi1}\tag{$\ddagger_{i+1/4}$}
\angle( \xi_t^u(p), \xi_t(p) ) < \frac \theta {32} - \frac{3\beta}{4}.
\end{equation}

For $t \in \partial K_*$, Remark \ref{r:relatif} shows that the homotopy $\zeta^u_t$ (and thus $\xi^u_t$), $u \in[i,i+1/4]$, is completely stationary. Indeed, the intersection of the flow cylinder of $D_t$ with the domain bounded by $S$ (which contains the support of the homotopy) is an interval fiber bundle over $D_t$ on which $\zeta^i_t$ is assumed to be already integrable for every $t \in \partial K_* \subset K^{i-1}$.

Moreover, the same remark together with Claim \ref{a:D't} shows that for every $t \in K_*$, the homotopy $\xi^u_t$ is stationary on $N_{\mu_{i+1/4}d}(F_d)$. 

Since the neighbourhoods $V_d=N_{d\delta/2}(\sigma)$ of the different special simplices $\sigma$ are disjoint (by definition of $\delta$ and according to Claim \ref{a:speciaux}), we can apply Lemma \ref{l:defmod} to all of them simultaneously and we obtain constants $d_{i+1/4}$, $\mu_{i+1/4}$ and $c_{i+1/4,m}$ independent of $\sigma$ (cf. Remarks  \ref{r:const} and  \ref{r:constantes}).

\subsubsection*{Deformation near the other simplices}

\noindent\textbf{$0$-simplices.} Let $q$ be a vertex of $A_d$ belonging neither to $F_d$ nor to any special simplex, and $V_d$ its $d \delta/2$-neighbourhood. Note that $V_d$ is disjoint from the $d \delta/2$-neighbourhoods of the special simplices, so that $\xi^{i+1/4}_t$ coincides with $\xi^i_t$ on $V_d$ for every $t \in K_*$. Again, denote by $h_d$ the homothety of factor $d$ and center $q$, by $V$ the inverse image of  $V_d$, and by $\zeta_t^{i+1/4}$, $t \in K_*$, the pull-back $(h_d)^*\xi^{i+1/4}_t$ defined on $V$. For $S$, take the intersection of $V$ with a plane perpendicular to $\zeta^{i+1/4}_s(q)$ for some $s \in K_*$. Define the coordinate axes as follows:
\begin{itemize}
\item $\partial_y(q) \in \zeta^{i+1/4}_s(q)$ is orthogonal to $S$;
\item $\partial_x(q) \in T_qS$ is orthogonal to $\zeta^{i+1/4}_s(q)$.
\end{itemize}

\begin{figure}[htbp]
\centering
\includegraphics[width=6cm,height=5cm]{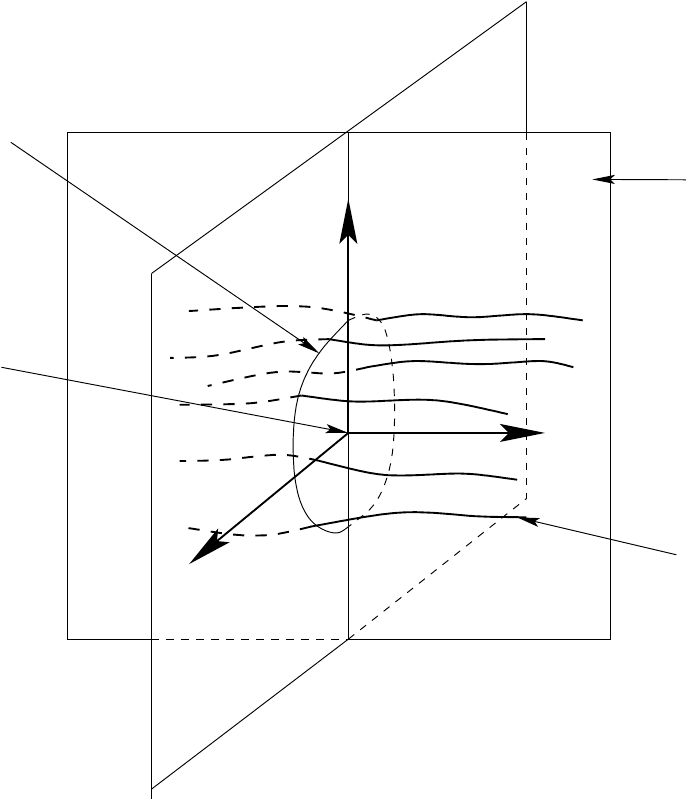}
\put(-180,77){$\scriptstyle q$}
\put(-178,117){$\scriptstyle S$}
\put(-127,33){$\scriptstyle \partial_x$}
\put(-33,60){$\scriptstyle \partial_y$}
\put(-82,105){$\scriptstyle \partial_z$}
\put(0,40){$\scriptstyle \zeta^{i+1/4}_t\cap\eta$}
\put(3,110){$\scriptstyle \zeta^{i+1/4}_s(q)$}
\caption{Choice of coordinates near a vertex $q$}
\label{fig:sommet}
\end{figure}

Combining \eqref{e:vp}, \eqref{e:di1} and \eqref{e:vhi1}, one can check that condition ($*$) (cf. Section \ref{ss:modele}) is satisfied by every $\zeta^{i+1/4}_t$, $t \in K_*$, for $\tilde\theta = \t/8$. With these notations, one easily proves an analogue of Claim \ref{a:Dt}, which provides numbers $d_{i+1/2}$, $\mu = \mu_{i+1/2}$, $\kappa = \pi/2 - \t/8$ and disks $D_t$ which can be taken to be independent of $t$ and contained in the $\mu_{i+1/4}$-neighbourhood of $q$. From now on we assume $d\le d_{i+1/2}$. Lemma \ref{l:defmod} then gives a homotopy $\zeta^u$, $u \in [i+1/4,i+1/2]$, of $K_*$-plane fields whose rescaling $\xi^u$ by a factor $d$ has the following properties: 
\begin{itemize}
\item $\xi_t^u$ coincides with $\xi_t^{i+1/4}$ outside of $N_{2 \mu_{i+1/2} d}(q)$;
\item $\xi_t^{i+1/2}$ is integrable on the $\mu_{i+1/2} d$-neighbourhood of $q$;
\item for every $m\ge 1$, there is a constant $c_{i+1/2,m}$ such that $\Bars{\xi_t^u}_{d,m}\le c_{i+1/2,m}d$ for all $(t,u) \in K_* \times [0,i+1/2]$. 
\end{itemize}
Reducing $d_{i+1/2}$ if necessary so that
\begin{equation}\label{e:di2}\tag{$\diamond_{i+1/2}$}
2 d_{i+1/2} c_{i+1/2,1} < \frac \beta{4},
\end{equation}
one can make sure (as for special simplices) that for all $(t,u) \in K_* \times[i+1/4,i+1/2]$ and all $p \in V$,
\begin{equation}\label{e:ddi2}\tag{$\diamond\diamond_{i+1/2}$}
\angle(\xi^u_t(p), \xi^{i+1/4}_t(p)) < \frac\beta{4}.
\end{equation}
Inequalities \eqref{e:vhi1} and \eqref{e:ddi2} imply that for all $(t,u)
\in K_*\times [0,i+1/2]$, 
\begin{equation}\label{e:vhi2}\tag{$\ddagger_{i+1/2}$}
	\angle( \xi_t^u(p), \xi_t(p) ) < \frac \theta {32}-\frac{\beta}{2}.
\end{equation}

For $t \in \partial K_*$, Remark \ref{r:relatif} shows once again that the homotopy $\zeta^u_t$ (and thus $\xi^u_t$), $u \in [i+1/4,i+1/2]$, is completely stationary, since $D_t$ is contained in $N_{\mu_{i+1/4}}(q)$. As for special simplices, we apply Lemma \ref{l:defmod} simultaneously near all vertices.\medskip

\begin{figure}[htbp]
\centering
\includegraphics[width=5cm,height=4cm]{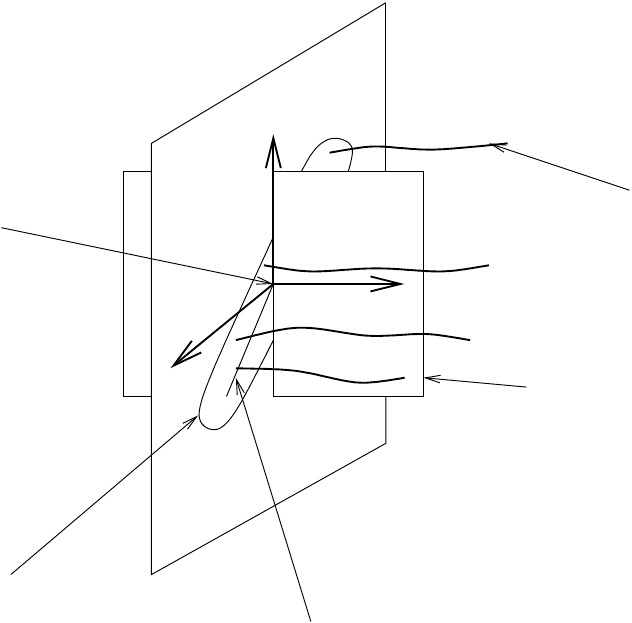}
\put(-150,5){$\scriptstyle S$}
\put(-70,-5){$\scriptstyle \alpha$}
\put(-150,72){$\scriptstyle q$}
\put(-105,53){$\scriptstyle \partial_x$}
\put(-55,55){$\scriptstyle \partial_y$}
\put(-84,92){$\scriptstyle \partial_z$}
\put(-17,40){$\scriptstyle \zeta^{i+1/2}_t(q)$}
\put(3,77){$\scriptstyle \zeta^{i+1/2}_t\cap\eta$}
\caption{Choice of coordinates near an edge $\alpha$}
\label{fig:arete}
\end{figure}

\noindent\textbf{$1$-simplices.} We now consider an edge $\alpha_d$ of $A_d$ contained neither in $F_d$ nor in any special simplex and we denote by $q$ its midpoint and by $V_d$ its $d \delta/2$-neighbourhood. Again, denote by $h_d$ the homothety of ratio $d$ and center $q$, by $V$ and $\alpha$ the inverse images of  $V_d$ and $\alpha_d$ under $h_d$, and by $\zeta_t^{i+1/2}$, $t \in K_*$, the pull-back $(h_d)^*\xi^{i+1/2}_t$ defined on $V$. For the surface $S$, we take the intersection of $V$ with a plane that contains $\alpha$ and is perpendicular to $\zeta^{i+1/2}_s(q)$ for some $s\in K_*$. The coordinate axes are defined as follows: the vector $\partial_y(q) \in \zeta^{i+1/2}_s(q)$ is orthogonal to $S$ and the vector
$\partial_x(q) \in T_qS$ is orthogonal to $\zeta^{i+1/2}_s(q)$.

According to \eqref{e:vp}, \eqref{e:di2} and \eqref{e:vhi2}, condition ($*$) is satisfied by every $\zeta^{i+1/2}_t$, $t \in K_*$, with $\tilde\theta = \t/8$. Here again, one can prove an analogue of Claim \ref{a:Dt}, which provides numbers $d_{i+3/4}$, $\mu = \mu_{i+3/4}$, $\kappa = \pi/2 - \t/8$ and disks $D_t$ which can be taken independent of
$t$ and contained in the $\mu_{i+1/2}$-neighbourhood of $\alpha$. From now on we assume $d\le d_{i+3/4}$.

Lemma \ref{l:defmod} then gives a homotopy $\zeta^u$, $u \in [i+1/2,i+3/4]$, of $K_*$-plane fields whose rescaling $\xi^u$ by a factor $d$ satisfies:
\begin{itemize}
\item $\xi_t^u$ coincides with $\xi_t^{i+1/2}$ outside of $N_{2 \mu_{i+3/4} d}(\alpha_d)$;
\item $\xi_t^{i+3/4}$ is integrable on the $\mu_{i+3/4} d$-neighbourhood of $\alpha_d$;
\item for every $m\ge 1$, there is a constant $c_{i+3/4,m}$ such that $\Bars{\xi_t^u}_{d,m}\le c_{i+3/4,m}d$ for all $(t,u) \in K_* \times [0,i+3/4]$, 
\end{itemize}
and analogues $\diamond_{i+3/4}$, $\diamond\diamond_{i+3/4}$ and $\ddagger_{i+3/4}$ of $\diamond_{i+1/2}$, $\diamond\diamond_{i+1/2}$ and $\ddagger_{i+1/2}$ respectively. In particular, for every $t \in \partial K_*$, Remark \ref{r:relatif} shows once again that the homotopy $\xi^u_t$, $u \in [i+1/2,i+3/4]$, is completely stationary. Moreover, for every $t \in K_*$, every integral curve of $\zeta^{i+1/2}_t \cap \eta$ intersecting $D_t$ meets $N_{\mu_{i+1/2}}(\partial \alpha)$ along an interval, for $N_{\mu_{i+1/2}}(\partial \alpha)$ is made of two strictly convex balls. It then follows from Remark \ref{r:relatif} that $\xi^u_t = \xi^{i+1/2}_t$ on $N_{\mu_{i+1/2} d}(\partial \alpha_d)$ for all $u \in [i+1/2,i+3/4]$. In other words, the deformation changes nothing in the $\mu_{i+1/2} d$-neighbourhood of the $0$-skeleton. One can thus perform the deformations simultaneously near all edges.\medskip

\noindent\textbf{(non-special) $2$-simplices.} Finally, let $\sigma_d$ be a non-special face of $A_d$ not contained in $F_d$, $V_d$ its $d \delta/2$-neighbourhood and $q$ its center. Again, denote by $h_d$ the homothety of factor $d$ and center $q$, by $V$ and $\sigma$ the inverse images of  $V_d$ and $\sigma_d$ under $h_d$, and by $\zeta_t^{i+3/4}$, $t \in K_*$, the pull-back $(h_d)^*\xi^{i+3/4}_t$ defined on $V$. For the surface $S$, take the intersection of
$V$ with the plane containing $\sigma$.  We fix the coordinate axes as follows:
\begin{itemize}
\item $\partial_y(q)$ belongs to $\zeta^{i+3/4}_s(q)$ for some $s \in
K_*$ and has maximal angle with $\sigma$ ;
\item $\partial_x(q)$ is orthogonal to $\zeta^{i+3/4}_s(q)$.
\end{itemize}
Here again, every $\zeta^{i+3/4}_t$, $t \in K_*$, satisfies condition ($*$)
for $\tilde\t = \t/8$. 

\begin{figure}[htbp]
\centering
\includegraphics[width=5cm,height=5cm]{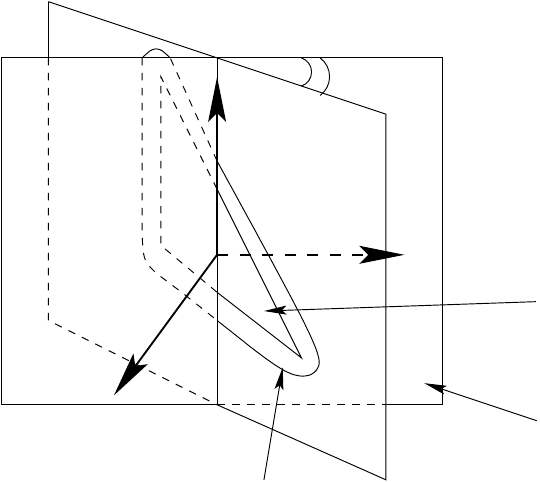}
\put(-75,-8){$\scriptstyle S$}
\put(-92,67){$\scriptstyle q$}
\put(2,51){$\scriptstyle \sigma$}
\put(-122,24){$\scriptstyle \partial_x$}
\put(-35,64){$\scriptstyle \partial_y$}
\put(-80,110){$\scriptstyle \partial_z$}
\put(2,15){$\scriptstyle \zeta_t^{i+3/4}(q)$}
\put(-51,116){$\scriptstyle \ge \t/4$}
\caption{Choice of coordinates near a face $\sigma$}
\label{fig:face}
\end{figure}

Moreover, \emph{since $\sigma$ is non-special}, $\angle(\zeta^i_t(p), \sigma) \ge \t/2$ for all $(t,p) \in K_* \times \sigma$. This, together with inequalities ($\diamond\diamond_{i+k/4}$), $1 \le k \le 3$, and ($\diamond_{i+3/4}$) (left as an exercise to the reader) implies that $\angle(\zeta_t^{i+3/4}(p), \sigma) \ge \t/2 - \beta \ge \t/4$ for all $(t,p) \in K_* \times V$. This lower bound allows us to prove an analogue of Claim \ref{a:Dt}, which provides constants $d_{i+1}$, $\mu = \mu_{i+1}$, $\kappa = \t/4$ and disks $D_t$ which can be taken independent of $t$ and contained in the $\mu_{i+3/4}$-neighbourhood of $\sigma$. We may assume $d\le d_{i+1}$. Lemma \ref{l:defmod} then gives a homotopy $\zeta^u$, $u \in  [i+3/4,i+1]$, of $K_*$-plane fields whose rescaling $\xi^u$ by a factor $d$ has properties analogous to the ones we listed for $u\in[0,3/4]$. Since every integral curve of $\zeta^{i+3/4}_t \cap \eta$ which intersects $D_t$ meets the $\mu_{i+3/4}$-neighbourhood of each edge of $\sigma$ along an interval, the deformation does not affect $N_{\mu_{i+3/4}}(\partial \sigma)$. One can thus once again (and for the last time) make the modifications simultaneously on all faces. \medskip

Carrying out this construction for every $i$-simplex $K_*$ of $K^i$, we finally obtain a homotopy $\xi^u$, $u \in [0,i+1]$, of $K^i$-plane fields on $U$ with all the properties needed to conclude step $i+1$ of the induction.


\begin{thebibliography}{MM}

\bibitem{B-E}
C.~\textsc{Bonatti} and H.~\textsc{Eynard-Bontemps} ---
\textit{Connectedness of the space of smooth actions of $\Z^n$ on the interval }. Ergodic Theory and Dynamical Systems (2015), available on CJO2015, doi: 10.1017/etds.2015.3.

\bibitem{B-F}
C.~\textsc{Bonatti} and S.~\textsc{Firmo} ---
\textit{Feuilles compactes d'un feuilletage g\'en\'erique en codimension 1}. (French. English, French summary) [Compact leaves of a generic foliation in codimension 1]
Ann. Sci. \'Ecole Norm. Sup. (4) \textbf{27} (1994), no. 4, 407--462. 

\bibitem{B-H}
C.~\textsc{Bonatti} and A.~\textsc{Haefliger} ---
\textit{D\'eformations de feuilletages}. (French) [Deformations of foliations]
Topology \textbf{29} (1990), no. 2, 205--229. 

\bibitem{Bo}
J.~\textsc{Bowden} ---
\textit{Contact structures, deformations and taut foliations}. To appear in Geom. Topol. Preprint 
 arXiv:1304.3833. 

\bibitem{C-C1}
A.~\textsc{Candel} and L.~\textsc{Conlon} ---
\textit{Foliations I}. \
Grad. Studies in Math. \textbf{23} (2000).

\bibitem{C-C2}
J.~\textsc{Cantwell} and L.~\textsc{Conlon} ---
\textit{Leaves with isolated ends in foliated 3-manifolds}. \
Topology  \textbf{16} (1977), no. 4, 311--322.

\bibitem{Ce}
J.~\textsc{Cerf} --- \textit{Sur les diff\'eomorphismes de la sph\`ere de dimension
trois $(\Gamma_4 = 0)$}. \ Springer Lecture Notes in Mathematics \textbf{53},
Springer-Verlag, Berlin (1968).

\bibitem{El}
Y.~\textsc{Eliashberg} ---
\textit{Classification of overtwisted contact structures on $3$-manifolds}. \
Invent. Math. \textbf{98} (1989), 623--637.

\bibitem{Ey}
H.~\textsc{Eynard-Bontemps} ---
\textit{Sur deux questions de connexit\'e concernant les feuilletages et leurs holonomies}. \
Th\`ese ENS Lyon 2009. http://tel.archives-ouvertes.fr/tel-00436304/fr/

\bibitem{Ge}
H.~\textsc{Geiges} ---
\textit{An Introduction to Contact Topology}. \
Cambridge studies in advanced mathematics \textbf{109}. Cambridge Univ. Press, Cambridge (2008).

\bibitem{Go}
S.~\textsc{Goodman} ---
\textit{Closed leaves in foliations of codimension one}. \
Comm. Math. Helv. \textbf{50} (1975), 383--388.

\bibitem{He1}
M.~R.~\textsc{Herman} ---
\textit{Simplicit\'e du groupe des diff\'eomorphismes de classe $\Cinf$, isotopes \`a l'identit\'e, du tore de dimension $n$.} (French)
C. R. Acad. Sci. Paris S\'er. A-B \textbf{273} 1971 A232--A234. \

\bibitem{He}
M.~R.~\textsc{Herman} ---
\textit{Sur la conjugaison diff\'erentiable des diff\'eomorphismes du cercle \`a des
rotations}. \
Inst. Hautes \'Etudes Sci. Publ. Math. \textbf{49} (1979), 5--233.

\bibitem{Ko} 
N.~\textsc{Kopell} ---
\textit{Commuting diffeomorphisms}. \ 
In \textit{Global Analysis}, 
Proc. Sympos. Pure Math. XIV, Amer. Math. Soc. (1968), 165--184.

\bibitem{La}
A.~\textsc{Larcanch\'e} ---
\textit{Topologie locale des espaces de feuilletages en surfaces des vari\'et\'es
ferm\'ees de dimension 3}. \
Comm. Math. Helv. \textbf{82} (2007), 385--411.

\bibitem{Li}
W. B. R.~\textsc{Lickorish} ---
\textit{A foliation for $3$-manifolds}. \
Ann. of Math. (2)  \textbf{82} (1965), 414--420.

\bibitem{No}
S. P. \textsc{Novikov} ---
\textit{Topology of foliations}.\
Trans. Moscow Math. Soc. \textbf{14} (1965), 248--278 (Russian), A.M.S
Translation (1967), 268--304.

\bibitem{Re}
G.~\textsc{Reeb} ---
\textit{Sur certaines propri\'et\'es topologiques des vari\'et\'es feuillet\'ees}. \   
Actual. Sci. Ind. \textbf{1183}, Hermann, Paris (1952).

\bibitem{Sc}
P.~\textsc{Schweitzer} ---
\textit{Codimension one foliations without compact leaves}. \   
Comm. Math. Helv. \textbf{70} (1995), 171 -- 209.

\bibitem{Sz} 
G.~\textsc{Szekeres} ---
\textit{Regular iteration of real and complex functions}.  \ 
Acta Math. \textbf{100} (1958), 203--258.

\bibitem{Ta}
F.~\textsc{Takens} ---
\textit{Normal forms for certain singularities of vector fields}. \ 
Ann. Inst. Fourier \textbf{23} (1973), 163--195.

\bibitem{Th1}
W.~P.~\textsc{Thurston} ---
\textit{Foliations of Three-Manifolds which are Circle Bundles}. \
Ph. D. Dissertation, University of California, Berkeley (1967).

\bibitem{Th2}
W.~P.~\textsc{Thurston} ---
\textit{A local construction of foliations for three-manifolds}. \
Differential topology (Proc. Sympos. Pure Math. \textbf{27}, Stanford Univ.,
California, 1973), Amer. Math. Soc. (1975), 315--319.

\bibitem{Th3}
W.~P.~\textsc{Thurston} ---
\textit{The theory of foliations of codimension greater than one}. \
Comm. Math. Helv.  \textbf{49}  (1974), 214--231.

\bibitem{Th4}
W.~P.~\textsc{Thurston} ---
\textit{Existence of codimension-one foliations}. \
Ann. of Math. (2)  \textbf{104}  (1976), no. 2, 249--268. 

\bibitem{Vo}
T.~\textsc{Vogel} ---
\textit{Uniqueness of the contact structure approximating a foliation}. \
Preprint  arXiv:1302.5672

\bibitem{Wo}
J.~\textsc{Wood} ---
\textit{Foliations on 3-manifolds}. \
 Ann. of Math. \textbf{89}  (1969), 336--358.
\end{thebibliography}
\end{document}